\documentclass[11pt,namelimits,sumlimits,a4paper]{amsart}
\usepackage{cite}
\usepackage{comment}

\usepackage{amssymb,amsmath}
\usepackage[mathscr]{eucal}
\usepackage{slashed}

\usepackage[latin1]{inputenc}
\usepackage[T1]{fontenc}
\usepackage[UKenglish]{babel}
\usepackage{amsfonts}
\usepackage{fancyhdr}
\usepackage{graphicx}
\usepackage{amsmath}
\usepackage{amsthm}
\usepackage{amsmath,amscd}
\usepackage{latexsym}
\usepackage{cite}
\usepackage{amssymb,amsmath}

\usepackage{comment}
\usepackage[mathscr]{eucal}
\usepackage{slashed}
\usepackage{times,psfrag}
\usepackage{mathtools}
\usepackage{hyperref}
\usepackage{multirow}
\usepackage{longtable}
\usepackage{bm}
\usepackage[all]{xypic}
\usepackage{fancyhdr}
\usepackage{float}
\usepackage{esint}
\usepackage{layout}
\usepackage{enumitem}
\usepackage{lmodern}
\usepackage[rgb]{xcolor}

\numberwithin{equation}{section}
\newtheorem{theorem}[equation]{Theorem}
\newtheorem{lemma}[equation]{Lemma}
\newtheorem{prop}[equation]{Proposition}
\newtheorem{corollary}[equation]{Corollary}

\theoremstyle{definition}
\newtheorem{definition}[equation]{Definition}
\newtheorem{example}[equation]{Example}

\theoremstyle{remark}
\newtheorem{remark}[equation]{Remark}
\newtheorem*{remark*}{Remark}

\newtheorem*{assumption}{Assumption}

\newcommand{\ie}{\emph{i.e.} }
\newcommand{\eg}{\emph{e.g.} }
\newcommand{\cf}{\emph{cf.} }

\newcommand{\vspan}{\operatorname{span}}

\newcommand{\beq}{\begin{equation}}
\newcommand{\eeq}{\end{equation}}
\newcommand{\bea}{\begin{eqnarray}}
\newcommand{\eea}{\end{eqnarray}}

\newcommand{\C}{\mathbb{C}}

\newcommand{\R}{\mathbb{R}}

\newcommand{\Z}{\mathbb{Z}}
\newcommand{\N}{\mathbb{N}}

\newcommand{\HH}{\mathbb{H}}
\newcommand{\CP}{\mathbb{CP}}

\newcommand{\ra}{\rightarrow}

\newcommand{\vol}{\operatorname{Vol}}
\newcommand{\dvol}{\operatorname{dv}}
\newcommand{\Real}{\operatorname{Re}}
\newcommand{\Imag}{\operatorname{Im}}

\newcommand{\Ric}{\operatorname{Ric}}

\newcommand{\Lie}[1]{\mathfrak{#1}}

\newcommand{\tu}[1]{\textup{#1}}

\newcommand{\gtwo}{\ensuremath{\textup{G}_2}}

\newcommand{\gtstr}{\gtwo--structure}

\newcommand{\gtmfd}{\gtwo--manifold}
\newcommand{\gtmetric}{\gtwo--metric}
\newcommand{\gthol}{\gtwo--holonomy\ }

\newcommand{\unitary}[1]{\textup{U$(#1)$}}

\newcommand{\sunitary}[1]{\textup{SU$(#1)$}}
\newcommand{\sunitaryn}{\textup{SU$(n)$}}

\newcommand{\suthreestr}{\sunitary{3}--structure}
\newcommand{\sutwostr}{\sunitary{2}--structure}
\newcommand{\sorth}[1]{\textup{SO$(#1)$}}

\newcommand{\spin}[1]{\textup{Spin$(#1)$}}

\def\co{\colon\thinspace}

\begin{document}

\title[Complete non-compact \gtwo--manifolds from AC Calabi--Yau $3$-folds]{Complete non-compact \gtwo--manifolds from asymptotically conical Calabi--Yau $3$-folds}
\author[L.~Foscolo]{Lorenzo~Foscolo}
\author[M.~Haskins]{Mark~Haskins}
\author[J.~Nordstr\"om]{Johannes~Nordstr\"om}

\maketitle

\begin{abstract}
We develop a powerful new analytic method to construct complete non-compact Ricci-flat $7$-manifolds, more specifically 
\gtmfd s, \ie Riemannian $7$-manifolds $(M,g)$ whose holonomy group is the compact exceptional Lie group
\gtwo. Our construction gives the first general analytic construction of complete non-compact Ricci-flat metrics in any odd dimension
and establishes a link with the  Cheeger--Fukaya--Gromov theory of collapse with bounded curvature.

The construction starts with a complete non-compact asymptotically conical 
Calabi--Yau $3$-fold $B$ and a circle bundle $M\ra B$ satisfying a necessary topological condition. 
Our method then produces a $1$-parameter family of circle-invariant complete 
\gtmetric s $g_\epsilon$ on $M$ 
that collapses with bounded curvature as $\epsilon \to 0$ to the original Calabi--Yau metric on the base $B$.
The \gtmetric s we construct have controlled asymptotic geometry at infinity, so-called 
asymptotically locally conical (ALC) metrics;
these are the natural higher-dimensional analogues of the ALF metrics that are well known 
in $4$-dimensional hyperk\"ahler geometry. 

We give two illustrations of the strength of our method. Firstly we use it to construct infinitely many diffeomorphism types 
of complete non-compact simply connected \gtmfd s; previously only a handful of such diffeomorphism types was known. 
Secondly we use it to prove the existence of 
continuous families of complete non-compact \gtmetric s of arbitrarily high dimension; 
previously only rigid or $1$-parameter families of complete non-compact \gtmetric s were known.
\end{abstract}

\section{Introduction}
Despite the centrality of Ricci curvature in modern Riemannian geometry, 
constructing Ricci-flat metrics remains extremely challenging. At present the only two tools available are holonomy reduction 
methods or symmetry reduction (in the non-compact setting). 
In even dimensions, holonomy reduction techniques related to K\"ahler geometry have proven very powerful for constructing both 
compact and complete non-compact examples. 
The only possible irreducible holonomy reduction for odd-dimensional manifolds is that of \gtwo~holonomy for $7$-manifolds. 
While in recent years the number of known families of compact \gtmfd s has grown quite considerably, there have remained very few known families of  complete non-compact irreducible \gtmfd s.

In this paper we develop a new analytic method to construct complete non-compact \gtmfd s, 
that is Riemannian $7$-manifolds $(M,g)$ whose holonomy group is the compact exceptional Lie group
\gtwo. The manifolds we construct have a controlled asymptotic geometry at infinity, so-called \emph{asymptotically locally conical} (ALC) geometry, which is the natural higher-dimensional analogue of the asymptotic geometry of $4$-dimensional ALF hyperk\"ahler manifolds. 
As an illustration of the strength of the method developed here we use it to construct infinitely many
diffeomorphism types of complete non-compact simply connected \gtmfd s and to prove the existence of continuous families of complete non-compact \gtmetric s of arbitrarily high dimension.
Prior to our work only a handful of diffeomorphism types of complete non-compact \gtmfd s (or more generally irreducible Ricci-flat $7$-manifolds) were known.

\subsubsection*{Einstein and Ricci-flat metrics with circle symmetry}
In the case of Einstein manifolds with positive scalar curvature, looking for circle-invariant Einstein metrics on the total space of 
circle bundles or more generally torus-invariant metrics on torus bundles \cite{Friedrich:Kath,Wang:Ziller} over compact K\"ahler--Einstein Fano manifolds (or more generally orbifolds) has proven a quite powerful way 
to generate new Einstein metrics from existing ones. While the absence of Killing fields
rules out such bundle constructions in the compact irreducible Ricci-flat case, it is still natural to consider such circle-invariant 
Ricci-flat metrics in the complete non-compact setting---particularly in light of the great utility of the Gibbons--Hawking ansatz as a way
to produce interesting complete (and incomplete) circle-invariant $4$-dimensional hyperk\"ahler~metrics. One concrete idea, given further motivation below, is to look for families of complete circle-invariant Ricci-flat metrics $g_\epsilon$ with submaximal volume growth on the total space $M$ of a circle bundle
over an asymptotically conical Ricci-flat manifold $(B,g_B)$ that collapse in the Gromov--Hausdorff sense as $\epsilon \to 0$ to $(B,g_B)$. 
To make this idea into a powerful generally applicable tool for constructing large families of interesting new 
complete Ricci-flat metrics a number of ingredients are necessary: (i) tools to construct enough interesting asymptotically conical 
Ricci-flat metrics for use as the base metric; (ii) a good understanding of the deformation theory of such asymptotically conical 
Ricci-flat metrics and (iii) a well-behaved perturbation theory to allow for the correction by PDE methods of almost Ricci-flat metrics on $M$ to genuine Ricci-flat metrics. In the general Ricci-flat setting all three issues cause serious difficulty.

\subsubsection*{\gthol and Calabi--Yau geometry in $6$ dimensions}
However, in the special case where we try to construct \gthol metrics on circle bundles over asymptotically conical (AC) Calabi--Yau $3$-folds 
we have additional tools available that allow us to overcome all three issues and therefore develop a generally applicable analytic construction 
method: see Theorem~\ref{thm:Main:Theorem:technical} below for a precise statement.  
Perhaps the most surprising aspect of our main theorem is not that such a construction works in some cases, but that the tools now available are sufficiently powerful 
that it works in complete generality, independent of the particular AC Calabi--Yau $3$-fold or circle bundle in question. Moreover, 
to demonstrate that the existence of very many complete non-compact \gtmfd s follows from our construction we will rely on many of the latest developments 
in the metric aspects of complete non-compact Calabi--Yau manifolds; specifically this includes the recent extension of the equivalence between $K$-stability 
and the existence of K\"ahler--Einstein metrics on Fano manifolds \cite{CDS} to the setting of Sasaki--Einstein metrics/Calabi--Yau cone metrics \cite{Collins:Szekelyhidi:JDG,Collins:Szekelyhidi:GT}.
Our work establishes a new intimate connection between metric Calabi--Yau geometry in 6 real dimensions and  irreducible \gthol metrics in 7 dimensions; 
it gives the first systematic understanding why certain non-compact complete Calabi--Yau metrics in $6$ dimensions 
lead to closely-related non-compact complete irreducible \gthol metrics in $7$ dimensions.

Another point to stress is that in the context of the Gibbons--Hawking ansatz in 
hyperk\"ahler geometry in $4$ dimensions, one must allow circle actions with fixed points 
(corresponding to poles of the positive harmonic function used) to generate non-trivial complete examples. 
Instead in 7 dimensions by considering only free circle actions we can still produce a bountiful supply 
of complete non-compact \gtmfd s. This is due to the fact that---unlike $\R^3$---many AC Calabi--Yau $3$-folds have non-trivial second cohomology and therefore support non-trivial circle bundles.

\subsubsection*{Cohomogeneity one examples: AC and ALC geometry}
We now discuss some further motivation for our basic approach and the particular kinds of asymptotic geometry that we choose to consider. 
This motivation comes from the geometry of the classical highly symmetric complete non-compact \gtmfd s constructed by Bryant and Salamon
in 1989, and more recent deformations thereof considered in the M theory community. 
Recall that Bryant and Salamon constructed the first complete non-compact \gtmfd s in\cite{Bryant:Salamon}. 
 There are three Bryant--Salamon \gtmetric s up to scaling, all admitting a cohomogeneity one action, that is, a compact Lie group acts isometrically on $(M,g)$
with generic orbit of codimension 1. The large symmetry group affords a reduction of the system of nonlinear 
partial differential equations for 
a torsion-free \gtstr~to a family of nonlinear ordinary differential equations. 
The geometry at infinity of the three Bryant--Salamon examples is \emph{asymptotically conical}, that is, 
outside a compact subset the non-compact $7$-manifold $M$ is diffeomorphic to a cone $C(S)$ over a smooth 
Riemannian $6$-manifold $(S,g_S)$ and the metric $g$ on $M$ becomes asymptotic 
to the conical metric $g_C=dr^2+r^2 g_S$ on $C(S)$. 

In the early 2000s, because of the importance of \gtmfd s 
in supersymmetric compactifications in M theory, 
several different groups of theoretical physicists revisited the cohomogeneity one approach \cite{BGGG,Brandhuber,CGLP:C7:tilde,CGLP:M:Conifolds,Hori:A7}: 
they wrote down ODE systems that govern more general cohomogeneity one \gtmfd s, and thereby
discovered a new explicit non-compact ALC \gtmfd  \cite{BGGG}. 
(Cohomogeneity one ALC $8$-manifolds with exceptional holonomy group $\textup{Spin}_7$ were found a little earlier in \cite{CGLP:ALC:Spin7}).
By studying numerical solutions to these ODE systems they also 
gave strong numerical evidence for the existence of four different $1$-parameter families of 
complete non-compact \gtmetric s, denoted by the physicists as $\mathbb{A}_7$, $\mathbb{B}_7$, $\mathbb{C}_7$ and $\mathbb{D}_7$. 
The existence of the $\mathbb{B}_7$ family has since been established rigorously  \cite{Bogoyavlenskaya}. 
A rigorous ODE-based proof of the existence of the other three conjectured $1$-parameter families of solutions has not yet been given. (Since the first version of this paper appeared, the authors \cite{FHN:Coho1:ALC} have used ODE-based methods to establish the existence of the $\mathbb{C}_7$ and $\mathbb{D}_7$ families, including far from the collapsed limit considered here.)

The geometry at infinity (of the generic member of each) of these $1$-parameter families of complete non-compact \gtmfd s 
has the following common feature: the complement of 
a compact subset of the non-compact $7$-manifold $M$ is diffeomorphic to the total space 
of a principal circle bundle over a $6$-dimensional Riemannian cone $(C(\Sigma),g_C)$, 
and the metric $g$ on $M$ approaches $g_C + \theta_\infty^2$ for some connection $\theta_\infty$ 
on this circle bundle. 
Physicists termed this asymptotic geometry
\emph{asymptotically local conical} (ALC), thinking of it as a natural higher-dimensional generalisation 
of the asymptotically locally flat (ALF) geometry of the Taub--NUT metric. In mathematics,  the asymptotic geometry of ALC manifolds is a 
special case of the \emph{fibred boundary metrics} introduced by Mazzeo--Melrose \cite{Mazzeo:Melrose:fibred}.

\subsubsection*{Degenerations of ALC metrics and the motivation for our approach}
One interesting common feature of the various $1$-parameter families of  cohomogeneity one 
ALC \gtmetric s is that within each family two different  non-ALC asymptotic geometries 
arise as limits. In one limit the ALC geometry at infinity transitions 
to  asymptotically conical (AC) geometry and one of the original Bryant--Salamon AC \gthol metrics is recovered.
This is analogous to the way that ALE gravitational instantons can
appear as limits of ALF gravitational instantons, \eg in the Gibbons--Hawking construction of multi--Taub--NUT and multi--Eguchi--Hanson spaces.

However, the motivation for the approach taken in the present paper is the other non-ALC limit.
In this other limit, the $1$-parameter family of $7$-dimensional ALC \gthol metrics $g_\epsilon$ 
collapses as $\epsilon \to 0$ to a $6$-dimensional AC metric $g_0$ on a Calabi--Yau manifold $B$.
 
Because all the ALC \gtmetric s $g_\epsilon$ discovered by physicists are of cohomogeneity one, 
so are the associated 6-dimensional collapsed limits $g_0$. 
By considering the different known cohomogeneity one AC Calabi--Yau 3-folds---the small resolution of the conifold, the smoothing of the conifold (and its quotient by the standard anti-holomorphic involution) 
and the Calabi metric on $K_{\CP^1\times\CP^1}$---one gains an important insight into the origin of the four known/conjectured 
$1$-parameter families of ALC \gtmetric s; each of these $1$-parameter families collapses to a different AC Calabi--Yau 3-fold. 

\subsubsection*{Highly collapsed \gtmetric s from Calabi--Yau $3$-folds: an analytic approach}
The previous discussion naturally suggests we try to reverse the above procedure: we wish to start from a given 
AC metric on a Calabi--Yau $3$-fold $B$ and construct a family of highly collapsed ALC \gtmetric s 
on a suitable non-compact $7$-manifold $M$, built from $B$ and some further auxiliary data, that collapses 
back to the given AC Calabi--Yau metric. 
If successful, such an approach is potentially very powerful because 
a large number of AC Calabi--Yau 3-folds have now been constructed by PDE methods, that is, by proving existence of 
solutions with controlled asymptotics to a complex Monge--Amp\`ere equation on suitable non-compact complex $3$-folds. 

In most cases such AC Calabi--Yau 3-folds will not have many (or indeed any) continuous symmetries 
and therefore any resulting ALC \gtmfd s also need not have a high degree of symmetry, 
in strong contrast to the cohomogeneity one solutions explored by physicists. 
This necessitates adopting a PDE-based rather than ODE-based approach to the problem.
In this paper we develop such a general PDE-based method. Our main result is the following general analytic existence theorem, 
whose statement also includes a more precise description of what is meant by ALC geometry.

\begin{theorem}\label{thm:Main:Theorem:technical}
Let $(B,g_0,\omega_0,\Omega_0)$ be an AC Calabi--Yau $3$--fold asymptotic with rate $\mu<0$ to the Calabi--Yau cone $\left( \tu{C}(\Sigma),g_\tu{C},\omega_\tu{C},\Omega_\tu{C}\right)$ over a smooth Sasaki--Einstein $5$--manifold $\Sigma$. Let $M\ra B$ be a principal $U(1)$--bundle such that $c_1 (M)\neq 0$ but $c_1(M)\cup [\omega_0] =0\in H^4(B)$. 

Then there exists $\epsilon_0>0$ such that for every $\epsilon\in (0,\epsilon_0)$ the $7$--manifold $M$ carries an $S^1$--invariant torsion-free \gtstr~$\varphi_\epsilon$ with the following properties. Let $g_\epsilon$ denote the Riemannian metric on $M$ induced by $\varphi_\epsilon$.
\begin{enumerate}
\item $g_\epsilon$ has restricted holonomy $\tu{Hol}^0(g_\epsilon)=G_2$.
\item $(M,g_\epsilon)$ is an ALC manifold: outside of a compact set $K$, $M$ is identified with the total space of a principal $U(1)$--bundle over an exterior region $\{ r> R\}$ in the cone $\tu{C}(\Sigma)$. Under this identification
\[
g_\epsilon = g_\tu{C} + \epsilon^2 \theta_\infty + O\left( r^{-\min{\{1,-\mu\}}}\right)
\]
with analogous decay for all covariant derivatives. Here $\theta_\infty$ is a connection on the principal circle bundle $M\setminus K \ra \{ r\geq R\}\subset \tu{C}(\Sigma)$.
\item There exists a connection $\theta$ on $M\ra B$ such that the difference between $g_\epsilon$ and the Riemannian submersion with fibres of constant length $2\pi\epsilon$
\[
g_0 + \epsilon^2 \theta^2
\]
converges to zero as $\epsilon \ra 0$ in $C^{k,\alpha}$ for every $k\geq 0$ and $\alpha\in (0,1)$. In particular, $(M,g_\epsilon)$ collapses with bounded curvature to $(B,g_0)$ as $\epsilon\ra 0$.
\end{enumerate}
\end{theorem}
%
%
\subsubsection*{Consequences of the main theorem}
In combination with the powerful methods now available (see Section~\ref{sec:examples} for details and references to the recent literature) 
to construct many interesting AC Calabi--Yau $3$-folds, 
Theorem \ref{thm:Main:Theorem:technical} leads to the construction of a plethora of new complete non-compact \gtmetric s 
with ALC geometry. 
A small sample of the many new complete ALC \gtmetric s that arise from our method is described 
in Section \ref{sec:examples}.  We defer a more systematic study of the variety of possible ALC \gtmetric s that 
arise from our construction to elsewhere, concentrating here instead on presenting the analytic details of 
our general construction.
More specifically in Section \ref{sec:examples} we focus on AC Calabi--Yau metrics on \emph{small resolutions} of Calabi--Yau cones: in this case $H^4(B)=0$ and hence the condition $c_1(M)\cup [\omega_0] =0\in H^4(B)$ is trivially satisfied. Until recently only one example of a $3$-dimensional Calabi--Yau cone admitting a small resolution was known to exist. Recent results about existence of Sasaki--Einstein metrics and K-stability \cite{Collins:Szekelyhidi:GT}, however, yield the existence of an infinite family of $3$-dimensional Calabi--Yau cones with small resolutions. In Corollary \ref{cor:cAp:ALC:G2} we exploit this infinite family to produce infinitely many diffeomorphism types of simply connected ALC \gtmfd s and families of ALC \gtmetric s of arbitrarily high dimension.

\subsubsection*{The adiabatic limit of circle-invariant torsion-free \gtstr s.}
The central objects of study in this paper are therefore circle-invariant torsion-free \gtstr s $\varphi$ 
on a principal circle bundle $M^7$ over a non-compact $6$-manifold $B^6$, as
first studied by Apostolov--Salamon in \cite{Apostolov:Salamon}.
The complexity of the Apostolov--Salamon equations is such that at present little can be said about its solutions in any generality. Our strategy will therefore be to study these equations in the natural adiabatic limit where the circle fibres shrink to zero length. More specifically, we will be interested in a family $\varphi_\epsilon$ of circle-invariant torsion-free \gtstr s 
on the total space $M$ with circle fibres shrinking to zero length as $\epsilon \to 0$. 
Such a $1$-parameter family of collapsing circle-invariant \gtstr s $\varphi_\epsilon$ can be written as 
\[
\varphi_\epsilon = \epsilon\,\theta_\epsilon\wedge\omega_\epsilon + (h_\epsilon)^\frac{3}{4}\Real\Omega_\epsilon,
\]
where for each $\epsilon>0$, $(\omega_\epsilon, \Omega_\epsilon)$ is an {\suthreestr}, 
$h_\epsilon$ is a positive function on $B$ and $\theta_\epsilon$ is a connection $1$-form on the circle bundle $M \to B$.

The condition that $\varphi_\epsilon$ be torsion-free, \ie the closure and coclosure of $\varphi_\epsilon$, 
gives rise to a complicated $\epsilon$-dependent nonlinear system of PDEs for the quadruple 
$(\omega_\epsilon, \Omega_\epsilon, h_\epsilon, \theta_\epsilon)$, that following \cite{Apostolov:Salamon} we call the (rescaled) Apostolov--Salamon equations.
These equations can be viewed as coupled equations for the torsion of the {\suthreestr} 
$(\omega_\epsilon, \Omega_\epsilon)$ together with a gauge-theoretic equation for the pair $(h_\epsilon, \theta_\epsilon)$. 
In the formal limit in which $\epsilon \to 0$ the rescaled Apostolov--Salamon equations simplify considerably: 
one finds that the function $h_0$ should be a constant (that we can assume to be $1$) and 
that $(\omega_0,\Omega_0)$ should be a Calabi--Yau structure on $B$.

We are now seeking to reverse this collapsing process: we assume given the AC Calabi--Yau structure $(\omega_0,\Omega_0)$ on $B$ and we aim to reconstruct the $1$-parameter family of torsion-free \gtstr s $\varphi_\epsilon$ on $M$ for non-zero but small $\epsilon$. The problem can be subdivided into two steps: first we need to construct an approximate solution and then correct it to a torsion-free \gtstr. Both steps involve non-trivial linear analysis on $B$ that we now describe in more detail.  

\subsubsection*{Construction of an approximate solution}
In our context, a good approximate solution is a closed $S^1$-invariant \gtstr~on $M \to B$ with small torsion.  
There is no a priori or obvious choice of such an approximate solution and we have to use analysis on $B$ to produce one.

To this end, it is natural to linearise the rescaled Apostolov--Salamon equations on the limiting Calabi--Yau 3-fold
$(B,\omega_0,\Omega_0)$. This results in a system of coupled linear PDEs for an infinitesimal deformation $(\sigma,\rho+i\hat{\rho})$ of the {\suthreestr} $(\omega_0,\Omega_0)$, a function $h$ and a connection $1$-form $\theta$. The linearised equations simplify somewhat if one assumes that $(B,\omega_\epsilon)$ is a fixed symplectic manifold, in which
case we can take the $2$-form $\sigma$ to be zero; since this will be the only case relevant to this paper hereafter 
we restrict attention to this case.

The linearised equations contain a subsystem of equations involving only the function $h$ and the connection $1$-form $\theta$
and which is well known: the abelian Calabi--Yau monopole equations
\[
dh = *(d\theta\wedge\Real{\Omega_0}), \qquad d\theta\wedge\omega_0^2=0.
\]
Since $d\Real{\Omega_0}=0$ 
these equations force the function $h$ to be harmonic; therefore in many cases one can conclude that 
$h$ must be constant. In this case the Calabi--Yau monopole equations reduce to the (abelian) 
Hermitian Yang--Mills (HYM) equations, \ie the condition that the curvature $2$-form $d\theta$ be a primitive $(1,1)$-form.
This is also equivalent to the condition that $d\theta$ be $\omega_0$-anti-self-dual, \ie $\ast d\theta = -\omega_0 \wedge d\theta$.
For some problems it is of interest to consider solutions to the linearised Apostolov--Salamon equations 
in which we allow genuine Calabi--Yau monopoles with nonconstant $h$; 
one natural way to obtain such solutions is to permit Dirac-type singularities along a smooth compact special Lagrangian 
submanifold $L$ in $B$. This case will be treated in a future paper; in this paper we consider only solutions to the linearised Apostolov--Salamon equations where $\theta$ is HYM.

In the Hermitian Yang--Mills case $\theta$ satisfies the equations
\[
\label{eqn:HYM}
d\theta \wedge \omega_0^2 =0= d\theta \wedge \Real{\Omega_0}.
\tag{HYM}
\]
In our setting where $B$ is a complete non-compact asymptotically conical Calabi--Yau $3$-fold
we will use analytic methods to prove that a HYM connection $\theta$ exists on \emph{any} circle bundle over $B$. 
Having found such a HYM connection $\theta$ we can therefore consider the family of $S^1$-invariant \gtstr s $\varphi_\epsilon = \epsilon\, \theta \wedge \omega_0 + \Real{\Omega_0}$.
However,
$\varphi_\epsilon$ is \emph{not} closed in general since $d\varphi_\epsilon = \epsilon d \theta \wedge \omega_0$. 
For \gtstr s at present we only know how to reduce the torsion-free condition to an equivalent tractable elliptic PDE problem in the case of \emph{closed} structures with small torsion \cite[\S\S 10.3-10.4, \S 11.6]{Joyce:Book}. 
So before we can expect to proceed further we must understand how to improve $\varphi_\epsilon$ to produce a family 
of closed $S^1$-invariant \gtstr s with small torsion.
Given a HYM connection $\theta$ the remaining part of the linearised Apostolov--Salamon equations is the following 
linear inhomogeneous system
\[
\label{eqn:LAS}
d \rho = -d\theta \wedge \omega_0, \quad d \hat{\rho} =0.
\tag{LAS}
\]
The first equation  of \eqref{eqn:LAS} will guarantee that the modified  family of $S^1$-invariant \gtstr s $\varphi^{(1)}_\epsilon = \epsilon\, \theta \wedge \omega_0 + \Real{\Omega_0}+ \epsilon \rho$ are all closed. 
However the solvability of the first equation of \eqref{eqn:LAS} now imposes a
topological constraint on the circle bundle $M \ra B$, namely that 
\[
\label{eqn:c1:cup:omega}
c_1(M) \cup [\omega_0] = 0 \in H^4(B).
\tag{CH1}
\]
We will establish that provided \eqref{eqn:c1:cup:omega}
is satisfied then one can indeed solve \eqref{eqn:LAS}.
%
To control the asymptotic geometry of the 
metric induced on $M$ by the \gtstr~we will also clearly need to control the asymptotics of the solutions to the systems of linear PDEs that we produce.

Therefore the first step in implementing the proof strategy is to adopt an appropriate analytic framework for doing elliptic analysis 
on our complete non-compact Calabi--Yau $3$-fold $B$. More specifically, 
 we need to choose appropriate function spaces adapted to the asymptotic geometry of the non-compact Calabi--Yau $3$-fold $B$. 
 On AC manifolds it is now well understood that suitable weighted H\"older spaces are the appropriate setting 
in which to develop the analytic results required. 
While this general analytic framework is now well established, as we will see below we need some rather detailed analytic properties
for a number of rather particular operators. 
Therefore even assuming this basic analytic framework, establishing all the linear analytic properties  we require entails
a non-trivial amount of work: it occupies both Sections \ref{sec:CY:cones} and \ref{sec:forms:ac:cy} in this paper.
 
Given these preliminary results, our first main goal, achieved in Theorem \ref{thm:GH:G2:linearised:HYM},  is to solve the linearised Apostolov--Salamon equations \eqref{eqn:HYM} and \eqref{eqn:LAS} for the quadruple $(\sigma, \rho+i\hat{\rho},h,\theta)$, 
under our standing assumption that the function $h$ and the $2$-form $\sigma$ both vanish. 
First we prove that any circle bundle $M$ over an irreducible AC Calabi--Yau $3$-fold $B$ 
admits a HYM connection $\theta$. This follows by applying 
a non-compact Hodge-type theorem on $B$ for closed and coclosed $2$-forms with appropriate decay, 
together with vanishing results for decaying harmonic functions and $1$-forms on $B$. To solve the remaining inhomogeneous 
linear system \eqref{eqn:LAS} for $\rho$ and $\hat{\rho}$ turns out to be equivalent to solving the equation
\[
\label{eqn:LAS'}
(d+d^*) \rho= *d\theta
\tag{LAS'}
\]
and then to define $\hat{\rho}$ to be equal to $-\!\ast\!\!\rho$. 
Here the inhomogeneous term $d\theta$ is the curvature of the HYM connection $\theta$ already constructed. The advantage of \eqref{eqn:LAS'} is that the general theory of elliptic operators on weighted spaces 
applies immediately to analyse the obstructions to solving this equation. By understanding the closed and coclosed 2-forms on $B$ within a certain range of decay rates
we are able to establish that the necessary condition 
$[*d \theta]=[-d\theta \wedge \omega_0]= -c_1(M) \cup [\omega_0]=0 \in H^4(B)$ also suffices to solve \eqref{eqn:LAS'} and therefore also \eqref{eqn:LAS}. With such solutions $\theta$ and $\rho$ now in hand the $S^1$-invariant \gtstr s $\varphi^{(1)}_\epsilon = \epsilon\, \theta \wedge \omega_0 + \Real{\Omega_0}+ \epsilon \rho$ are all closed and have torsion of order $O(\epsilon^2)$. Moreover, our control of the asymptotic behaviour of the solutions $\theta$ and $\rho$ 
is sufficient to conclude that the metrics induced by $\varphi^{(1)}_\epsilon$ are all ALC metrics in the sense explained in the statement 
of the main theorem. These  $S^1$-invariant \gtstr s $\varphi^{(1)}_\epsilon$ constitute our background approximate solutions 
and we now seek to prove that for $\epsilon>0$ sufficiently small we can correct them to $S^1$-invariant torsion-free \gtstr s $\varphi_\epsilon$
maintaining the ALC metric asymptotics.

\subsubsection*{Perturbation to a torsion-free \gtstr.}
To achieve this correction to torsion-free there are two main approaches. At first sight the most obvious approach, given all previous analytic constructions 
of torsion-free \gtstr s, is to attempt to adapt Joyce's perturbation theory \cite[\S\S 10.3-10.4]{Joyce:Book} to our present non-compact setting. 
There are two main complications to doing this. The first is that such analysis would then take place on the
$7$-manifold $M$ endowed with an asymptotically locally conical (ALC) metric. 
So we would need to develop the full Fredholm package for (sufficiently general) elliptic operators on appropriate weighted function spaces on such ALC spaces. 
While this theory can indeed be developed it is not currently available in the literature
and presenting the required ALC weighted elliptic analysis here would considerably add to the length of this paper.  

However, even with the requisite ALC elliptic 
analysis package in place, Joyce's perturbation theory includes an $L^2$-smallness requirement on the torsion \cite[Theorem 11.6.1]{Joyce:Book}
which in our non-compact setting entails sufficiently fast decay of the torsion. 
In our case the decay of the torsion of  the closed $S^1$-invariant \gtstr s $\varphi^{(1)}_\epsilon$ that we construct in the first step is significantly too 
slow to be able to apply Joyce's general closed almost torsion-free \gtwo--perturbation machinery directly. We therefore need to find a way to improve the decay of the torsion 
of our approximate solutions sufficiently so that the general machinery would then apply. 

To achieve this the most natural approach seems to be to remain downstairs, \ie to work directly with the Apostolov--Salamon equations on $B$,
rather than with the torsion-free equations for the \gtstr~on $M$. The strategy is then to understand the detailed mapping properties of the linearisation of the Apostolov--Salamon equations on our approximate solutions. This way one can then attempt to construct successively better $1$-parameter families of approximations $\varphi^{(k)}_\epsilon$ to 
the torsion-free structures $\varphi_\epsilon$ which have torsion of order $O(\epsilon^{k+1})$. 
By iterating this process a finite number of times downstairs on $B$ we will be able to improve the decay of the torsion of  $\varphi^{(k)}_\epsilon$ sufficiently for the general theory to apply
to the resulting $S^1$-invariant ALC \gtstr~on $M$.

However, it turns out that we need to iterate this process
three times before we can guarantee sufficient decay of the torsion. 
To be able to iterate this many times forces us to understand completely the 
mapping properties of the linearisation of the Apostolov--Salamon equations on our approximate solutions.
With those mapping properties understood one can then continue this iteration process indefinitely 
and construct a solution to the original nonlinear Apostolov--Salamon equations globally on $B$ 
as a formal power series solution in $\epsilon$.
To construct a genuine solution one needs to solve the 
linearised equations with suitable estimates in order to prove that 
the formal power series solution actually converges (in appropriate function spaces) for sufficiently small $\epsilon$.

Therefore the most efficient implementation of our proof strategy is to work with the Apostolov--Salamon equations directly
using analysis on the AC space $B$; 
we thereby avoid analysis on the ALC space $M$ altogether.
We regard this approach to solving the Apostolov--Salamon equations directly in terms of a convergent formal power series 
 as a natural extension of the Tian--Todorov \cite{Tian,Todorov} approach to 
proving unobstructedness of complex structure deformations on Calabi--Yau manifolds.

\subsubsection*{Solving the Apostolov--Salamon equations iteratively.}
Having described the basic technical approach we adopt and explained why it seems the most efficient method,
we now describe in more detail how we go about finding a formal power series solution 
to the nonlinear Apostolov--Salamon equations working directly on $B$.
We therefore look for an {\suthreestr} $(\omega_\epsilon,\Omega_\epsilon)$ with $\omega_\epsilon=\omega_0$ and a positive function $h_\epsilon$ on $B$ and a connection $\theta_\epsilon$ on $M\ra B$ expressed as power series in $\epsilon$. We write schematically
\[
\varphi_\epsilon = \left( h_\epsilon,\epsilon\,\theta_\epsilon,\Real\Omega_\epsilon\right) = (1,0,\Real\Omega_0) + \sum_{k\geq 1}{\epsilon^k \varphi_k}.
\]
In view of the solution $(\theta,\rho)$ of the linearised Apostolov--Salamon equations discussed above we
set $\varphi_1 = (0,\theta,\rho)$.

Suppose that the triples $\varphi_i$ have been determined for $i \le k-1$ and that the Apostolov--Salamon 
equations are satisfied up to terms of order $O(\epsilon^k)$.
Then the vanishing of the degree $k$ truncation of the Apostolov--Salamon system is equivalent 
to a first-order linear system of PDEs 
\begin{equation*}
\label{eqn:LINk}
\mathcal{L}(\varphi_k)=\Theta_k.
\tag{$\mathcal{L}_k$}
\end{equation*}
Here $\mathcal{L}$ is a given linear first-order operator, depending only on the background AC Calabi--Yau $3$-fold 
$(B,\omega_0,\Omega_0)$, 
acting on the triple $\varphi_k$
and $\Theta_k$ is an algebraic expression in $\varphi_1, \ldots ,\varphi_{k-1}$.

To construct a formal power series solution to the Apostolov--Salamon equations 
the task remaining is therefore to understand the mapping properties of this first-order linear operator 
$\mathcal{L}$ and in particular to guarantee that at every order the error term $\Theta_k$ 
belongs to the image of $\mathcal{L}$. To go further and actually obtain 
a convergent power series solution for $\epsilon$ sufficiently small, in addition one needs 
uniform estimates for the solution $\varphi_k$ in terms of the size of the order $k$ error $\Theta_k$.

The linear system \eqref{eqn:LINk} contains a decoupled subsystem involving the linearisation of the Calabi--Yau monopole equations at the HYM connection $\theta$. Because of special algebraic facts about the spinor bundle of a Calabi--Yau $3$--fold, this decoupled subsystem can be reinterpreted as an inhomogeneous Dirac 
equation. For a wide range of decay rates we prove that the Dirac operator $\slashed{D}$ acting on 
weighted H\"older spaces on $B$ 
is an isomorphism. Therefore the subsystem of \eqref{eqn:LINk} consisting of the linearised Calabi--Yau monopole equations can always be solved in weighted H\"older spaces given an appropriate choice of weight.

Once the subsystem of  \eqref{eqn:LINk} consisting of the linearised Calabi--Yau monopole equations has been solved, the remaining equations 
in the system take the form
\[
d\rho_k = \alpha_{k}, \qquad d\hat{\rho}_k = \beta_{k}.
\]
Here $\alpha_{k}, \beta_{k}$ are given exact $4$-forms, while the unknown $\rho_k +i \hat{\rho}_k$ is an infinitesimal deformation of the complex volume form $\Omega_0$, that is, $\hat{\rho}_k$ is the image of $\rho_k$ under the linearisation of Hitchin's duality map for stable $3$-forms on $6$-manifolds. This linearisation is a composition of the Hodge star operator together with a pointwise decomposition of $3$-forms into irreducible representations of $\sunitary{3}$. Concretely
any $3$-form $\rho$ can be decomposed as $\rho^+ + \rho^-$ so that 
$\hat{\rho}$ is equal to $\ast \rho^+ - \!\ast \rho^-$. 

To understand the mapping properties of the operator $\rho \mapsto (d\rho, d\hat{\rho})$ is non-trivial. There are two main ingredients 
to our approach. The first is to observe that it is possible to introduce additional degrees of freedom---a pair 
of real functions and a vector field on $B$---to the Apostolov--Salamon equations. The key observation, see Proposition \ref{prop:Modified:GH:G2}, 
is that under certain conditions any solution of these extended Apostolov--Salamon
equations in fact solves the original Apostolov--Salamon equations. The proof of this relies on the existence for any {\suthreestr}
of relations between different torsion components that appear in the exterior derivatives of the defining forms $\omega$ and $\Omega$.

The second ingredient is a careful use of the Dirac operator to derive normal forms for suitably decaying exact $4$-forms on $B$ (Proposition \ref{prop:Exact:4form:AC:CY} and Corollary \ref{cor:Exact:4form:AC:CY}). These normal forms are used to relate the operator $\rho \mapsto (d\rho, d\hat{\rho})$ to the more familiar operator $\rho \mapsto (d\rho, d {\ast}\rho)$ and thereby understand the mapping properties of the linearised first-order operator $\mathcal{L}$. In fact, motivated by the extended Apostolov--Salamon equations, 
we will understand the mapping properties of an extended linearisation in which we perturb $\mathcal{L}$ by additional terms arising from a pair of functions and a vector field. In Theorem \ref{thm:Linearisation:GH:G2} we obtain unique solvability
of this extended linearised operator and uniform estimates for the size of the solution in terms of the size of the initial data, 
provided we work in weighted H\"older spaces within a certain range of decay rates.

A final technical point concerns the proof of convergence of the formal power series solutions to the Apostolov--Salamon equations. In adiabatic limit problems, once it has been established that the $\epsilon$-dependent equations can be solved to arbitrarily high order in $\epsilon$, it is customary to truncate the formal power series solution to sufficiently high degree and apply an appropriate version of the Implicit Function Theorem to perturb the resulting approximate solution into an exact solution of the problem. Since the space of {\suthreestr}s on $\R^6$ is a non-linear manifold, truncation of the power series solutions to high order in $\epsilon$ must be followed by a pointwise exponential map. The linearisation of the Apostolov--Salamon equations is sufficiently complicated that we were unable to choose this algebraic map in a suitable way to set up a contraction mapping argument. Instead, exploiting our uniform estimates for solutions to \eqref{eqn:LINk}, we prove that the power series solution constructed by solving \eqref{eqn:LINk} iteratively for all $k\geq 1$ has a positive radius of convergence in appropriate weighted H\"older spaces, adapting the argument used by Kodaira--Nirenberg--Spencer in their proof of the existence of deformations of complex structures \cite{Kodaira:Nirenberg:Spencer}. The fact that we construct solutions that depend real analytically on the parameter $\epsilon$ might appear surprising at first, since from the $7$-dimensional perspective the collapsed Calabi--Yau limit at $\epsilon=0$ is a degenerate solution. However, imposing $\tu{S}^1$--invariance, working directly in $6$ dimensions and considering only the case of collapse with bounded curvature make the rescaled Apostolov--Salamon equations depend real analytically on the parameter $\epsilon$ up to and including $\epsilon=0$.

\subsubsection*{Further extensions}
The method developed in the present paper allows the construction of a plethora of new complete non-compact
\gtmetric s with ALC geometry many of which, unlike previous methods, admit only a 1-dimensional isometry group. 
However the method as developed herein does not cover all  four of the conjectural/numerical cohomogeneity one 
families of ALC \gtmetric s mentioned at the outset: it covers collapse to the 
small resolutions of the conifold and to $K_{\CP^1\times\CP^1}$ (the families denoted by $\mathbb{D}_7$ and $\mathbb{C}_7$ 
respectively in the physics literature), but not to the smoothing of the conifold, nor to its quotient by the standard anti-holomorphic involution
(the families denoted by $\mathbb{B}_7$ and $\mathbb{A}_7$ respectively).
We explain now briefly what extensions to the current method would be needed to cover the remaining cases. 
These extensions will appear in future work.

All the known/conjectured $1$-parameter families of cohomogeneity one ALC \gtmetric s do admit a collapsed $6$-dimensional  limit. 
However, in some cases the collapse occurs with \emph{unbounded} rather than bounded curvature. In the simplest such case, namely
collapse to the Candelas--de la Ossa--Stenzel AC metric on the smoothing of the conifold, collapse with unbounded curvature occurs but still 
in a circle-invariant way (this is no longer the case in the $\mathbb{A}_7$ family). However in this case the
isometric circle action is no longer free. The fixed point set of the circle action is a smooth 
compact 3-manifold (a 3-sphere in the case of the smoothing of the conifold) 
and collapse with bounded curvature occurs away from this fixed point set. 
In this case the $7$-manifold $M$ can be thought of as the union of two pieces: the neighbourhood 
of the fixed point set and the part sufficiently far from the fixed point set. Close to the collapsed limit, the latter part 
is still modelled by a circle bundle over (a subset of) an AC Calabi--Yau base, 
but the topology and geometry close to the fixed point set of the circle action requires a different model. 
A systematic analytic treatment of such circle-invariant collapse with unbounded curvature is therefore
more complicated and we will not consider it further in the current paper. 
At the linear level it would involve studying solutions to the Calabi--Yau monopole equations 
with Dirac-type singularities along a smooth compact $3$-manifold.

One further important point to stress here is that, while this extension to circle actions with non-trivial fixed point sets is certainly very natural 
(both mathematically and physically), for most practical purposes the circle bundle construction presented in this paper is, for the foreseeable future, likely to remain the most powerful method for the construction of large families of complete non-compact \gtwo--metrics. The reason 
for this is that the smooth compact $3$-manifolds along which the Dirac-type singularities occur are required to be special 
Lagrangian $3$-folds. At present, although we can certainly exhibit some specific smooth compact special Lagrangian $3$-folds in specific relatively
simple asymptotically conical Calabi--Yau $3$-folds, we have no general analytic methods to construct special Lagrangians
independent of the precise details of the Calabi--Yau metric. 
Until such methods are developed the power of such an extension to our present circle bundle 
construction method will therefore remain more limited as a practical tool to construct \gthol metrics. 
Again this is in contrast with the case of circle-invariant $4$-dimensional hyperk\"ahler metrics where the interesting complete examples all have non-trivial fixed point sets: but
since the fixed point sets are isolated points in $\R^3$ they have no further geometric structure. 

\subsubsection*{Connections to physics}
There is a direct connection between our work in this paper and the so-called weak coupling Type IIA String Theory limit 
of M theory: one can think of the work in this paper as establishing a precise geometric meaning to some of the statements 
made or conjectured in the physics literature (at least in the non-compact setting within which we work). In this weak coupling 
limit there is supposed to be a relation between one physical theory in 11 dimensions 
and another physical theory in 10 dimensions. One should imagine that seven dimensions of the 11-dimensional space-time take the form 
of a \gthol manifold $M^7$ which itself is a circle bundle (or similar) over a $6$-dimensional space $B^6$, which should itself be thought of as forming six of the spatial dimensions in the 10-dimensional space-time.
Geometrically, this weak coupling limit is precisely the highly collapsed limit in which the {\gtmfd}  $M^7$ 
Gromov--Hausdorff converges to the Calabi--Yau $3$-fold $B$. Therefore our description of the solutions to the 
Apostolov--Salamon equations (which constitute part of the equations in Type IIA String Theory) by expanding in terms of a power series 
solution about the Calabi--Yau limit in powers of the scale $\epsilon$ of the circle fibres seems to accord very well with physical expectations. 

The free isometric circle actions considered throughout this paper correspond physically to the case of IIA theory
with a non-zero Ramond--Ramond (RR) $2$-form flux but no $D6$-branes. The richness of the possible solutions that we find in the absence of any $D6$-branes does not seem to have been 
anticipated in the physics literature.
The inclusion of $D6$-branes in the IIA theory would correspond mathematically 
to the extension of our theory mentioned earlier in which the global isometric circle action acquires a $3$-dimensional fixed point set. 

We note also that physicists have proposed dualities between pairs of Type IIA String Theory  backgrounds: one involving only RR fluxes 
on one (non-compact) Calabi--Yau 3-fold and the other involving D6-branes on a different (non-compact) Calabi--Yau 3-fold. 
In its simplest manifestation this duality is between $N$ units of RR flux through the $\CP^1$ in the small resolution of the conifold 
and $N$ $D6$-branes wrapping the 3-sphere in the smoothing of the conifold \cite{vafa:top:strings}.
From the physics perspective some useful geometric insight into this IIA duality was obtained by considering their lifts to M theory, 
\ie by considering the associated \gthol manifolds \cite{Acharya:sym,Atiyah:Maldacena:Vafa}. 
(We note in passing that  Atiyah--Maldacena--Vafa  \cite{Atiyah:Maldacena:Vafa}  predicted the existence of the $1$-parameter families of 
ALC \gtmetric s $\mathbb{B}^7$ and $\mathbb{D}^7$ that collapse to the smoothing  and the small resolution of the conifold 
respectively and what their isometry groups ought to be).
Our current work puts aspects of these physical considerations on firmer mathematical foundations, 
while these proposed dualities give further motivation for pursuing the extension of our analytic approach mentioned above.

\subsection*{Organisation} The rest of the paper is organised as follows. 
In Section \ref{sec:su3} we recall some basic results about {\suthreestr}s on 
$6$-manifolds and in particular about torsion-free {\suthreestr}s, \ie Calabi--Yau $3$-folds. 
The reader familiar with the algebra, geometry and analysis of {\suthreestr}s is encouraged to 
skip to the next section, returning to this section as necessary. 

Section \ref{sec:GH:G2} initiates the study of torsion-free \gtstr s $(M^7\!,\varphi)$ that admit a free isometric circle action preserving the fundamental $3$-form $\varphi$. 
In Section \ref{sec:as} we describe some basic features of what we call the \emph{Apostolov--Salamon equations}, which describe the reduction of the torsion-free condition for a circle-invariant \gtstr~to the quotient space $B^6=M/S^1$.
Section \ref{sec:as:adiabatic} introduces the central idea in the paper, the study of solutions to the Apostolov--Salamon equations
via their \emph{adiabatic limit}: we view $M$ as the total space of a circle bundle over $B$ and consider 
a family of $S^1$-invariant torsion-free \gtstr s $\varphi_\epsilon$ with  fibres 
shrinking to zero length as $\epsilon \to 0$.

Sections \ref{sec:CY:cones} and \ref{sec:forms:ac:cy} establish the detailed linear analytic properties of AC Calabi--Yau 3-folds 
that we will require in the remaining sections of the paper. 
For clarity of exposition we collect together in these two sections essentially all the linear analytic results that 
will be required later in the paper to solve the Apostolov--Salamon equations.
Following the general methodology summarised in Appendix \ref{Appendix:Analysic:AC}, 
the asymptotic behaviour of solutions to the natural elliptic equations on any AC manifold is 
controlled by the behaviour of solutions to these equations on its asymptotic cone.
To this end Section \ref{sec:CY:cones} establishes basic facts about geometry and analysis 
on $3$-dimensional Calabi--Yau cones. 
The indicial roots associated with various natural elliptic operators play a particularly important role. Section \ref{sec:forms:ac:cy} applies these results to understand 
the linear analysis of differential forms on an AC Calabi--Yau 3-fold. 
The results we prove about decaying closed and coclosed 2-forms and $3$-forms and decaying exact 4-forms turn out to be of central importance.

The main result of Section \ref{sec:approx:solns} is Theorem \ref{thm:GH:G2:linearised:HYM}. 
Here we use the linear analytic results of Section \ref{sec:forms:ac:cy} to  solve the linearisation of the Apostolov--Salamon equations in the adiabatic limit on an AC Calabi--Yau $3$-fold; our solutions belong to weighted H\"older spaces  with appropriately chosen decay rates. 
For $\epsilon$ sufficiently small this theorem enables us to construct a $1$-parameter family of closed \gtstr s $\varphi^{(1)}_\epsilon$ 
whose torsion is of order $O(\epsilon^2)$ and whose asymptotic geometry is ALC.

Section \ref{sec:lin:AS} develops the tools to proceed beyond this initial approximation and to solve the Apostolov--Salamon equations 
order by order as a formal power series. The main effort goes into understanding the mapping properties 
of the linearisation of the Apostolov--Salamon system.

The main result of Section \ref{sec:g2:existence} is Theorem \ref{thm:GH:G2:Existence}, which states that every solution of the linearised Apostolov--Salamon equations provided by Theorem \ref{thm:GH:G2:linearised:HYM} can be integrated to a solution of the nonlinear Apostolov--Salamon equations and therefore concludes the proof of the main result of the paper, Theorem \ref{thm:Main:Theorem:technical}. In Section \ref{sec:as:formal} we establish the existence of formal power series solutions  $\varphi_\epsilon$ to the Apostolov--Salamon equations. 
In Section \ref{sec:ac:convergence} we establish that our formal power series solution $\varphi_\epsilon$ 
actually converges in weighted H\"older spaces for $\epsilon$ sufficiently small. 

Section \ref{sec:examples} presents some illustrative applications of Theorem \ref{thm:Main:Theorem:technical} to the construction 
of new complete ALC \gtmfd s, arising from AC Calabi--Yau 3-folds. Given all the recent work on AC Calabi--Yau 3-folds we do not attempt to be
too systematic. 

There are two Appendices. Appendix \ref{Appendix:homogenous:harmonic} contains general results, needed 
elsewhere in the paper and due to Cheeger \cite{Cheeger:unpublished,Cheeger:PNAS,Cheeger:cones}, on homogeneous harmonic forms on Riemannian cones.
Appendix \ref{Appendix:Analysic:AC} gives a telegraphic summary of the requisite features of analysis on weighted H\"older spaces on AC manifolds that underpins the technical core of the paper, Sections \ref{sec:CY:cones}--\ref{sec:g2:existence}.

\subsection*{Acknowledgements}
MH and JN would like to thank the Simons Foundation for its support of their research 
under the Simons Collaboration on Special Holonomy in Geometry, Analysis and Physics (grants \#488620, Mark Haskins, and \#488631, Johannes Nordstr\"om). LF is supported by a Royal Society University Research Fellowship. LF would also like to thank NSF for the partial support of his work under grant DMS-1608143.

All three authors would like to thank MSRI and the program organisers for providing an extremely productive environment
during the Differential Geometry Program at MSRI in Spring 2016. Research at MSRI was partially supported by NSF grant DMS-1440140.
MH would also like to thank the Simons Center for Geometry and Physics 
for hosting him in August 2016.

The authors would like to thank Bobby Acharya for useful discussions concerning intuition from physics and the connections of our work 
to physics and Jeff Cheeger for referring them to his work on the spectral geometry of cones \cite{Cheeger:PNAS,Cheeger:cones} and in particular for providing them with a copy of the unpublished preprint \cite{Cheeger:unpublished}.
The authors would also like to thank the referees for their careful reading of the paper.

\section{Preliminaries on \suthreestr s on $6$-manifolds and Calabi--Yau $3$-folds}
\label{sec:su3}

The holonomy reduction of a Riemannian $7$-manifold to $\gtwo$ is conveniently expressed as the existence of a closed and coclosed (in fact, parallel) $3$-form $\varphi$ with special algebraic properties at each point. The natural action of $\tu{GL}(7,\R)$ on $\Lambda^3(\R^7)^\ast$ has two open orbits; one of these is isomorphic to $\tu{GL}(7,\R)/\gtwo$ and we say that a $3$-form $\varphi$ on a $7$-manifold $M$ is positive if $\varphi_x$ lies in this orbit for every $x\in M$. Since the stabiliser of a positive $3$--form is conjugate to $\gtwo$, the existence of $\varphi$ is equivalent to the reduction of the frame bundle of $M$ to $\gtwo$. Moreover, since $\gtwo$ is a subgroup of $\sorth{7}$ every positive $3$-form $\varphi$ defines a Riemannian metric $g_\varphi$ and volume form $\dvol_\varphi$ on $M$.

Now, the positive $3$-forms we will consider in this paper are invariant under a circle action on $M$. As we will see in the next section, in this case the holonomy reduction to $\gtwo$ can be completely expressed in terms of the induced {\suthreestr} on the $6$-dimensional orbit space $B=M/S^1$ together with additional gauge-theoretic data. In this preliminary section we collect the facts about \suthreestr s on $6$-manifolds that will be used in the rest of the paper.

\subsection{\suthreestr s on $6$-manifolds}

In analogy to the fact that the reduction of the frame bundle of a $7$-manifold to $\gtwo$ can be described by the existence of a positive $3$-form, the reduction of the frame bundle of a $6$-manifold $B$ to $\sunitary{3}$ is equivalent to the existence of a pair of differential forms with special algebraic properties at each point of $B$. 

\begin{definition}\label{def:SU(3):structure}
An {\suthreestr} on a $6$-manifold $B$ is a pair of smooth differential forms $(\omega,\Omega)$, where $\omega$ is a real non-degenerate $2$-form and $\Omega$ is a complex volume form (\ie a locally decomposable complex $3$-form such that $\Omega\wedge\overline{\Omega}$ is nowhere zero), satisfying the algebraic constraints 
\begin{equation}\label{eq:SU(3):structure:Constraints}
\omega\wedge\Real\Omega=0, \qquad \tfrac{1}{6}\omega^3 = \tfrac{1}{4}\Real\Omega\wedge\Imag\Omega.
\end{equation}
\end{definition}
On $\R^6 \simeq \C^3$ with holomorphic coordinates $(z_1,z_2,z_3)$ we define the standard parallel {\suthreestr} $(\omega_0, \Omega_0)$ by
\[
\omega_0 = \tfrac{i}{2} \left( dz_1 \wedge d\overline{z}_1 + dz_2 \wedge d\overline{z}_2 + dz_3 \wedge d\overline{z}_3 \right), \qquad \Omega_0 = dz_1 \wedge dz_2 \wedge dz_3. 
\]
An {\suthreestr} on a $6$-manifold $B$ in the sense of Definition \ref{def:SU(3):structure} defines a reduction of the structure group of the tangent bundle of $M$ to $\sunitary{3}$ by considering the subbundle of the frame-bundle of $M$ defined by $\{ u\co \R^6 \stackrel{\sim}{\longrightarrow} T_{x}B \, |\, u^\ast (\omega_x,\Omega_x)=(\omega_0,\Omega_0) \}$.

\begin{remark}\label{rmk:SU(3):spin}
Since $\sunitary{3}\subset \sunitary{4}\simeq \spin{6}$ is precisely the stabiliser of a non-zero vector in $\C^4$, we could also define an {\suthreestr} as the choice of a spin structure on $B$ together with a non-vanishing spinor.
\end{remark}

Note that every {\suthreestr} induces a Riemannian metric $g$ because $\sunitary{3}\subset \sorth{6}$. Moreover, since $\sunitary{3}\subset \tu{GL}(3,\C)$ every $6$-manifold with an {\suthreestr} is endowed with an almost complex structure $J$: a $1$-form $\gamma$ is of type $(1,0)$ if and only if $\gamma\wedge\Omega=0$.

We collect here identities involving $1$-forms/vector fields on a manifold endowed with an {\suthreestr} $(\omega,\Omega)$. They can all be proved using a basis of $\C^3$ adapted to the standard {\suthreestr} $(\omega_0,\Omega_0)$.

\begin{lemma}\label{lem:identities:1forms}
Let $X$ be a vector field on $(B,\omega,\Omega)$. Then
\begin{enumerate}
\item $X\lrcorner\,\omega = -JX^\flat$;
\item $X\lrcorner\Imag\Omega = -(JX)\lrcorner\Real\Omega$;
\item $(X\lrcorner\Real\Omega)\wedge\omega = JX^\flat\wedge\Real\Omega=X^\flat\wedge\Imag\Omega$;
\item $(X\lrcorner\Real\Omega)\wedge\Real\Omega = X^\flat\wedge\omega^2$ and $(X\lrcorner\Real\Omega)\wedge\Imag\Omega = -JX^\flat\wedge\omega^2$;
\item $\ast X^\flat = -\tfrac{1}{2}JX^\flat\wedge\omega^2$;
\end{enumerate}
\end{lemma}

\subsubsection{Decomposition of the space of differential forms}

We will make very frequent use of the the decomposition of the $\sunitary{3}$--representation $\Lambda^\ast\R^6$ into irreducible representations. 
This is well known, see for example \cite[\S 2]{Moroianu:Nagy:Semmelmann} for a detailed presentation. 
In this section we therefore simply record all the main facts that we will subsequently need. 
The decomposition is usually stated after complexification in terms of the $(p,q)$--type decomposition induced by the standard complex structure $J_0$ on $\R^6\simeq \C^3$ and in terms of primitive forms. We prefer to stick with real representations and use the uniform notation $\Lambda^k_\ell$ for an irreducible component of $\Lambda^k\R^6$ of dimension $\ell$.

\begin{lemma}\label{lem:decomposition:forms}
We have the following orthogonal decompositions into irreducible $\sunitary{3}$--representations:
\[
\Lambda ^{2}\R^6= \Lambda^2_1 \oplus \Lambda^2_6 \oplus \Lambda^2_8,
\]
where $\Lambda^2_1 = \R\,\omega$, $\Lambda^2_6 = \{ X\lrcorner\Real\Omega \, |\, X\in\R^6\}$ and $\Lambda^2_8$ is the space of primitive forms of type $(1,1)$.
\[
\Lambda ^{3}\R^6= \Lambda^3 _6 \oplus \Lambda^3 _{1\oplus 1} \oplus \Lambda^3_{12},
\]
where $\Lambda^3_6 = \{ X^\flat\wedge \omega \, |\, X\in\R^6\}$, $\Lambda^3 _{1\oplus 1} = \R \Real\Omega \oplus \R \Imag\Omega$ and $\Lambda^3_{12}$ is the space of primitive forms of type $(1,2)+(2,1)$, \ie $\Lambda^3_{12} = \{ S_{\ast}\Real{\Omega} \, | \, S \in Sym^{2}(\R ^{6}), SJ+JS=0 \}$.
\end{lemma}
For a $6$-manifold $B$ endowed with an {\suthreestr} $(\omega,\Omega)$ we denote by $\Omega^k_\ell (B)$ the space of smooth sections of the bundle over $B$ with fibre $\Lambda^k_\ell$.

The following identities follow from \cite[Equations (12), (17), (18) and (19)]{Moroianu:Nagy:Semmelmann}.

\begin{lemma}\label{lem:Hodge:star}
In the decomposition of Lemma \ref{lem:decomposition:forms} the Hodge $\ast$ operator is given by:
\begin{enumerate}
\item $\ast\omega = \tfrac{1}{2}\omega^2$;
\item $\ast (X\lrcorner\Real\Omega) = -JX^\flat\wedge\Real\Omega = X^\flat\wedge\Imag\Omega$;
\item $\ast (\eta _{0} \wedge \omega)=-\eta _{0}$ for all $\eta _{0} \in \Omega ^2_8$;
\item $\ast (X^\flat \wedge \omega) = \frac{1}{2} X \lrcorner \omega ^{2}=-JX^\flat \wedge \omega$;
\item $\ast \Real{\Omega}=\Imag{\Omega}$ and $\ast\Imag\Omega = -\Real\Omega$;
\end{enumerate}
\end{lemma}

Lemma \ref{lem:Hodge:star} can be used to deduce useful identities and characterisations of the different types of forms, see \cite[Lemmas 2.10 and 2.11]{def:nK}.

\begin{lemma}\label{lem:identities:2forms}
If $\sigma = \sigma _{0}+\lambda \omega + X \lrcorner\Real{\Omega} \in \Omega ^{2}$ with $\sigma_0\in\Omega^2_8$, then the following holds:
\begin{enumerate}
\item $\ast (\sigma \wedge \omega) = -\sigma _{0}+2\lambda \omega + X \lrcorner \Real{\Omega}$;
\item $\ast (\sigma \wedge \Real{\Omega}) = -2JX^\flat$ and $\ast (\sigma \wedge \Imag{\Omega})=-2X^\flat$;
\item $\ast (\sigma \wedge \omega ^{2}) = 6\lambda$.
\end{enumerate}
In particular, $\sigma \in \Omega ^2_8$ if and only if $\sigma \wedge \omega^{2} =0=\sigma \wedge \Real{\Omega}$ if and only if $\sigma \wedge \omega =-\ast \sigma$.
\end{lemma}

\begin{lemma}\label{lem:identities:3forms}
If $\rho = \gamma \wedge \omega + \lambda \Real{\Omega} + \mu \Imag{\Omega}+ \rho_0 \in \Omega ^{3}$ with $\rho_0\in\Omega^3_{12}$, then the following holds:
\begin{enumerate}
\item $\ast (\rho \wedge \omega) = -2J\gamma$;
\item $\ast (\rho \wedge \Real{\Omega}) = -4\mu$;
\item $\ast (\rho \wedge \Imag{\Omega})=4\lambda$.
\end{enumerate}
In particular, $\rho \in \Omega ^3_{12}$ if and only if $\rho \wedge \omega =0=\rho \wedge \Omega$.
\end{lemma}

\subsubsection{Hitchin's duality map for stable $3$-forms}

Definition \ref{def:SU(3):structure} can be generalised to any higher even dimension to define $\sunitaryn$--structures on $2n$-dimensional manifolds. However, something special occurs in complex dimension $3$: the real part $\Real\Omega$ of the complex volume form $\Omega$ determines uniquely the imaginary part $\Imag\Omega$, as explained by Hitchin in \cite[\S 2.2]{Hitchin:3forms}. This is related to the fact that in real dimension $6$, $\Real\Omega_0$ is a \emph{stable} form \cite{Hitchin:Stable:forms}, \ie its orbit in $\Lambda^3(\R^6)^\ast$ under $\tu{GL}(6,\R)$ is open.

The map $\Real\Omega\mapsto\Imag\Omega$ will be referred to as \emph{Hitchin's duality map} for stable $3$-forms on $6$-manifolds. It will be useful to have an explicit formula for the linearisation of Hitchin's duality map in terms of the Hodge star $\ast$ and the decomposition of forms into types.

\begin{prop}\label{prop:Linearisation:Hitchin:dual}
Given an {\suthreestr} $(\omega,\Omega)$ on $B$, let $\rho\in\Omega^3 (B)$ be a form with small enough $C^0$--norm so that $\Real\Omega+\rho$ is still a stable form. Decomposing into types we write $\rho = \rho _6 + \rho _{1\oplus 1} + \rho _{12}$. Then the image $\hat{\rho}$ of $\rho$ under the linearisation of Hitchin's duality map at $\Real\Omega$ is
\[
\hat{\rho} = \ast (\rho_6 + \rho _{1\oplus 1}) -\ast \rho_{12}. 
\]
\end{prop}
\noindent
For a proof see \cite[Proposition 2.12]{def:nK}.

\subsubsection{The torsion of an \suthreestr}

Given a subgroup $\tu{G}$ of $\sorth{n}$, we define a $\tu{G}$--structure on a Riemannian manifold $(B^n,g)$ as a sub-bundle $\mathcal{P}$ of the orthogonal frame bundle of $B$ with structure group $\tu{G}$. The \emph{intrinsic torsion} of $\mathcal{P}$ is a measure of how far $\mathcal{P}$ is from being parallel with respect to the Levi--Civita connection $\nabla$ of $(B,g)$. More precisely, restricting $\nabla$ to $\mathcal{P}$ yields a $\Lie{so}(n)$--valued $1$-form $\Theta$ on $\mathcal{P}$. Choose a complement $\Lie{m}$ of the Lie algebra of $\tu{G}$ in $\Lie{so}(n)$. Projection of $\Theta$ onto $\Lie{m}$ yields a $1$-form on $B$ with values in the bundle $\mathcal{P}\times_\tu{G} \Lie{m}$. By abuse of notation we will still denote this $1$-form by $\Theta$. This is the (intrinsic) torsion of the $\tu{G}$--structure $\mathcal{P}$. 

For an {\suthreestr} on a $6$-manifold $B$ one can check that $\nabla (\omega,\Omega) = \Theta_\ast (\omega,\Omega)$ where $\Theta$ acts on differential forms via the representation of $\Lie{m}\subset\Lie{so}(6)$ on $\Lambda^\ast(\R^6)$. It turns out that $\Theta$ itself is uniquely recovered by knowledge of the anti-symmetric part of $\Theta_\ast (\omega,\Omega)$, \ie the knowledge of $d\omega$ and $d\Omega$.

\begin{prop}[Chiossi--Salamon {\cite[\S 1]{Chiossi:Salamon}}]\label{prop:Torsion:SU(3):structures}
Let $(\omega,\Omega)$ be an {\suthreestr} on $B^6$. Then there exist functions $w_1,\hat{w}_1$, primitive $(1,1)$--forms $w_2,\hat{w}_2$, a $3$-form $w_3\in\Omega^3_{12}(B)$ and $1$-forms $w_4,w_5$ on $B$ such that
\[
\begin{aligned}
&d\omega = 3w_1 \Real\Omega + 3\hat{w}_1 \Imag\Omega + w_3 + w_4 \wedge\omega,\\
&d\Real\Omega = 2\hat{w}_1\omega^2 + w_5\wedge\Real\Omega + w_2 \wedge\omega,\\
&d\Imag\Omega =-2w_1\omega^2 +w_5 \wedge\Imag\Omega + \hat{w}_2\wedge\omega.
\end{aligned}
\]
Moreover $(w_1,\hat{w}_1,w_2,\hat{w}_2,w_3,w_4,w_5)$ is identified with the intrinsic torsion of the \suthreestr.
\end{prop}

\begin{remark}\label{rmk:Torsion:SU(3):Spinor}
In Remark \ref{rmk:SU(3):spin} we observed that an {\suthreestr} on a $6$-manifold can be defined as a spin structure together with a non-vanishing spinor $\psi$ (normalised to have unit length). In this spinorial approach, the intrinsic torsion $\Theta$ of the {\suthreestr} can be identified with $\nabla\psi$.  
\end{remark}

\subsection{Calabi--Yau $3$-folds}

We now restrict to torsion-free \suthreestr s.

\begin{definition}\label{def:CY3}
A Calabi--Yau $3$-fold is a $6$-manifold $B$ endowed with a torsion-free {\suthreestr} $(\omega,\Omega)$, \ie an {\suthreestr} $(\omega,\Omega)$ such that
\[
d\omega =0= d\Omega.
\]
\end{definition}

\begin{remark}\label{rmk:CY3:Spinor}
By Remark \ref{rmk:Torsion:SU(3):Spinor}, a Calabi--Yau $3$-fold can be equivalently defined as a spin $6$-manifold together with a parallel spinor.
\end{remark}
 
We will now collect algebraic identities for a number of differential operators acting on functions and $1$-forms on a Calabi--Yau $3$-fold $(B,\omega,\Omega)$. These operators are defined using exterior differential, codifferential and the type decomposition of differential forms on $B$ and will appear at various points in the rest of the paper.

First of all, we define a \emph{curl-operator} on $1$-forms by
\begin{equation}\label{eq:curl}
\tu{curl} \,\gamma=\ast (d\gamma\wedge\Real\Omega).
\end{equation}

\begin{prop}\label{prop:differential:2forms}
For every $f\in C^\infty(B)$, $\gamma\in\Omega^1(B)$ and vector field $X$ on $B$ we have
\begin{enumerate}
\item $d(f\omega)=df\wedge\omega$ and $d^\ast (f\omega) = -\!\ast d\bigl( \frac{1}{2}f\omega^2\bigr) = J df$;
\item $d(f\omega^2) = df\wedge\omega^2$ and $d^\ast(f\omega^2) = 2Jdf\wedge\omega$;
\item $d^\ast \gamma = \ast (dJ\gamma\wedge\tfrac{1}{2}\omega^2)$;
\item $d\gamma = -\tfrac{1}{3}d^\ast (J\gamma)\,\omega +\tfrac{1}{2}\left( J\tu{curl}\,\gamma\right)^\sharp\lrcorner\Real\Omega + \pi_8(d\gamma)$;
\item $d(X \lrcorner \Real{\Omega}) = \tfrac{1}{2}\tu{curl}\, X^\flat \wedge \omega -\tfrac{1}{2}(d^{\ast}X^\flat) \Real{\Omega} +\tfrac{1}{2}d^{\ast}(JX^\flat) \Imag{\Omega} + \pi_{12}\left( d(X \lrcorner \Real{\Omega})\right)$;
\item $d^\ast(X\lrcorner\Real{\Omega})=J\tu{curl}\,X^\flat$.
\end{enumerate}
\proof
The identities in (i) and (ii) follow immediately from Lemmas \ref{lem:identities:1forms} (i) and \ref{lem:Hodge:star} (i) and the fact that $\omega$ is closed. Similarly, (iii) follows immediately from Lemma \ref{lem:identities:1forms}. Now (iv) follows from (iii), the definition of $\tu{curl}$ and Lemma \ref{lem:identities:2forms} (ii).

In order to derive (v) we use Lemma \ref{lem:identities:3forms} to study each component of $d(X \lrcorner \Real{\Omega})$ independently. For instance, in order to identify its $\Omega^3_6$--component it is enough to consider $d(X \lrcorner \Real{\Omega}) \wedge \omega$ and use Lemmas \ref{lem:identities:1forms} (iii) and \ref{lem:identities:2forms} (ii). Similarly, the $\Omega^3_{1\oplus 1}$--component of $d(X \lrcorner \Real{\Omega})$ can be explicitly understood using Lemma \ref{lem:identities:1forms} (iv) and the identity (iv) that we just proved. Finally,
\[
d^\ast (X\lrcorner\Real{\Omega}) = -\!\ast \!d(X^\flat\wedge\Imag\Omega) = -\!\ast\! \bigl( dX^\flat \wedge\Imag\Omega\bigr) = J\tu{curl}\,X^\flat
\]
using Lemma \ref{lem:identities:2forms} (ii) and the definition of $\tu{curl}$.
\endproof
\end{prop}

By exploiting Hitchin's notion of stable forms we can add the following identities.

\begin{lemma}\label{lem:d:3forms:type6}
For every vector field $X$ on $B$ we have
\[
\ast d(X\lrcorner\Real\Omega) - d(X\lrcorner\Imag\Omega)\in\Omega^3_{12}, \qquad \ast d(X\lrcorner\Real\Omega) + d(X\lrcorner\Imag\Omega)\in\Omega^3_{1\oplus 1}\oplus\Omega^3_6.
\]
\proof
Since $\Omega$ is closed, $d(X\lrcorner\Real\Omega) = \mathcal{L}_{X}\Real\Omega$ and $d(X\lrcorner\Imag\Omega) = \mathcal{L}_X\Imag\Omega$. Moreover, the equivariance of Hitchin's duality map for stable $3$-forms under diffeomorphisms implies that $\mathcal{L}_X\Imag\Omega$ is the image of $\mathcal{L}_X\Real\Omega$ under the linearisation of Hitchin's duality map. Thus Proposition \ref{prop:Linearisation:Hitchin:dual} yields
\[
\pi_{1\oplus 1\oplus 6}\mathcal{L}_X\Imag\Omega = \ast \pi_{1\oplus 1\oplus 6}\mathcal{L}_X\Real\Omega = \pi_{1\oplus 1\oplus 6}\ast\mathcal{L}_X\Real\Omega, \qquad \pi_{12}\mathcal{L}_X\Imag\Omega  =  -\pi_{12}\ast\mathcal{L}_X\Real\Omega.\qedhere
\]
\end{lemma}

\begin{remark}\label{rmk:d:3forms:type6}
Using Proposition \ref{prop:differential:2forms} (iv) and Lemmas \ref{lem:identities:1forms} (ii) and \ref{lem:Hodge:star} (iv), the $\Omega^3_6$--component of this identity yields
\[
\tu{curl}\, JX^\flat = -J\tu{curl}\, X^\flat
\]
for every vector field $X$ on $B$.
\end{remark}

\begin{remark}\label{rmk:d:3forms:type6:1}
A further consequence of Lemma \ref{lem:d:3forms:type6} is that for every $\rho_0\in\Omega^3_{12}$ and (compactly supported, say) vector field $X$ we have
\[
\langle d\rho_0, \ast (X\lrcorner\Real\Omega)\rangle_{L^2} = \langle d^\ast\rho_0, X\lrcorner\Imag\Omega\rangle_{L^2}.
\]
We therefore deduce that there is a relation between the $\Omega^4_6$--components of $d\rho_0$ and $d\ast\rho_0$. This is the linearised version of the relation between the $\Omega^4_6$--components of $d\Real\Omega$ and $d\Imag\Omega$ in Proposition \ref{prop:Torsion:SU(3):structures} since $\rho_0-i\!\ast\!\rho_0$ is an infinitesimal deformation of the complex volume form $\Omega$ by Proposition \ref{prop:Linearisation:Hitchin:dual}. In particular, if $d^\ast\rho_0=0$ then $d\rho_0\in \Omega^4_8$. Indeed, for every $\rho_0\in \Omega^3_{12}$ we have $\pi_1 (d\rho_0)=0$ since $d\rho_0\wedge\omega = d(\rho_0\wedge\omega)=0$.
\end{remark}

We will also need to consider the following second-order operators.

\begin{lemma}\label{lem:Gauge:fixing:operator:AC:CY}
Let $(B,\omega,\Omega)$ be a Calabi--Yau $3$-fold. Then for every function $g$
\[
\pi_1 dd^\ast \bigl( \tfrac{1}{2}g\omega^2\bigr) = \tfrac{1}{3}(\triangle g)\omega^2
\]
and the operator $(f,\gamma) \mapsto \pi_{1\oplus 6}d^\ast d \bigl( f\omega + \gamma^\sharp\lrcorner\Real\Omega\bigr)$ can be identified with
\[
(f,\gamma)\longmapsto \left( \tfrac{2}{3}\triangle f, dd^\ast\gamma + \tfrac{2}{3}d^\ast d\gamma\right).
\]
\proof
Using Proposition \ref{prop:differential:2forms} (ii) we have $dd^\ast \bigl( \tfrac{1}{2}g\omega_0^2\bigr) = (dJdg)\wedge\omega_0$. The type decomposition of the differential of a $1$-form Proposition \ref{prop:differential:2forms} (iv) then immediately implies the first statement.

The second statement can be deduced quickly from \cite[Proposition 2.24]{Karigiannis:Lotay} applied to the torsion-free {\gtstr} $\varphi = dt\wedge\omega + \Real\Omega$ and the vector field $X=f\partial_t + \gamma^\sharp$ on $B\times\R_t$.
\endproof
\end{lemma}

\subsubsection{The Dirac operator} We close this section by deriving a formula for the Dirac operator on a Calabi--Yau $3$-fold. By Remark \ref{rmk:SU(3):spin} every $6$-manifold $B$ with an {\suthreestr} $(\omega,\Omega)$ is spin and it is endowed with a unit spinor $\psi$. Considering the decomposition into irreducible $\sunitary{3}$--representations, the real spinor bundle $\slashed{S}(B)$ is isomorphic to $\underline{\R}\oplus\underline{\R}\oplus T^\ast B$, where $\underline{\R}$ is the trivial real line bundle. The isomorphism is
\[
(f,g,\gamma)\longmapsto f\psi + g\vol\cdot\psi + \gamma\cdot\psi,
\]
where $\cdot$ denotes Clifford multiplication. We now describe the Dirac operator $\slashed{D}$ of a Calabi--Yau $3$-fold in terms of this isomorphism.

By Remark \ref{rmk:CY3:Spinor} the unit spinor $\psi$ defining the Calabi--Yau structure is parallel. In particular, $\slashed{D}\psi = 0=\slashed{D}(\vol\cdot\psi)$. Thus
\begin{subequations}\label{eq:Dirac:SU(3)}
\begin{equation}\label{eq:Dirac:SU(3):functions}
\slashed{D}\left( f\psi + g\vol\cdot\psi\right) = \left( df +Jd g \right)\cdot \psi,
\end{equation}
since the complex structure $J$ on $M$ defined by $(\omega,\Omega)$ satisfies $(J\gamma)\cdot\psi = -\vol\cdot \gamma\cdot\psi=\gamma\cdot\vol\cdot\psi$ \cite[Equation (5.11)]{Friedrich:al}. On the other hand,
\begin{equation}\label{eq:Dirac:SU(3):1forms}
\slashed{D}(\gamma\cdot\psi) = \sum_{i=1}^6{e_i\cdot\nabla_{e_i}\gamma\cdot\psi} =d\gamma\cdot\psi + (d^\ast \gamma)\, \psi.
\end{equation}
\end{subequations}

An explicit computation using an orthonormal coframe adapted to the {\suthreestr} yields the following formula for the action of a $2$-form $\sigma$ on the spinor $\psi$. Decompose $\sigma$ into types: $\sigma=\lambda\omega + Y\lrcorner\Real\Omega + \sigma_0$, with $\sigma_0$ the primitive $(1,1)$--component of $\sigma$. Then
\begin{equation}\label{eq:Clifford:multiplication:2forms}
\sigma\cdot\psi = 3\lambda \vol\cdot\psi + 2JY\cdot\psi.
\end{equation}

Combining \eqref{eq:Dirac:SU(3)}, Proposition \ref{prop:differential:2forms} (iv) and \eqref{eq:Clifford:multiplication:2forms} yields a proof of the following Lemma.

\begin{lemma}\label{lem:Dirac:cone}
Let $(M,\omega,\Omega)$ be a Calabi--Yau $3$-fold. Then for every $f,g\in C^\infty (M)$ and $\gamma\in\Omega^1(M)$
\[
\slashed{D}(f,g,\gamma) = (d^\ast\gamma, -d^\ast J\gamma, \tu{curl}\,\gamma + df-Jdg).
\]
In particular if $f=0=g$ then $\gamma$ is in the kernel of $\slashed{D}$ if and only if $d^\ast\gamma=0$ and $d\gamma$ is a primitive $(1,1)$--form.
\end{lemma}

\section{Collapsed $\tu{S}^1$--invariant torsion-free \gtstr s}\label{sec:GH:G2}

In \cite{Apostolov:Salamon} Apostolov--Salamon (see also \cite{CGLP:Almost:Special:Holonomy,Kaste:et:al:I,Kaste:et:al:II} for earlier work in the physics literature) considered \gtmfd s (necessarily reducible, non-compact or incomplete) admitting an isometric circle action. Apostolov--Salamon focus on the special case where the quotient by the circle action is a K\"ahler manifold. In this section we reconsider the general case: closer in spirit to \cite{CGLP:Almost:Special:Holonomy,Kaste:et:al:I,Kaste:et:al:II}, we interpret the dimensional reduction of the {\gthol} equations in terms of the intrinsic torsion of the {\suthreestr} induced on the $6$-dimensional quotient by the circle action and a coupled abelian Calabi--Yau monopole. These equations can be thought of as an analogue in $\gtwo$--geometry of the Gibbons--Hawking ansatz for $4$-dimensional hyperk\"ahler metrics with a triholomorphic circle action. In contrast to the $4$-dimensional hyperk\"ahler case, however, the dimensional reduction of the {\gthol} equations to $6$ dimensions still consists of a \emph{nonlinear} system of equations and it is not clear how to solve them directly. We therefore consider the \emph{adiabatic limit} of these equations when the circle fibres have small length. In this section we focus on formal aspects: after deriving the equations satisfied by a circle-invariant torsion-free {\gtstr} we
write down the adiabatic limit equations. In the rest of the paper we will use these equations to construct new ALC \gtmfd s from Hermitian Yang--Mills connections on asymptotically conical Calabi--Yau $3$-folds.

\subsection{Gibbons--Hawking-type ansatz for {\gthol} metrics} 
\label{sec:as}

Let $M^7$ be a principal circle bundle over a $6$-manifold $B$. Denote by $V$ the vector field that generates the fibrewise circle action, normalised to have period $2\pi$. Any $\tu{S}^1$--invariant {\gtstr} $\varphi$ on $M$ can be written in the form
\begin{subequations}\label{eq:circle:invariant:G2}
\begin{equation}
\varphi = \theta \wedge \omega + h^{\frac{3}{4}}\Real\Omega,
\end{equation}
where $(\omega,\Omega)$ is an {\suthreestr} on $B$, $h$ is a positive function on $B$ and $\theta$ is a connection $1$-form on the principal $\unitary{1}$--bundle $M\ra B$. In particular, $\theta (V)=1$. Note that
\begin{equation}
\ast_{\varphi}\varphi = -h^{\frac{1}{4}}\theta\wedge\Imag\Omega + \tfrac{1}{2}h\omega^2
\end{equation}
and the induced metric is
\begin{equation}
\label{eqn:g7}
g_\varphi = \sqrt{h}\, g_B + h^{-1}\theta^2,
\end{equation}
where $g_B$ is the metric on $B$ induced by the {\suthreestr} $(\omega,\Omega)$.
\end{subequations}

\begin{remark*}
There is some arbitrariness in the choice of conformal factor in front of $g_B$ in the expression for $g_\varphi$ (and therefore the definition of the forms $\omega,\Omega$ on $B$). The choice we made is convenient because, as we will see, requiring that $\varphi$ be closed implies that $\omega$ is closed, \ie $\omega$ is a symplectic form on $B$. Other choices of conformal factor simplify other equations but have the disadvantage that $\omega$ is no longer closed. Physicists have another privileged choice of conformal factor based on the notion of a \emph{string frame} (see for example \cite[\S 2]{Kaste:et:al:II}), but its mathematical significance is not clear to us.    
\end{remark*}

A straightforward calculation using \eqref{eq:circle:invariant:G2} allows one to express the equations $d\varphi=0=d\!\ast\!\varphi$ for an $\tu{S}^1$--invariant torsion-free {\gtstr} as a system of PDEs for the $4$-tuple $(\omega,\Omega,h,\theta)$ on $B$.

\begin{lemma}[Apostolov--Salamon {\cite[\S 1]{Apostolov:Salamon}}]\label{lem:GH:G2} The $\tu{S}^1$--invariant {\gtstr} $\varphi$ on $M$ determined by the $4$-tuple $(\omega,\Omega,h,\theta)$ on $B$ is torsion-free if and only if
\begin{equation}\label{eq:GH:G2}
\begin{aligned}
d\omega =0&, \qquad &d\big( h^{\frac{3}{4}}\Real\Omega\big) = -d\theta\wedge\omega,\\
d\bigl( h^{\frac{1}{4}}\Imag\Omega\bigr) = 0&, \qquad &\tfrac{1}{2}dh\wedge \omega^2 = h^{\frac{1}{4}}d\theta\wedge\Imag\Omega.
\end{aligned}
\end{equation}
\proof
From the expressions for $\varphi$ and $\ast_\varphi\varphi$ in \eqref{eq:circle:invariant:G2} we compute
\begin{gather*}
d\varphi = -\theta\wedge d\omega + d\theta\wedge\omega + d\bigl( h^{\frac{3}{4}}\Real\Omega \bigr),\\
d\!\ast_\varphi\!\varphi = \theta\wedge d\bigl( h^{\frac{1}{4}}\Imag\Omega\bigr) -  h^{\frac{1}{4}}d\theta\wedge\Imag\Omega +\tfrac{1}{2}d\bigl( h\omega^2\bigr).
\end{gather*}
Projection onto the image of $\theta\wedge\,\cdot\,$ and the kernel of $V\lrcorner\,\cdot\,$ yields the result.
\endproof 
\end{lemma}
 
\begin{remark}
\label{rmk:closedG2:obst}
Note that the first row of equations in \eqref{eq:GH:G2} is equivalent to the $\tu{S}^1$--invariant {\gtstr} $\varphi$ on $M$ determined by 
$(\omega,\Omega,h,\theta)$ being closed. 
In particular if $\varphi$ is closed then $\omega$ is a symplectic form on $B$ and we have the cohomological condition 
that 
\[
[d\theta] \cup [\omega] =c_1(M) \cup [\omega]= 0 \in H^4(B). 
\]
To obtain non-trivial torsion-free $\tu{S}^1$--invariant \gtstr s we need to 
consider non-compact $6$-manifolds $B$, but we will need to take into account the potential cohomological obstruction 
to the existence of a closed $\tu{S}^1$--invariant {\gtstr} on a given circle bundle $M$ over $B$ just described.
\end{remark}

We will refer to the equations in \eqref{eq:GH:G2} as the \emph{Apostolov--Salamon equations}. They can be interpreted as coupled equations for the torsion of the {\suthreestr} $(\omega,\Omega)$ and the pair $(h,\theta)$, \cf \cite[\S\S 3-4]{Kaste:et:al:II}.

\begin{lemma}\label{lem:Torsion:GH:G2}
The {\suthreestr} $(\omega,\Omega)$ arising from a solution $(\omega,\Omega,h,\theta)$ of \eqref{eq:GH:G2} has torsion
\begin{equation}\label{eq:Torsion:GH:G2}
w_1=\hat{w}_1=\hat{w}_2=w_3=w_4=0, \qquad w_5 = -\tfrac{1}{4}h^{-1}dh, \qquad w_2 = -h^{-\frac{3}{4}}\kappa_0,
\end{equation}
where $\kappa_0$ is the projection of the curvature $d\theta$ of $\theta$ onto the space of primitive $(1,1)$--forms. Moreover, $\big( h,\theta\big)$ satisfies
\begin{equation}\label{eq:CY:monopole:eq}
d\Bigl( \tfrac{4}{3}h^{\frac{3}{4}}\Bigr) = \ast \left( d\theta\wedge\Real\Omega\right), \qquad d\theta\wedge\omega^2=0, 
\end{equation}
or equivalently $d\theta = -\frac{1}{2}h^{-\frac{1}{4}}(J\nabla h) \lrcorner\Real\Omega + \kappa_0$, where $J$ is the almost complex structure induced by $\Omega$.

Conversely, let $(\omega,\Omega)$ be an {\suthreestr} on $B$ whose torsion satisfies \eqref{eq:Torsion:GH:G2} for some function $h>0$ and primitive $(1,1)$--form $\kappa_0$. Assume that $-\frac{1}{2}h^{-\frac{1}{4}}(J\nabla h) \lrcorner\Real\Omega + \kappa_0$ is the curvature of a connection $\theta$ on a principal circle bundle $M$ over $B$ (\ie it is a closed $2$-form representing an integral cohomology class). Then $(h,\theta,\omega,\Omega)$ is a solution to \eqref{eq:GH:G2} coming from an $\tu{S}^1$--invariant torsion-free {\gtstr} on $M$.
\proof
Rewrite the first three equations in \eqref{eq:GH:G2} as
\[
d\omega=0, \qquad d\Imag\Omega = -\tfrac{1}{4}h^{-1}dh\wedge \Imag\Omega, \qquad d\Real\Omega = -\tfrac{3}{4}h^{-1}dh\wedge \Real\Omega - h^{-\frac{3}{4}}d\theta\wedge\omega.
\]
Note that \eqref{eq:Torsion:GH:G2} follows from these equations and Proposition \ref{prop:Torsion:SU(3):structures}.
 
Now, since $\omega$ is closed and $\omega\wedge\Real\Omega=0$, then also $\omega\wedge d\Real\Omega=0$  and therefore the third equation above yields $d\theta\wedge\omega^2=0$. Moreover, using Lemma \ref{lem:identities:1forms} (iii), the relation between the $\Omega^4_6$--components of $d\Real\Omega$ and $d\Imag\Omega$ in Proposition \ref{prop:Torsion:SU(3):structures} (\ie the fact that it is the same $1$-form $w_5$ that appears in both expressions) also determines the $\Omega^2_6$--component of $d\theta$. We conclude that $d\theta = -\frac{1}{2}h^{-\frac{1}{4}}(J\nabla h) \lrcorner\Real\Omega + \kappa_0$. Using Lemma \ref{lem:identities:1forms} (iv) and Lemma \ref{lem:identities:2forms} (ii), one can then check that this expression for $d\theta$ is equivalent to $d\theta\wedge\omega^2=0$ together with the fourth equation in \eqref{eq:GH:G2} and that the latter is equivalent to $d\bigl( \tfrac{4}{3}h^{\frac{3}{4}}\bigr) = \ast \left( d\theta\wedge\Real\Omega\right)$.
\endproof
\end{lemma}

\begin{remark*}
The equations \eqref{eq:CY:monopole:eq} are gauge-theoretic equations which arise as the dimensional reduction of the (abelian) $\gtwo$--instanton equations to $6$-dimensions. They are called the \emph{abelian Calabi--Yau monopole} equations, \cf \cite[Definition 3.1.2]{Oliveira:Thesis}. While Calabi--Yau monopoles can be defined for arbitrary structure group, only the abelian $\unitary{1}$ case is relevant for this paper. 
\end{remark*}

\begin{remark*}
The Nijenhuis tensor of the almost complex structure $J$ induced by $\Omega$ only depends on $w_1,\hat{w}_1,w_2,\hat{w}_2$ \cite[p. 118]{Chiossi:Salamon}. Thus the almost complex structure $J$ arising from a solution of \eqref{eq:GH:G2} is integrable if and only if $\kappa_0=0$. This is the case considered by Apostolov--Salamon in \cite{Apostolov:Salamon}: $(B,\omega,J)$ is then a K\"ahler manifold and one can further consider its K\"ahler quotient by the Hamiltonian vector field $J\nabla h$. In general, $\kappa_0\neq 0$ and we are forced to consider non-integrable almost complex $6$-manifolds.
\end{remark*}

\subsection{The adiabatic limit of $\tu{S}^1$--invariant torsion-free \gtstr s}
\label{sec:as:adiabatic}

In contrast with the Gibbons--Hawking ansatz, which allows one to construct $4$-dimensional hyperk\"ahler metrics with a triholomorphic circle action from a positive harmonic function on an open subset of $\R^3$, the equations \eqref{eq:GH:G2} are still nonlinear and it is unclear how to find solutions in general. In this section we consider the \emph{adiabatic limit} of \eqref{eq:GH:G2} when the circle fibres shrink to zero length. The collapsed limit is a Calabi--Yau structure on $B$. The linearisation of \eqref{eq:GH:G2} at this degenerate solution reduces to a coupled system for a Calabi--Yau monopole and a $3$-form on $B$.

Let $\varphi_\epsilon$ be a family of $\tu{S}^1$--invariant torsion-free \gtstr s on the total space $M$ of a principal circle bundle over $B$ with circle fibres shrinking to zero length as $\epsilon \ra 0$. By rescaling along the fibres we write 
\[
\varphi_\epsilon = \epsilon\,\theta_\epsilon\wedge\omega_\epsilon + {(h_\epsilon)}^\frac{3}{4}\Real\Omega_\epsilon.
\]
Note that the metric induced by $\varphi_\epsilon$ is $\sqrt{h_\epsilon}\, g_\epsilon + \epsilon^2 h_\epsilon^{-1}\theta_\epsilon^2$, where $g_\epsilon$ is the metric on $B$ induced by $(\omega_\epsilon,\Omega_\epsilon)$. The PDE system \eqref{eq:GH:G2} for $\varphi_\epsilon$ then becomes
\begin{equation}\label{eq:GH:G2:collapsing:sequence}
\begin{gathered}
d\omega_\epsilon=0,\qquad
\tfrac{1}{2}dh_\epsilon\wedge \omega_\epsilon^2 = \epsilon\,  (h_\epsilon)^{\frac{1}{4}}d\theta_\epsilon\wedge\Imag\Omega_\epsilon,\qquad \epsilon\, d\theta_\epsilon\wedge\omega_\epsilon^2=0,\\
d\Real\Omega_\epsilon = -\tfrac{3}{4}h_\epsilon^{-1}dh_\epsilon\wedge \Real\Omega_\epsilon-\epsilon (h_\epsilon)^{-\frac{3}{4}} d\theta_\epsilon\wedge\omega_\epsilon,\qquad
d\Imag\Omega_\epsilon = -\tfrac{1}{4}h_\epsilon^{-1}dh_\epsilon\wedge \Imag\Omega_\epsilon.
\end{gathered}
\end{equation}
Here we have used Lemma \ref{lem:Torsion:GH:G2} to add the equation $d\theta_\epsilon\wedge\omega_\epsilon^2=0$ which is necessary for \eqref{eq:GH:G2} to hold. For $\epsilon>0$ the system \eqref{eq:GH:G2:collapsing:sequence} is equivalent to \eqref{eq:GH:G2}, but in the limit $\epsilon\ra 0$ it simplifies. Indeed, the formal limit of the second equation as $\epsilon\ra 0$ implies that $h_0=\lim_{\epsilon\ra 0}h_\epsilon$ is constant---we will assume that $h_0\equiv 1$---and hence the rest of the system implies that $\Real\Omega_0$ and $\Imag\Omega_0$ are both closed. Thus $(\omega_0,\Omega_0)$ is a Calabi--Yau structure on $B$.

We want to find a better approximation to $(h_\epsilon,\omega_\epsilon,\Omega_\epsilon)$ for small $\epsilon$ by linearising \eqref{eq:GH:G2:collapsing:sequence} on the limiting Calabi--Yau $3$-fold $(B,\omega_0,\Omega_0)$. To this end we write
\begin{alignat*}{2}
h_\epsilon &= 1+\epsilon\, h +O(\epsilon^2), & \qquad \quad \epsilon\,\theta_\epsilon &= \epsilon\,\theta + O(\epsilon^2),\\
\omega_\epsilon &= \omega_0 + \epsilon\,\sigma + O(\epsilon^2), &  \Omega_\epsilon &= \Omega_0 + \epsilon\, (\rho + i \hat{\rho}) + O(\epsilon^2).
\end{alignat*}
Here $\hat{\rho}$ is the image of the $3$-form $\rho$ under the linearisation of Hitchin's duality map for stable $3$-forms on a $6$-manifold, \cf Proposition \ref{prop:Linearisation:Hitchin:dual}. Ignoring terms of order $\epsilon^2$ in \eqref{eq:GH:G2:collapsing:sequence} implies that $(\sigma,\rho+i\hat{\rho},h,\theta)$ satisfies the following system of linear equations
\begin{equation}\label{eq:GH:G2:linearised}
\begin{gathered}
d\sigma =0,\qquad
\tfrac{1}{2}dh\wedge\omega_0^2 = d\theta\wedge\Imag\Omega_0,\qquad d\theta\wedge\omega_0^2=0,\\
d\rho = -\tfrac{3}{4}dh\wedge\Real\Omega_0-d\theta\wedge\omega_0, \qquad  
d\hat{\rho}=-\tfrac{1}{4}dh\wedge\Imag\Omega_0,\\
\omega_0\wedge\left( \rho + i\hat{\rho}\right) + \sigma\wedge \Omega_0=0,\qquad
\Real\Omega_0\wedge\hat{\rho}+\rho\wedge\Imag\Omega_0 = 2\sigma\wedge\omega_0^2.
\end{gathered}
\end{equation}
Here the last two equations are the linearisation of the algebraic constraints for an \suthreestr, \ie $\omega\wedge\Omega=0$ and $\tfrac{1}{6}\omega^3 = \tfrac{1}{4}\Real\Omega\wedge\Imag\Omega$. The system \eqref{eq:GH:G2:linearised} can be simplified further by assuming that $(B,\omega_\epsilon)$ is a \emph{fixed} symplectic manifold. We will justify this assumption in the cases of interest later in the paper. In particular we are mostly interested in solutions to \eqref{eq:GH:G2:linearised} with $\sigma=0$. 

Note that if $(h,\theta,\sigma,\rho)$ is a solution to \eqref{eq:GH:G2:linearised} then $(h,\theta)$ is an abelian Calabi--Yau monopole on the Calabi--Yau $3$-fold $(B,\omega_0,\Omega_0)$. Indeed, $\tfrac{1}{2}dh\wedge\omega_0^2 = d\theta\wedge\Imag\Omega_0$ is equivalent to $dh = \ast (d\theta\wedge\Real\Omega_0)$. In particular, since $\Real\Omega_0$ is closed we have that $h$ is a harmonic function. In many circumstances (\eg when $B$ is complete and $h$ is bounded) one then concludes that $h$ is constant. In this case the Calabi--Yau monopole equations \eqref{eq:CY:monopole:eq} reduce to the requirement that $\theta$ be a \emph{Hermitian Yang--Mills} (HYM) connection on $B$, \ie the curvature $d\theta$ is a primitive $(1,1)$--form. Solutions with non-constant $h$ can be obtained by allowing ``Dirac-type singularities'' along a special Lagrangian submanifold $L$ in $B$: given $k\in \Z$ we require that 
\[
h=\frac{k}{2\,\textup{dist}(\,\cdot\,,L)}+O(1)
\]  
in a tubular neighbourhood of $L$. Note that this implies that $\theta$ is a connection on a circle bundle on $B\setminus L$ whose first Chern class evaluated on a $2$-sphere linking $L$ is $k$.
Referring back to the metric behaviour in \eqref{eqn:g7} it is natural to consider solutions in which $h \to +\infty$ when one wants to model the adiabatic limit 
geometry in the neighbourhood of a component of the fixed point set of an isometric circle action. 

Returning to \eqref{eq:GH:G2:linearised}, a Calabi--Yau monopole $(h,\theta)$ determines an infinitesimal deformation $\rho+i\hat{\rho}$ of the complex volume form $\Omega_0$ on $B$ given as a solution of the linear inhomogeneous equations
\[
d\rho = -\tfrac{3}{4}dh\wedge\Real\Omega_0-d\theta\wedge\omega_0, \qquad  
d\hat{\rho}=-\tfrac{1}{4}dh\wedge\Imag\Omega_0.
\]

\subsection{From AC Calabi--Yau $3$-folds to ALC \gtmfd s}

We can now explain in detail the strategy to prove Theorem \ref{thm:Main:Theorem:technical}. For $\epsilon$ small enough we will solve the $\epsilon$--dependent Apostolov--Salamon equations \eqref{eq:GH:G2:collapsing:sequence} as a power series in $\epsilon$.

As we have seen, at order $0$ in $\epsilon$ the equations \eqref{eq:GH:G2:collapsing:sequence} state that $(B,\omega_0,\Omega_0)$ is a Calabi--Yau $3$-fold. In order to construct ALC {\gtmfd}s we will start with an \emph{asymptotically conical} Calabi--Yau $3$-fold.

\begin{definition}\label{def:AC:CY}
Let $(\Sigma,g_\Sigma)$ be a closed smooth connected Riemannian $5$-manifold.
\begin{enumerate}
\item The cone $\tu{C}(\Sigma)$ over $\Sigma$ is the incomplete Riemannian manifold $(0,\infty)\times\Sigma$ endowed with the metric
\[
g_\tu{C} = dr^2 + r^2 g_\Sigma.
\]
We say that $\tu{C}(\Sigma)$ is a \emph{$3$-dimensional Calabi--Yau cone} if there exists a Calabi--Yau structure $(\omega_\tu{C},\Omega_\tu{C})$ on $\tu{C}(\Sigma)$ inducing the metric $g_\tu{C}$.
\item Let $(B,g_0,\omega_0,\Omega_0)$ be a complete Calabi--Yau $3$-fold. We say that $B$ is an \emph{asymptotically conical (AC) Calabi--Yau $3$-fold} asymptotic to the Calabi--Yau cone $\tu{C}(\Sigma)$ with rate $\mu<0$ if there exists a compact set $K\subset B$, $R>0$ and a diffeomorphism $f\co (R,\infty)\times\Sigma \ra B\setminus K$ such that
\[
\left|\nabla_{g_\tu{C}} ^k\left( f^\ast\omega_0 - \omega_\tu{C}\right)\right|_{g_\tu{C}} + \left|\nabla_{g_\tu{C}} ^k\left( f^\ast\Omega_0 - \Omega_\tu{C}\right)\right|_{g_\tu{C}} =O(r^{\mu-k})
\]
for every $k\geq 0$.
\end{enumerate}
\end{definition}

\begin{remark*} Since the metric $g_0$ is uniquely determined by the {\suthreestr} $(\omega_0,\Omega_0)$ we also have
\[
\left|\nabla_{g_\tu{C}} ^k\left( f^\ast g_0 - g_\tu{C}\right)\right|_{g_\tu{C}} =O(r^{\mu-k})
\]
for all $k\geq 0$ and therefore an AC Calabi--Yau $3$-fold is an AC Riemannian manifold in the sense of Definition \ref{def:AC} in Appendix \ref{Appendix:Analysic:AC}.
\end{remark*}

A systematic theory of AC Calabi--Yau manifolds has been developed in recent years by various authors and a very satisfactory existence and uniqueness theory is available, \cf Theorem \ref{thm:AC:CY:crepant:resolutions} in the final section of the paper.

Let then $(B,\omega_0,\Omega_0)$ be an AC Calabi--Yau $3$-fold asymptotic to the Calabi--Yau cone $\tu{C}(\Sigma)$. We want to solve the Apostolov--Salamon equations \eqref{eq:GH:G2:collapsing:sequence} to first order in $\epsilon$ on a principal circle bundle $M$ over $B$. Following the discussion of the previous subsection, we look for solutions of the linear system \eqref{eq:GH:G2:linearised}. In this paper we solve this system assuming that the Calabi--Yau monopole $(h,\theta)$ reduces to a HYM connection. The case of monopoles with singularities along a smooth compact special Lagrangian submanifold will be treated in a future paper. We therefore look for solutions of \eqref{eq:GH:G2:linearised}  with $\sigma=0=h$: the resulting coupled linear system for a $U(1)$--connection $\theta$ and a $3$--form $\rho$ on $B$ is
\begin{equation}\label{eq:GH:G2:linearised:HYM:1}
\begin{gathered}
d\theta\wedge\Imag\Omega_0=0,\quad d\theta\wedge\omega_0^2=0,\qquad
d\rho = -d\theta\wedge\omega_0, \quad  
d\hat{\rho}=0,\\
\omega_0\wedge\left( \rho + i\hat{\rho}\right) =0,\quad
\Real\Omega_0\wedge\hat{\rho}+\rho\wedge\Imag\Omega_0 = 0.
\end{gathered}
\end{equation}
By Lemma \ref{lem:identities:2forms} the first two equations are equivalent to the condition that $d\theta$ be a primitive $(1,1)$--form, \ie $\theta$ is a HYM connection.

Now, every solution $(\theta,\rho)$ of the linearised equations \eqref{eq:GH:G2:linearised:HYM:1} on $(B,\omega_0,\Omega_0)$ yields a $1$-parameter family of ALC \gtstr s
\[
\varphi^{(1)}_\epsilon = \epsilon\, \theta\wedge\omega_0 + \Real\Omega_0 + \epsilon\, \rho
\]
on $M$. Since $d\varphi^{(1)}_\epsilon = \epsilon (d\theta\wedge\omega_0 + d\rho)$, the third equation in \eqref{eq:GH:G2:linearised:HYM:1} guarantees that $\varphi^{(1)}_\epsilon$ is \emph{closed}. The remaining equations in \eqref{eq:GH:G2:linearised:HYM:1} imply that the rest of the torsion is of order $O(\epsilon^2)$. We will then show that for $\epsilon$ sufficiently small any such approximate solution can be perturbed to a solution to the Apostolov--Salamon equations \eqref{eq:GH:G2:collapsing:sequence}. We will first construct a solution of the Apostolov--Salamon equations as a formal power series in $\epsilon$ by solving iteratively \eqref{eq:GH:G2:collapsing:sequence} to all order in $\epsilon$. This step will require a complete understanding of the mapping properties of the linearisation of the Apostolov--Salamon equations. In order to conclude the proof of Theorem \ref{thm:Main:Theorem:technical} we will then show that the formal power series solutions that we construct in fact converge in appropriate weighted H\"older spaces for $\epsilon$ sufficiently small. Given that we will be able to solve the linearisation of the Apostolov--Salamon equations \emph{with estimates}, this final step follows very closely Kodaira--Nirenberg--Spencer's proof of the existence of analytic deformations of complex structures \cite[\S 5]{Kodaira:Nirenberg:Spencer}.

Our main tools to implement this strategy are analytic, more specifically the theory of linear elliptic operators acting on weighted H\"older spaces $C^{k,\alpha}_\nu$ on the AC Calabi--Yau $3$-fold $B$. Though this material is by now fairly standard, to make this article more readable and self-contained, in Appendix \ref{Appendix:Analysic:AC} we have collected the most relevant background material on analysis on AC manifolds. 
These results will be used throughout the rest of the paper.

As just described, the main steps in proving Theorem \ref{thm:Main:Theorem:technical} are to construct solutions  of \eqref{eq:GH:G2:linearised:HYM:1} and to understand the mapping properties of the linearisation of the Apostolov--Salamon equations. In the next two sections we study differential forms on Calabi--Yau cones and on AC Calabi--Yau manifolds as preliminary steps to address these two goals.

\section{Three-dimensional Calabi--Yau cones}
\label{sec:CY:cones}

In the rest of the paper we will work with asymptotically conical (AC) Calabi--Yau $3$-folds, 
as just defined in Definition \ref{def:AC:CY}. In Appendix \ref{Appendix:Analysic:AC} we have collected various facts about AC manifolds and the requisite analytic tools. 
Our goal is to understand natural differential operators such as the Dirac operator, the Dirac-type operator $d+d^\ast$ and the Laplacian $dd^\ast + d^\ast d$ acting on differential forms on AC Calabi--Yau 3-folds. By the results of Appendix \ref{Appendix:Analysic:AC}
the behaviour at infinity of these differential operators on any AC manifold is controlled by the properties of the analogous differential operators 
on the asymptotic Calabi--Yau cone $\tu{C}(\Sigma)$. By separation of variables, understanding the kernels of these operators on a Calabi--Yau cone $\tu{C}(\Sigma)$ is intimately related to properties of the cross-section $\Sigma$, which must be a Sasaki--Einstein $5$-manifold.

\subsection{Sasaki--Einstein $5$-manifolds}\label{sec:Sasaki:Einstein}

In this section we collect the results about Sasaki--Einstein $5$-manifolds that are necessary to understand the asymptotic behaviour of sections in the kernel of various differential operators on $3$-dimensional Calabi--Yau cones. As for many geometric structures arising from holonomy reduction, there are two main approaches to describe Sasaki--Einstein structures, either via differential forms or spinors. We will recall both of them here, mostly following the monograph \cite{Boyer:Galicki}, the survey paper \cite{Sparks:SE} and \cite[\S 4.3]{Friedrich:al}.

First of all a Sasaki--Einstein structure on a $5$-manifold $\Sigma$ is a special type of \sutwostr: this is a $4$-tuple $(\eta, \omega_1,\omega_2,\omega_3)$ of differential forms on $\Sigma$ satisfying the following algebraic constraints. $\eta$ is a nowhere vanishing $1$-form and therefore defines a codimension $1$ distribution $\ker{\eta}$ that we will denote by $\mathcal{H}$. $(\omega_1,\omega_2,\omega_3)$ is a triple of $2$-forms that span at every point a definite subspace of $\Lambda^2\mathcal{H}^\ast$: $\eta\wedge \omega_1^2\neq 0$ and
\[
\omega_i\wedge\omega_j = \delta_{ij}\,\omega_1^2
\]
for $i,j=1,2,3$. Since for any oriented $4$-dimensional vector space $V$ the $3$-dimensional space $\Lambda^+V^\ast$ of self-dual $2$-forms on $V$ is naturally oriented, it makes sense to require  further that $(\omega_1,\omega_2,\omega_3)$ is an oriented basis of the subspace of $\Lambda^2\mathcal{H}^\ast$ they span. Since $\sunitary{2}\subset\sorth{4}\subset\sorth{5}$ every {\sutwostr} induces a Riemannian metric $g_\Sigma$.

The {\sutwostr} $(\eta,\omega_1,\omega_2,\omega_3)$ is called \emph{Sasaki--Einstein} if
\begin{equation}\label{eq:Sasaki:Einstein}
d\eta = 2\omega_1, \qquad d\omega_2 = -3\eta\wedge\omega_3, \qquad d\omega_3 = 3\eta\wedge\omega_2.
\end{equation}
These equations are equivalent to the fact that the conical {\suthreestr} on $\tu{C}(\Sigma)$ defined by
\begin{equation}\label{eq:conical:CY}
\omega_\tu{C} = rdr\wedge\eta + r^2\omega_1, \qquad \Omega_\tu{C} = r^2(dr+ir\eta)\wedge (\omega_2 + i\omega_3)
\end{equation}
is torsion-free, \ie $\omega_\tu{C}$ and $\Omega_\tu{C}$ are both closed, and therefore defines a conical Calabi--Yau structure on $\tu{C}(\Sigma)$. Since every Calabi--Yau manifold is Ricci-flat, we immediately deduce that every Sasaki--Einstein $5$-manifold is Einstein with positive scalar curvature $\text{Scal}(g_\Sigma)=20$. In particular, complete Sasaki--Einstein $5$-manifolds are compact with finite fundamental group.

We now recall a few basic algebraic facts about differential forms on a $5$-manifold $\Sigma$ endowed with an {\sutwostr} $(\eta,\omega_1,\omega_2,\omega_3)$. We will refer to the vector field $\xi$ dual to $\eta$ via the metric $g_\Sigma$ as the \emph{Reeb vector field}. Note that $\xi$ is a unit length Killing field. The tangent bundle of $\Sigma$ splits as $\R\xi\oplus\mathcal{H}$. The distribution $\mathcal{H}$ inherits a triple of almost Hermitian structures given by restriction of the metric 
and the triple of $2$-forms to it. We denote the corresponding triple of transverse almost complex structures on $\mathcal{H}$ by $J_1, J_2, J_3$. Observe that $J_1, J_2, J_3$ satisfy the standard quaternionic relations, \ie $J_i J_j=-J_j J_i = J_k$ for $(ijk)$ any cyclic permutation of $(123)$. We extend $J_1,J_2,J_3$ to a triple of endomorphisms of $T\Sigma$ by setting $J_i \xi=0$.

The following lemma collects basic facts about the Hodge star operator and the decomposition of the space of differential forms on a $5$-manifold $\Sigma$ endowed with an {\sutwostr}. They can be proven easily by choosing a coframe on $\Sigma$ adapted to the {\sutwostr}.
\begin{lemma}\label{lem:Forms:Sasaki:Einstein}
Let $\Sigma$ be a $5$-manifold endowed with an {\sutwostr} $(\eta, \omega_1,\omega_2,\omega_3)$.
\begin{enumerate}
\item The volume form of the metric $g_\Sigma$ induced by the {\sutwostr} is $\dvol_\Sigma = \tfrac{1}{2}\eta\wedge\omega_1^2$.
\item $\ast \eta = \tfrac{1}{2}\omega_1^2$ and $\ast \gamma = -J_i\gamma \wedge \eta \wedge\omega_i$ for every $\gamma\in \mathcal{H}^\ast$ and $i=1,2,3$.
\item $\Lambda^2T^\ast\Sigma = \R\omega_1 \oplus \R\omega_2 \oplus \R \omega_3 \oplus \Lambda^{1,1}_0 \mathcal{H}^\ast\oplus \mathcal{H}^\ast\wedge\eta$, where $\Lambda^{1,1}_0\mathcal{H}^\ast$ denotes the space of $2$-forms on $\mathcal{H}$ which are primitive of type $(1,1)$ with respect to the almost Hermitian structure $(\omega_1,J_1)$.
\item For $i=1,2,3$ $\ast\omega_i = \eta\wedge\omega_i$, $\ast (\gamma\wedge\eta)=J_i\gamma\wedge\omega_i$ for every $\gamma\in\mathcal{H}^\ast$ and $\ast\sigma = -\eta\wedge\sigma$ for every $\sigma\in \Lambda^{1,1}_0 \mathcal{H}^\ast$.
\end{enumerate}
\end{lemma}

We now move on to describe Sasaki--Einstein $5$-manifolds from the spinorial point of view. Recall that $\spin{5}\simeq \tu{Sp}(2)$, the spin representation of $\spin{5}$ is isomorphic to the standard representation $\HH^2$ of $\tu{Sp}(2)$ and $\sunitary{2}\simeq\tu{Sp}(1)$ is precisely the stabiliser of a non-zero vector. Thus every $5$-manifold endowed with an {\sutwostr} is spin and endowed with a nowhere-vanishing spinor. Conversely, the choice of a nowhere-vanishing spinor $\psi$ on $\Sigma^5$ defines a reduction of the structure group of the tangent bundle of $\Sigma$ from $\sorth{5}$ to $\sunitary{2}$. The condition \eqref{eq:Sasaki:Einstein} that an {\sutwostr} must satisfy in order to define a Sasaki--Einstein structure has an equivalent reformulation as an equation satisfied by the defining spinor $\psi$:
\begin{equation}\label{eq:Killing:spinor}
\nabla _X\psi = \alpha X\cdot\psi
\end{equation}
with $\alpha =\tfrac{1}{2}$ (up to a choice of orientation). This equation, the \emph{real Killing spinor} equation,  is equivalent to the fact that the radial extension of $\psi$ to $\tu{C}(\Sigma)$ defines a parallel spinor.

We now recall further details of the algebraic theory of spinors in dimension $5$, following \cite[\S 4.3]{Friedrich:al} and \cite[\S 4.4]{Moroianu:Semmelmann:generalised:Killing}, and relate the spinorial presentation to our previous definition in terms of differential forms. Given a nowhere-vanishing spinor $\psi$ (which we assume satisfies \eqref{eq:Killing:spinor} with $\alpha =\tfrac{1}{2}$ in the Sasaki--Einstein case) the spin representation $\slashed{S}$ decomposes orthogonally in terms of irreducible $\sunitary{2}$--representations
\[
\slashed{S} = \R \psi \oplus \R I\psi \oplus \R J\psi \oplus \R K\psi \oplus \mathcal{H}\cdot \psi,
\]
where $I,J,K$ define the standard quaternionic structure of $\HH ^2$ and the last factor is the space of spinors of the form $X\cdot \psi$ for a vector $X\in\mathcal{H}=\ker\eta$. In fact $I\psi = \xi\cdot\psi$, where $\xi$ is the Reeb vector field. Furthermore, $I$ commutes with Clifford multiplication by tangent vectors and $J, K$ anticommute with it. Finally, the transverse almost complex structures $J_1,J_2,J_3$ can be defined by
\[
X\cdot I\psi = J_1 X\cdot\psi, \qquad X\cdot J\psi = J_2 X\cdot\psi,\qquad X\cdot K\psi = J_3 X\cdot\psi
\]
for every $X\in\mathcal{H}$.

\begin{remark*}
The expression \eqref{eq:conical:CY} for the holomorphic volume form $\Omega_\tu{C}$ on the cone $\tu{C}(\Sigma)$ in terms of the Sasaki--Einstein structure on $\Sigma$ implies that the complex structure $J_\tu{C}$ on $\tu{C}(\Sigma)$ acts by $J_\tu{C}\,\partial_r = \frac{1}{r}\xi$ and $J_\tu{C}\,|_{\mathcal{H}}=J_1$.
\end{remark*}

We will also need to consider the action of $2$--forms on the spinor $\psi$. Decomposing a $2$--form $\sigma = f_1 \omega_1 + f_2 \omega_2 + f_3 \omega_3 + \sigma_0 +\gamma\wedge\eta$ as in Lemma \ref{lem:Forms:Sasaki:Einstein} (iii) and choosing an orthonormal basis of $T\Sigma$ adapted to the {\sutwostr} we find
\begin{equation}\label{eq:Clifford:2:forms}
\sigma\cdot\psi = -2 f_1 I\psi -2f_2 J\psi -2f_3 K\psi + \bigl( J_1 \gamma^\sharp \bigr) \cdot\psi. 
\end{equation}

We have the following result about the number of linearly independent real Killing spinors on Sasaki--Einstein $5$-manifolds \cite{Friedrich:Kath}.

\begin{prop}\label{prop:Killing:spinor}
Let $\Sigma$ be a Sasaki--Einstein $5$-manifold and $\psi$ a unit spinor satisfying \eqref{eq:Killing:spinor} with $\alpha=\tfrac{1}{2}$. Then $I\psi$ also satisfies \eqref{eq:Killing:spinor} with $\alpha =\tfrac{1}{2}$ while $J\psi$ and $K\psi$ are solutions of \eqref{eq:Killing:spinor} with $\alpha = -\tfrac{1}{2}$. Moreover, if the universal cover of $\Sigma$ is not isometric to the round $5$-sphere then there are no solutions to \eqref{eq:Killing:spinor} with $\alpha = \tfrac{1}{2}$ other than constant complex multiples of $\psi$ and no solutions with $\alpha=-\tfrac{1}{2}$ other than constant complex multiples of $J\psi$.
\end{prop}

Finally, we deduce some useful properties of Killing vector fields on Sasaki--Einstein $5$-manifolds.

\begin{lemma}\label{lem:Killing:vectors:SE}
Let $\Sigma$ be a Sasaki--Einstein $5$-manifold whose universal cover is not isometric to the round $5$-sphere and let $X$ be a Killing field on $\Sigma$.
\begin{enumerate}
\item $X$ also preserves $\eta$ and $\omega_1$.
\item $X=a\,\xi + X_0$, where $a\in \R$ and $X_0$ preserves the whole {\sutwostr} (or equivalently the Killing spinor $\psi$).
\item The $2$-form $dX^\flat$ satisfies
\[
dX^\flat = \left( 2a -\langle X_0,\xi\rangle\right)  \omega_1 + 2\,(J_1 X^\flat)\wedge\eta + \sigma_0
\]
for some $\sigma_0\in \Lambda^{1,1}_0\mathcal{H}^\ast$.
\end{enumerate}
\proof
Let $\psi$ be the unit spinor on $\Sigma$ satisfying \eqref{eq:Killing:spinor} with $\alpha=\frac{1}{2}$. We are going to consider the Lie derivative of $\psi$ in the direction of $X$. Note that since $X$ is a Killing vector field there are no subtleties in defining this Lie derivative. 

Since $X$ is a Killing vector field, $\mathcal{L}_X\psi$ must be a spinor orthogonal to $\psi$ which solves \eqref{eq:Killing:spinor} with $\alpha=\tfrac{1}{2}$. Because of our assumption on $\Sigma$, Proposition \ref{prop:Killing:spinor} then implies that $\mathcal{L}_X\psi = \lambda I\psi$ and $\mathcal{L}_X (I\psi) = -\lambda \psi$ for some constant $\lambda\in\R$. Since $I\psi=\xi\cdot\psi$ we immediately deduce that $[\xi,X]=0$ and therefore $\mathcal{L}_X\eta=0$. Since $2\omega_1 = d\eta$ we also have $\mathcal{L}_X\omega_1=0$. This is the classical result \cite[Corollary 8.1.19]{Boyer:Galicki} that Killing fields on a Sasaki--Einstein manifold with non-constant curvature must preserve the Sasaki structure.

Now, recall that the (metric) Lie derivative of a spinor is given by $\mathcal{L}_X\psi = \nabla _X\psi - \tfrac{1}{4}dX^\flat \cdot \psi$ \cite[Proposition 17]{Bourguignon:Gauduchon}. Since $\psi$ satisfies \eqref{eq:Killing:spinor} with $\alpha=\tfrac{1}{2}$ we have
\begin{equation}\label{eq:Lie:derivative:spinor}
\mathcal{L}_X\psi = \tfrac{1}{2}X\cdot \psi - \tfrac{1}{4}dX^\flat\cdot\psi.
\end{equation}
In particular, applying this formula in the case where $X = a\,\xi$ for $a\in\R$ and using the facts that $\xi^\flat = \eta$, $d  \eta=2 \omega_1$ and \eqref{eq:Clifford:2:forms} we find that $\mathcal{L}_{a\xi}\psi = \tfrac{3}{2}a I\psi$. Thus if $\mathcal{L}_X\psi=\lambda I\psi$ then $X-\frac{2}{3}\lambda\xi$ is a Killing field that preserves $\psi$. It therefore now suffices to prove (iii) under the assumption that $X$ preserves $\psi$.

Write $X=\langle X,\xi\rangle\, \xi + X^\perp$ with $X^\perp\in\mathcal{H}$ and $dX^\flat = f_1 \omega_1 + f_2\omega_2 + f_3 \omega_3 + \sigma_0 + \gamma\wedge\eta$ with $\sigma_0\in \Lambda^{1,1}_0\mathcal{H}^\ast$ and $\gamma\in\mathcal{H}^\ast$. Using \eqref{eq:Clifford:2:forms} and \eqref{eq:Lie:derivative:spinor} we calculate
\[
2\mathcal{L}_X\psi = \left( \langle X,\xi\rangle + f_1 \right) I\psi + f_2 J\psi + f_3 K\psi + \bigl( X^\perp -\tfrac{1}{2}J_1\gamma^\sharp\bigr) \cdot\psi.
\]
Thus $\mathcal{L}_X\psi=0$ implies $f_1 = -\langle X,\xi\rangle$, $f_2=0=f_3$ and $\gamma = 2\, (J_1 X^\flat)$.  
\endproof
\end{lemma}

\subsubsection{Eigenvalue estimates}

Here we collect some eigenvalue estimates for the differential-form Laplacian on a Sasaki--Einstein $5$-manifold. By Appendix \ref{Appendix:homogenous:harmonic} harmonic forms on the Riemannian cone $\tu{C}(\Sigma)$ over a smooth manifold $\Sigma$ are controlled by the spectrum of the Laplacian on $\Sigma$ itself. Since $\Sigma$ is $5$-dimensional the spectrum of the Laplacian acting on differential forms is controlled by the spectrum of the Laplacian acting on functions, on coclosed $1$-forms and on coclosed $2$-forms. Lower bounds for the first (non-zero) eigenvalue of the Laplacian acting on functions and coclosed $1$-forms are classical and only use the lower bound for the Ricci-curvature of $\Sigma$. We will also prove a lower bound for the first non-zero eigenvalue of the Laplacian acting on coclosed $2$-forms on a \emph{regular} Sasaki--Einstein $5$-manifold, \ie in the case when the orbits of the Reeb vector field are closed and $\Sigma$ is the total space of a circle bundle over a smooth K\"ahler--Einstein del Pezzo surface $D$. Examples show that this lower bound fails if $\Sigma$ is not regular.

Throughout the section and in the rest of the paper we make the following
\begin{assumption} 
The Sasaki--Einstein $5$-manifold $\Sigma$ does not have constant curvature.
\end{assumption}
This assumption, that will suffice for our applications, is convenient for two main reasons. The first is given by Proposition \ref{prop:Killing:spinor}: the only real Killing spinors on $\Sigma$ are quaternionic multiples of $\psi$. The second one is the Lichnerowicz--Obata Theorem, \cf Proposition \ref{prop:Spectrum:Laplacian:Regular:SE} (i) below.

We will denote by $\mathcal{K}(\Sigma)$ the space of Killing fields on the Sasaki--Einstein manifold $\Sigma$ and by $\mathcal{K}_0(\Sigma)$ the subspace of those Killing fields that preserve the Killing spinor $\psi$. By Lemma \ref{lem:Killing:vectors:SE} (ii) $\mathcal{K}(\Sigma) = \mathcal{K}_0(\Sigma) \oplus \R \xi$. 

\begin{prop}\label{prop:Spectrum:Laplacian:Regular:SE}
Let $\Sigma$ be a Sasaki--Einstein $5$-manifold with non-constant curvature.
\begin{enumerate}
\item The first non-zero eigenvalue of the scalar Laplacian on $\Sigma$ is strictly greater than $5$.
\item The first eigenvalue of the Laplacian acting on coclosed $1$-forms is greater than or equal to $8$ and the eigenspace with eigenvalue $8$ consists of $1$-forms dual to Killing vector fields.
\item If $\Sigma$ is a regular Sasaki--Einstein $5$-manifold then the first non-zero eigenvalue of the Laplacian acting on coclosed $2$-forms is strictly greater than $4$.
\end{enumerate}
\proof
Part (i) is the Lichnerowicz--Obata Theorem \cite[\S 77]{Lichnerowicz}, \cite{Obata}, and part (ii) is due to the fact that $\Sigma$ is Einstein with scalar curvature $20$ \cite[\S 77]{Lichnerowicz}.


We now prove part (iii). Let $\tau$ be a coclosed $2$-form on $\Sigma$ such that $\triangle\tau =\mu^2 \tau$ for some $\mu\in (0,2]$. We begin with some remarks that hold on a general Sasaki--Einstein $5$-manifold.

By Theorem \ref{thm:Harmonic:forms:cone:n} (iv), $r^\mu \tau$ is a harmonic $2$-form on $\tu{C}(\Sigma)$ of rate $\mu-2\in (-2,0]$. Decomposing $r^\mu \tau$ into types and using the fact that each component is harmonic, Propositions \ref{prop:Harmonic:Functions:Cone} and \ref{prop:Harmonic:1Forms:Cone} below imply that $r^\mu \tau$ is a primitive $(1,1)$--form when $\mu\in (0,2)$ and the sum of a primitive $(1,1)$--form and a constant multiple of the K\"ahler form $\omega_\tu{C}$ when $\mu=2$.

We now exploit this observation to study the decomposition of $\tau \in \Omega^2 (\Sigma)$ and $d\tau\in\Omega^3(\Sigma)$ in terms of the type decomposition of Lemma \ref{lem:Forms:Sasaki:Einstein} (iii). Straightforward calculations using the definition \eqref{eq:conical:CY} of the conical Calabi--Yau structure $(\omega_\tu{C},\Omega_\tu{C})$ in terms of the Sasaki--Einstein structure $(\eta,\omega_1,\omega_2,\omega_3)$ show that the conditions $r^\mu\tau\wedge\omega_\tu{C}^2 = K \omega_\tu{C}^3$ for some $K\in\R$ ($K=0$ if $\mu\neq 2$) and $r^\mu\tau\wedge\Real\Omega_\tu{C}=0$ imply that
\[
\tau = K\omega_1 + \tau_0, \qquad \ast\tau = K\eta\wedge\omega_1 - \eta\wedge\tau_0
\]    
for some $\tau_0\in\Lambda^{1,1}_0\mathcal{H}^\ast$. Now, since $\tau$ is coclosed, $0=d\ast\tau = 2K\omega_1^2 + \eta\wedge d\tau_0$ implies that $K=0$ and $\tau=\tau_0\in\Lambda^{1,1}_0\mathcal{H}^\ast$ regardless of the value of $\mu\in (0,2]$.

Consider now $d\tau$. Since $d\tau\wedge\omega_i=0$ (by differentiating $\tau\wedge\omega_i=0$) we have $\langle d\tau,\eta\wedge\omega_i\rangle =0$ for all $i=1,2,3$. Thus $d\tau = \gamma\wedge\omega_1 + \mu\, \eta \wedge \tau'$ for some $1$-form $\gamma \in \Lambda^1 \mathcal{H}^\ast$ and $\tau'\in\Lambda^{1,1}_0\mathcal{H}^\ast$. The requirement $0=d\ast\tau=\eta\wedge d\tau$ then forces $\gamma$ to vanish. Using that $d^\ast d\tau=\mu^2\tau$ we further deduce that the pair $(\tau,\tau')$ satisfies the first order system
\[
d\tau = \mu\,\eta \wedge\tau'=-\mu\ast\tau', \qquad d\tau' = -\mu\eta\wedge\tau = \mu\ast \tau.
\]
Here we use Lemma \ref{lem:Forms:Sasaki:Einstein} (iv) and the fact that $\tau$ and $\tau'$ are sections of $\Lambda^{1,1}_0\mathcal{H}^\ast$. We then consider the complex-valued form $\tau^c = \tau + i\tau'\in\Lambda^{1,1}_0\mathcal{H}^\ast\otimes\C$: it satisfies $\triangle\tau^c =\mu^2 \tau^c$ and $\mathcal{L}_\xi \tau^c = -i\mu\tau^c$. We will show that $\tau^c=0$ whenever $\Sigma$ is a regular Sasaki--Einstein $5$-manifold.

Passing to the universal cover if necessary we assume that $\Sigma$ is a simply connected regular Sasaki--Einstein $5$-manifold. Then $\Sigma$ is the total space of the principal $\unitary{1}$--bundle associated with the line bundle $L=K_D^{1/I}$ over a K\"ahler--Einstein del Pezzo surface $D$ with Fano index $I$ and K\"ahler--Einstein metric $\omega_1$ satisfying $\Ric (\omega_1) = 6\omega_1$, \cf for example \cite[Theorem 2.1]{Sparks:SE}. Here the Fano index of $D$ is the divisibility of $K_D$ in the Picard group of $D$. Since $d\eta = 2\omega_1$ we deduce that $V=\frac{I}{3}\xi$ is the vector field of period $2\pi$ generating the fibre-wise circle action on $\Sigma \ra D$. The fact that $\mathcal{L}_\xi \tau^c = -i\mu\tau^c$ then immediately implies that $\mu = \frac{3}{I}m$ for some $m \in \Z_{>0}$. Since $I=1$ for all del Pezzo surfaces except for $\CP^2$ and $\CP^1\times \CP^1$ this already shows that $\mu \geq 3$ in all but these two cases. When $D=\CP^2$ or $\CP^1\times\CP^1$ then $I=3$ and $2$ respectively and we have to argue differently.

Since $\mathcal{L}_{V}\tau^c = -im\,\tau^c$, $\tau^c$ can be interpreted as an $L^m$--valued anti-self-dual $2$-form on $D$. Moreover, following \cite[Lemma 4.4]{Ammann:Bar} the equation $\triangle \tau^c = \mu^2 \tau^c$ implies that $\tau^c$ is a harmonic section of $\Lambda^- T^\ast D \otimes L^m$, where $L^m$ is endowed with the connection $A=i\frac{3m}{I}\eta$. Since $A$ is a self-dual connection, the Weitzenb\"ock formula for $\triangle_A$ coincides with the standard Weitzenb\"ock formula for the Laplacian on anti-self-dual $2$--forms: $0=\triangle_A \tau^c = \nabla_A^\ast\nabla_A \tau^c -2W^-(\tau^c) + \frac{1}{3}\text{Scal}\,\tau^c$. When $D=\CP^2$ the curvature term is strictly positive and therefore $\tau^c=0$. When $D=\CP^1\times\CP^1$ the bundle $\Lambda^-T^\ast D$ has a parallel decomposition into a trivial real line bundle spanned by the difference $\omega_0$ of the Fubini--Study metrics on the two factors and a complex line bundle. The curvature term is strictly positive on the latter and vanishes on the former. We conclude that $\tau^c$ must be an $L^m$--valued parallel multiple of $\omega_0$; since $L^m$ is non-trivial for $m\neq 0$ we deduce that $\tau^c=0$.     
\endproof
\end{prop}

\begin{remark}\label{rmk:Harmonic:2:forms:SE}
The initial part of the proof of part (iii) can be applied to harmonic $2$-forms on $\Sigma$: every such form lies in $\Lambda^{1,1}_0\mathcal{H}^\ast$.
\end{remark}

\subsection{Differential forms on Calabi--Yau cones}

Let $\Sigma$ be a Sasaki--Einstein $5$-manifold and consider the Calabi--Yau cone $\tu{C}(\Sigma)$. The aim of this section is to study elements in the kernel of $d+d^\ast$, the Laplacian acting on differential forms and the Dirac operator on $\tu{C}(\Sigma)$ that are homogeneous with respect to the action of scaling on the cone in the following sense. 

\begin{definition}\label{def:Homogeneous:forms}
A $k$-form $\alpha$ on the cone $\tu{C}(\Sigma)$ is \emph{homogeneous of order} $\lambda$ if
\[
\alpha = r^\lambda \bigl( r^{k-1}dr\wedge \alpha_{k-1} + r^k \alpha_k \bigr)
\]
for some $\alpha_{k-1}\in\Omega^{k-1}(\Sigma)$ and $\alpha_k\in\Omega^k (\Sigma)$.
\end{definition}
Observe that if $\alpha$ is homogeneous of order $\lambda$ then $|\alpha|_{g_\tu{C}}^2=r^{2\lambda} \bigl( |\alpha_{k-1}|^2_{g_\Sigma} + |\alpha_k|^2_{g_\Sigma}\bigr)$.

\begin{remark*}
Throughout the paper harmonic forms are differential forms in the kernel of the Hodge Laplacian $\triangle = dd^\ast + d^\ast d$. Since we work on non-compact manifolds, the space of harmonic forms in general strictly contains the space of closed and coclosed forms.
\end{remark*}

\subsubsection{Harmonic functions and $1$-forms}

The following two propositions follow from Theorem \ref{thm:Harmonic:forms:cone:n} and Proposition \ref{prop:Spectrum:Laplacian:Regular:SE} (i) and (ii).

\begin{prop}\label{prop:Harmonic:Functions:Cone} 
Let $u$ be a harmonic function on $\tu{C}(\Sigma)$ homogeneous of order $\lambda$. Then $u=0$ if $\lambda\in [-5,1]\setminus \{ -4,0\}$ and $u=Kr^\lambda$ for some $K\in\R$ if $\lambda=-4,0$.
\end{prop}

\begin{prop}\label{prop:Harmonic:1Forms:Cone}
Let $\gamma$ be a harmonic $1$-form homogeneous of order $\lambda$. Then
\[
\gamma = \begin{cases}
Krdr + d\bigl( \tfrac{1}{2}r^2\alpha \bigr) + r^2\beta, \text{ where } K\in\R,\, \triangle\alpha=12\alpha \text{ and } \beta^\flat\in\mathcal{K}(\Sigma) & \text{if } \lambda=1,\\
d\bigl( \tfrac{1}{\lambda+1}r^{\lambda+1}\alpha \bigr), \text{ where }\triangle\alpha=(\lambda +5)(\lambda +1)\alpha & \text{if } \lambda \in (0,1),\\
0 & \text{if } \lambda\in [-4,0],\\
r^\lambda \alpha\, dr -\frac{r^{\lambda+1}}{\lambda+3} d\alpha, \text{ where }\triangle\alpha = (\lambda+3)(\lambda-1)\alpha & \text{if } \lambda\in (-5,-4),\\
Kr^{-5}dr + (r^{-5}\alpha dr + \tfrac{1}{2}r^{-4}d\alpha) + r^{-4}\beta, \text{ where } K\in\R,\, \triangle \alpha = 12\alpha \text{ and } \beta^\flat\in\mathcal{K}(\Sigma) & \text{if } \lambda = -5.
\end{cases}
\]
\end{prop}

\begin{remark}\label{rmk:Harmonic:1:forms}
Remark \ref{rmk:Harmonic:forms:cone:n} shows that a harmonic $1$-form $\gamma$ homogeneous of order $\lambda\in [-5,1]$ is coclosed if and only if $\lambda\in [-4,1)$, or $\lambda=1$ and $\gamma$ does not contain a term of the form $Krdr$ for $K\in\R$, or $\lambda=-5$ and $\gamma = Kr^{-5}dr + r^{-4}\beta$ for some $K\in\R$ and $\beta^\flat\in\mathcal{K}(\Sigma)$. We will use this fact in the proof of Proposition \ref{prop:Dirac:cone}.
\end{remark}

We will also have to consider the operators
\begin{equation}\label{eq:Gauge:fixing:operator:AC:CY}
g \longmapsto \pi_1 dd^\ast \bigl( \tfrac{1}{2}g\omega_0^2\bigr), \qquad (f,\gamma)\longmapsto \pi_{1\oplus 6}d^\ast d \bigl( f\omega_0 + \gamma^\sharp\lrcorner\Real\Omega_0\bigr).
\end{equation}

\begin{prop}\label{prop:Gauge:Cone}
For $\lambda\in (-4,0)$ there are no elements homogeneous of order $\lambda$ in the kernel of the operators in \eqref{eq:Gauge:fixing:operator:AC:CY}.
\proof
By Lemma \ref{lem:Gauge:fixing:operator:AC:CY} we have to study the existence of homogeneous functions and $1$-forms on the Calabi--Yau cone $\tu{C}(\Sigma)$ in the kernel of $\triangle$ and $\triangle - \tfrac{1}{3}d^\ast d$ respectively. Proposition \ref{prop:Harmonic:Functions:Cone} already shows that there are no harmonic functions on $\tu{C}(\Sigma)$ homogeneous of order $\lambda\in (-4,0)$.

For $\beta_0\in\Omega^0(\Sigma)$ and $\beta_1\in\Omega^1(\Sigma)$ let then $\beta = r^\lambda (\beta_0\, dr + r\beta_1)$ be a homogeneous $1$-form on $\tu{C}(\Sigma)$. We calculate that $(\triangle - \tfrac{1}{3}d^\ast d)\beta=0$ if and only if
\begin{equation}\label{eq:Gauge:Cone}
\begin{cases}
\tfrac{2}{3}d^\ast d \beta_0 = (\lambda-1)(\lambda+5)\beta_0 -\tfrac{1}{3}(\lambda-5)d^\ast \beta_1,\\
(dd^\ast + \tfrac{2}{3}d^\ast d) \beta_1= \tfrac{2}{3}(\lambda+1)(\lambda+3)\beta_1 + \tfrac{1}{3}(\lambda+9)d\beta_0.
\end{cases}
\end{equation}
Note that integration over $\Sigma$ of the first equation of \eqref{eq:Gauge:Cone} already shows that $\beta_0$ has mean value zero whenever $\lambda\neq -5,1$.

Now, algebraic manipulations of \eqref{eq:Gauge:Cone} imply that
\[
\begin{gathered}
(\lambda-5)d^\ast\beta_1 -(\lambda-1)(\lambda+9)\beta_0 \text{ is an eigenfunction of eigenvalue } (\lambda+1)(\lambda+5),\\
d^\ast\beta_1 -(\lambda+5)\beta_0 \text{ is an eigenfunction of eigenvalue }(\lambda-1)(\lambda+3).
\end{gathered}
\]

If $\lambda\in [-4,0]$ then both $(\lambda+1)(\lambda+5)\leq 5$ and $(\lambda-1)(\lambda+3)\leq 5$. $L^2$--orthogonality of $d^\ast\beta_1$ and $\beta_0$ to constant functions and Proposition \ref{prop:Spectrum:Laplacian:Regular:SE} (i) imply
\[
(\lambda-5)d^\ast\beta_1 =(\lambda-1)(\lambda+9)\beta_0=(\lambda+5)(\lambda-5)\beta_0.
\]
Thus $\beta_0=0=d^\ast\beta_1$ unless $\lambda=-2$. In the latter case however $d^\ast\beta_1=3\beta_0$ and the first equation in \eqref{eq:Gauge:Cone} becomes $\triangle\beta_0 + 3\beta_0=0$. Thus $\beta_0=0=d^\ast\beta_1$ for all $\lambda\in [-4,0]$. The second equation of \eqref{eq:Gauge:Cone} then implies that $\beta_1$ is a coclosed eigenform of the Laplacian with eigenvalue $(\lambda+1)(\lambda+3)$. Since $(\lambda+1)(\lambda+3) <8$ for $\lambda\in (-5,1)$ Proposition \ref{prop:Spectrum:Laplacian:Regular:SE} (ii) forces $\beta_1$ to vanish.
\endproof
\end{prop}

\subsubsection{The Dirac operator}

Recall that by Lemma \ref{lem:Dirac:cone} the Dirac operator of the cone $\tu{C}(\Sigma)$ is identified with the operator
\[
\slashed{D}(f,g,\gamma) = (d^\ast\gamma, -d^\ast \!J\gamma, \tu{curl}\,\gamma + df-Jdg)
\]
acting on $\Omega^0\oplus\Omega^0\oplus\Omega^1$.

\begin{prop}\label{prop:Dirac:cone}
Let $(f,g,\gamma)$ be a harmonic spinor on the Calabi--Yau cone $\tu{C}(\Sigma)$, where $f,g,\gamma$ are homogeneous of order $\lambda$. Then
\[
(f,g,\gamma) = \begin{cases}
\bigl( 0,0,r^2\beta + d\bigl( \tfrac{1}{2}r^2\alpha\bigr) \bigr) \text{ where }\beta^\flat\in\mathcal{K}_0(\Sigma) \text{ and }\triangle\alpha=12\alpha & \text{if } \lambda=1,\\
\bigl( 0,0, d\bigl( \tfrac{1}{\lambda+1}r^{\lambda+1}\alpha\bigr)\bigr) \text{ where }\triangle\alpha=(\lambda+5)(\lambda+1)\alpha & \text{if } \lambda\in (0,1),\\
\left( K_1,K_2,0\right) \text{ where }K_1,K_2\in\R & \text{if } \lambda=0,\\
\left( 0,0,0\right) & \text{if } \lambda \in (-5,0),\\
\left( 0,0,K_1 r^{-5}dr + K_2 r^{-4}\eta \right) \text{ where }K_1,K_2\in\R & \text{if } \lambda=-5.
\end{cases}
\]
\proof
Since $\tu{C}(\Sigma)$ is Ricci-flat, by the Lichnerowicz Formula $\slashed{D}(f,g,\gamma)=0$ implies that $f,g,\gamma$ are all harmonic. We can therefore appeal to Propositions \ref{prop:Harmonic:Functions:Cone} and \ref{prop:Harmonic:1Forms:Cone} to deduce that $f=0=g$ if $\lambda\in [-5,1]\setminus\{ -4,0\}$ and $\gamma=0$ if $\lambda\in [-4,0]$. Furthermore, a straightforward computation shows that $\slashed{D}(K_1 r^{-4},K_2 r^{-4},0)=0$ implies $K_1=0=K_2$. This already proves the statement for all $\lambda\in [-4,0]$. Moreover, for $\lambda\in[-5,-4) \cup (0,1]$ $f=0=g$ and by Lemma \ref{lem:Dirac:cone} we have to understand whether a harmonic $1$-form $\gamma$ homogeneous of order $\lambda$ is also coclosed and such that $d\gamma$ is a primitive $(1,1)$--form. 

When $\lambda\in (0,1)$, $\gamma$ is exact by Proposition \ref{prop:Harmonic:1Forms:Cone}. By Remark \ref{rmk:Harmonic:1:forms}, $\gamma$ is also coclosed and therefore it defines an element in the kernel of $\slashed{D}$.

When $\lambda=1$, Proposition \ref{prop:Harmonic:1Forms:Cone} implies that $\gamma = Krdr + d\left( \tfrac{r^2}{2}\alpha\right) + r^2\beta$ for some $K\in\R$, a function $\alpha$ such that $\triangle\alpha=12\alpha$ and a $1$-form $\beta$ dual to a Killing field. By Remark \ref{rmk:Harmonic:1:forms}, $\gamma$ is coclosed if and only if $K=0$. Thus $\gamma$ defines an element in the kernel of $\slashed{D}$ if and only if $d(r^2\beta)= 2r dr\wedge \beta + r^2 d\beta$ is a primitive $(1,1)$--form. Since $\omega_\tu{C} = rdr\wedge\eta + r^2\omega_1$ and $\Real\Omega_\tu{C} = r^2 dr\wedge\omega_2 - r^3\eta\wedge\omega_3$, by the characterisation of primitive $(1,1)$--forms in Lemma \ref{lem:identities:2forms} we have that $d(r^2\beta)$ is primitive if and only if
\[
d\beta \wedge\eta\wedge\omega_1 +\beta\wedge\omega_1^2=0,
\]
(together with a second equation which is $d$ of this one) while $d(r^2\beta)$ is of type $(1,1)$ if and only if 
\[
d\beta\wedge\eta\wedge\omega_3=0, \qquad -2\beta\wedge\eta\wedge\omega_3 +d\beta\wedge\omega_2=0. 
\]

Now, by Lemma \ref{lem:Killing:vectors:SE} (iii), $d\beta = f_1 \omega_1 + 2(J_1 \beta)\wedge\eta + \sigma_0$, with $\sigma_0\in\Lambda^{1,1}_0\mathcal{H}^\ast$. Then we find
\[
d\beta\wedge\omega_i = \begin{cases}
f_1\omega_1^2 + 2J_1\beta\wedge\eta\wedge\omega_1 & \mbox{ if }i=1,\\
2J_1\beta\wedge\eta\wedge\omega_i & \mbox{ otherwise},
\end{cases}
\qquad \mbox{ and }\qquad
d\beta\wedge\eta\wedge\omega_i = \begin{cases}
f_1\eta\wedge\omega_1^2 & \mbox{ if }i=1,\\
0 & \mbox{ otherwise}.
\end{cases}
\]
Moreover
\begin{equation}\label{eq:Dirac:cone:proof}
2\beta\wedge\eta\wedge\omega_3 = -2(J_3^2\beta)\wedge\eta\wedge\omega_3 =    -2(J_2 J_3\beta)\wedge\eta\wedge\omega_2 = 2(J_1\beta)\wedge\eta\wedge\omega_2,
\end{equation}
where we used Lemma \ref{lem:Forms:Sasaki:Einstein} (ii) in the second equality and the identity $J_2 J_3=-J_1$ (as endomorphisms of $\mathcal{H}^\ast$) in the final one.

Using these facts we see that $d(r^2\beta)$ is always of type $(1,1)$, while the requirement that $d(r^2\beta)$ be primitive implies that $f_1 = -\beta(\xi)$. Lemma \ref{lem:Killing:vectors:SE} then shows that $\beta^\flat \in \mathcal{K}_0(\Sigma)$, \ie $\beta$ is dual to a Killing field that also preserves the Killing spinor on $\Sigma$.

By Proposition \ref{prop:Harmonic:1Forms:Cone} and Remark \ref{rmk:Harmonic:1:forms} $\gamma$ is never coclosed if $\lambda\in (-5,0]$. When $\lambda=-5$ then $\gamma$ is coclosed if and only if $\gamma = Kr^{-5}dr +r^{-4}\beta$ for some $K\in\R$ and a $1$-form $\beta$ dual to a Killing field. Thus $\gamma$ defines an element in the kernel of $
\slashed{D}$ if and only if $d(r^{-4}\beta) = -4r^{-5}dr\wedge\beta + r^{-4}d\beta$ is a primitive $(1,1)$--form. Note that $d(r^{-4}\beta) \wedge\Real\Omega_C=0$ if and only if $4\beta\wedge\eta\wedge\omega_3 +d\beta\wedge\omega_2=0$ and $d\beta\wedge\eta\wedge\omega_3=0$. As we saw above the second constraint is satisfied for any $1$-form $\beta$ dual to a Killing field. However \eqref{eq:Dirac:cone:proof} shows that $4\beta\wedge\eta\wedge\omega_3 +d\beta\wedge\omega_2=0$ if and only if $\beta\wedge\eta\wedge\omega_3=0$, \ie $\beta=f\eta$ for some function $f$. Since $\beta$ is dual to a Killing field and therefore in particular $d^\ast\beta=0=\mathcal{L}_{\beta^\sharp}\eta$, we deduce that $f$ must be constant. Finally, a straightforward computation shows that $d(r^{-4}\eta)$ is also primitive.
\endproof
\end{prop}

\begin{remark*}
As an aside we note that the fact that the kernel of $\slashed{D}$ is invariant under the complex structure $J$ on $\tu{C}(\Sigma)$ combined with the characterisation of harmonic spinors on $\tu{C}(\Sigma)$ homogeneous of order $1$ has the following corollary, \cf \cite[Theorem 5.1]{Futaki:Ono:Wang} and \cite[Theorem 2.14]{Hein:Sun}: the space $\mathcal{K}_0(\Sigma)$ of Killing fields that preserve the Killing spinor on $\Sigma$ is isomorphic to the space of basic functions on $\Sigma$ which are eigenfunctions of the Laplacian of eigenvalue $12$. The isomorphism is given by the map $f \longmapsto (f\eta - \tfrac{1}{2}J_1df)^\sharp$. Here a function $f$ on $\Sigma$ is basic if and only if $\mathcal{L}_\xi f=0$. In fact, by \cite[Theorem 2.14 (2)]{Hein:Sun} $12$ is the first non-zero eigenvalue of the Laplacian acting on basic functions on $\Sigma$; this is the extension of the Lichnerowicz--Matsushima Theorem for K\"ahler--Einstein del Pezzo surfaces to Sasaki--Einstein $5$--manifolds.
\end{remark*}

\subsubsection{Closed and coclosed forms}

In order to understand indicial roots of the first-order operator $d+d^\ast$, we will use the following two propositions about homogeneous harmonic $2$-forms and $3$-forms; both propositions follow immediately from Theorem \ref{thm:Harmonic:forms:cone:n}, the fact that $b_1(\Sigma)=0$ and the eigenvalue estimates given in Proposition \ref{prop:Spectrum:Laplacian:Regular:SE}.

\begin{prop}\label{prop:Harmonic:2forms:cone}
Let $\gamma$ be a harmonic $2$-form on $\tu{C}(\Sigma)$ homogeneous of order $\lambda$. Decompose $\gamma = \gamma_1+\gamma_2+\gamma_3+\gamma_4$ as in Theorem \ref{thm:Harmonic:forms:cone:n}.
\begin{enumerate}
\item $\gamma_1=r^{\lambda+1}dr\wedge df$ where $f$ is a function on $\Sigma$ satisfying $\triangle f=\lambda(\lambda+4)f$. In particular, $\gamma_1=0$ for all $\lambda\in [-5,1]$.
\item $\gamma_2=0$ for all $\lambda\in (-6,0)$. If $\lambda=-6$ or $\lambda=0$ then $\gamma_2 = d\left( \tfrac{1}{\lambda+2}r^{\lambda+2}\alpha\right)$, where $\alpha^\flat\in\mathcal{K}(\Sigma)$.
\item $\gamma_3=0$ for all $\lambda\in (-4,2)$. If $\lambda=-4$ or $\lambda=2$ then $\gamma_3 = r^{\lambda+1}dr\wedge\alpha - \tfrac{1}{\lambda+2}r^{\lambda+2}d\alpha$, where $\alpha^\flat\in\mathcal{K}(\Sigma)$.
\item $\gamma_4= r^{\lambda+2}\beta$ for a coclosed $2$-form $\beta$ on $\Sigma$ satisfying $\triangle\beta = (\lambda+2)^2\beta$. In particular, if $\Sigma$ is regular then $\gamma_4=0$ for all $\lambda\in [-4,0]\setminus\{ -2\}$.
\end{enumerate}
\end{prop}

\begin{remark}\label{rmk:Log:Harmonic:2forms:Cone}
We should also count harmonic $2$-forms on $\tu{C}(\Sigma)$ which are polynomials in $\log{r}$ with coefficients given by $2$-forms homogeneous of order $\lambda$. By Proposition \ref{prop:log:homogeneous:harmonic} we only need to consider the case $\lambda=-2$. In this case, for every harmonic $2$-form $\tau$ on $\Sigma$, the forms $\tau\,\log{r}$ and $\tau$ are harmonic on $\tu{C}(\Sigma)$.
\end{remark}

\begin{prop}\label{prop:Harmonic:3forms:Cone}
Let $\gamma$ be a harmonic $3$-form on $\tu{C}(\Sigma)$ homogeneous of order $\lambda$. Decompose $\gamma = \gamma_1+\gamma_2+\gamma_3+\gamma_4$ as in Theorem \ref{thm:Harmonic:forms:cone:n}.
\begin{enumerate}
\item $\gamma_1=0$ for all $\lambda\in (-5,1)\setminus\{-3,-1\}$. If $\lambda=-1$ or $\lambda=-3$ then $\gamma_1 = r^{\lambda+2}dr\wedge \alpha$ where $\alpha$ is a harmonic $2$-form. If $\lambda=1$ or $\lambda=-5$ then $\gamma_1 = r^{\lambda+2}dr\wedge d\beta$ for some $\beta^\flat\in\mathcal{K}(\Sigma)$.
\item $\gamma_2 = d\left( \tfrac{1}{\lambda+3}r^{\lambda+3}\alpha\right)$, where $\alpha$ is a coexact $2$-form on $\Sigma$ satisfying $\triangle\alpha = (\lambda+3)^2\alpha$. In particular, if $\Sigma$ is regular then $\gamma_2=0$ for all $\lambda\in [-5,-1]$.
\item $\gamma_3 = r^{\lambda+2}dr\wedge\alpha - \tfrac{1}{\lambda+1}d\alpha$, where $\alpha$ is a coexact $2$-form on $\Sigma$ satisfying $\triangle\alpha = (\lambda+1)^2\alpha$. In particular, if $\Sigma$ is regular then $\gamma_3=0$ for all $\lambda\in [-3,1]$.
\item $\ast\gamma_4$ satisfies the same conditions as $\gamma_1$.
\end{enumerate}
\end{prop}

\begin{remark}\label{rmk:Log:Harmonic:3forms:Cone}
In the regular case, Proposition \ref{prop:log:homogeneous:harmonic} shows that there are no harmonic $3$-forms on the cone which are non-trivial polynomials in $\log{r}$ with coefficients in the space of homogeneous forms of order $\lambda$.
\end{remark}

We will now use these results to study indicial roots of the first-order operator $d+d^\ast$. We are particularly interested in understanding closed and coclosed even-degree forms of rate $-2$ and odd-degree forms of rate $-3$.

\begin{prop}\label{prop:Closed:Even:Forms}
Let $\gamma$ be a closed and coclosed form of even degree on $\tu{C}(\Sigma)$ homogeneous of rate $\lambda=-2$. Then $\gamma$ has only components of pure degree $2$ and $4$, both individually closed and coclosed:
\[
\gamma = \tau_1 + rdr\wedge \eta\wedge\tau_2
\]
for harmonic $2$-forms $\tau_1,\tau_2$ on $\Sigma$. Moreover, if $\Sigma$ is a regular Sasaki--Einstein $5$-manifold then there are no closed and coclosed even-degree forms on $\tu{C}(\Sigma)$ homogeneous of rate $\lambda\in (-4,0) \setminus\{ -2\}$. 
\proof
Write $\gamma = \gamma_0 + \gamma_2 +\gamma_4 + \gamma_6$, with $\gamma_k \in \Omega^k$. Since each pure degree component is harmonic, Proposition \ref{prop:Harmonic:Functions:Cone} shows that $\gamma_0=0=\gamma_6$ whenever $\lambda\in (-4,0)$.

Writing $\gamma _k = r^\lambda (r^{k-1}dr\wedge\alpha_{k-1} + r^k \beta_k)$ we are then left to solve
\begin{align*}
&d^\ast\alpha_1 =0, & &d^\ast\beta_2 -(\lambda+4)\alpha_1=0,\\
&d\alpha_1 + d^\ast\alpha_3 - (\lambda+2)\beta_2 =0, & &d\beta_2 + d^\ast\beta_4 - (\lambda+2)\alpha_3 =0,\\
&d\alpha_3 - (\lambda+4)\beta_4 =0, & &d\beta_4=0.
\end{align*}
If $\lambda=-2$ then $\alpha_1$ and $\beta_4$ are both closed and coclosed and therefore vanish since $H^1(\Sigma) =0=H^4(\Sigma)$. Then $\beta_2$ and $\alpha_3$ also are closed and coclosed.

If $\Sigma$ is a regular Sasaki--Einstein $5$-manifold then Proposition \ref{prop:Harmonic:2forms:cone} shows that there are no harmonic forms of degree $2$ and $4$ of rate $\lambda \in (-4,0)\setminus\{ -2\}$.  
\endproof
\end{prop}

When $\Sigma$ is not necessarily assumed to be regular we can still exclude some decay rates for homogeneous closed and coclosed $2$-forms on $\tu{C}(\Sigma)$ using Remark \ref{rmk:Harmonic:forms:cone:n}.

\begin{prop}\label{prop:Closed:Even:Forms:2}
There are no closed and coclosed $2$-forms on $\tu{C}(\Sigma)$ homogeneous of rate $\lambda\in (-6,0)\setminus \{ -2\}$.\proof
By Remark \ref{rmk:Harmonic:forms:cone:n} if $\gamma$ is closed and coclosed then $\gamma=\gamma_1+\gamma_2+\gamma_4$. Moreover, $\gamma_1=0$ since $b_1(\Sigma)=0$ and $\gamma_4=0$ unless $\lambda=-2$. Proposition \ref{prop:Harmonic:2forms:cone} now concludes the proof.
\endproof 
\end{prop}
 
\begin{prop}\label{prop:Closed:Odd:Forms}
Let $\gamma$ be a closed and coclosed form of odd degree on $\tu{C}(\Sigma)$ homogeneous of rate $\lambda\in [-4,0]$. Then $\gamma$ is of pure degree $3$. Moreover, if $\lambda=-3$ then
\[
\gamma = \eta\wedge\tau_1 + \frac{dr}{r}\wedge\tau_2
\]
for harmonic $2$-forms $\tau_1,\tau_2$ on $\Sigma$ and if $\lambda\in [-4,0]\setminus \{ -3\}$ then $\gamma = d\bigl( \frac{r^{\lambda+3}}{\lambda+3}\alpha\bigr)$, where $\alpha$ is a coclosed $2$-form on $\Sigma$ satisfying $\triangle\alpha=(\lambda+3)^2\alpha$. In particular, if $\Sigma$ is a regular Sasaki--Einstein $5$-manifold then there are no closed and coclosed odd-degree forms on $\tu{C}(\Sigma)$ homogeneous of rate $\lambda\in [-4,-1]\setminus\{ -3\}$.
\proof
Write $\gamma = \gamma_1 + \gamma_3 +\gamma_5$, with $\gamma_k \in \Omega^k$. Since each pure degree component is harmonic, Proposition \ref{prop:Harmonic:1Forms:Cone} shows that $\gamma_1=0=\gamma_5$ whenever $\lambda\in [-4,0]$. Since $\gamma$ has pure degree $3$, the result now follows combining Remark \ref{rmk:Harmonic:forms:cone:n} with Proposition \ref{prop:Harmonic:3forms:Cone}.
\endproof
\end{prop}

\section{Differential forms on AC Calabi--Yau $3$-folds}
\label{sec:forms:ac:cy}

In this section we study differential forms on an AC Calabi--Yau $3$-fold $(B,\omega_0,\Omega_0)$. We are interested in describing the space of closed and coclosed forms on $B$, in particular $2$-forms, of prescribed decay at infinity. We will also derive ``normal forms'' for exact $4$-forms on $B$, \ie write every exact $4$-form $\sigma$ on $B$ with prescribed decay as $\sigma = d\rho$ for some appropriate choice of $3$-form $\rho$. These results will be used in the next two sections to study the linearisation of the Apostolov--Salamon equations.

\subsection{$L^2$--cohomology}

The starting point of our discussion is the description of the $L^2$--cohomology of $B$, \ie the spaces
\[
L^2\mathcal{H}^k(B) = \{ \sigma \in \Omega^k(B)\cap L^2\, |\, d\sigma =0=d^\ast\sigma\}.
\]
The description of the $L^2$--cohomology of AC manifolds in terms of topological data is well known, \cf \cite[Example 0.15]{Lockhart} and \cite[Theorem 1A]{HHM}. Here we give the statement only in the $6$-dimensional case relevant to us.

Regard $B$ as a manifold with boundary $\Sigma$. Then we have the long exact sequence in cohomology (with real coefficients) for the pair $(B,\Sigma)$:
\begin{equation}\label{eq:Exact:Sequence:Cohomology:Boundary}
\cdots \ra H^{k-1}(\Sigma)\ra H^k_c(B)\ra H^k(B)\ra H^k(\Sigma)\ra \cdots
\end{equation}

\begin{theorem}\label{thm:L2:coho}
Let $(B,\omega_0,\Omega_0)$ be an AC Calabi--Yau $3$-fold. Then there is a natural isomorphism
\[
L^2\mathcal{H}^k(B)\simeq\begin{dcases}
H^k_c(B) & \mbox{if } k=0,1,2\\
\tu{im}\, H^3_c(B)\ra H^3(B) & \mbox{if } k=3\\
H^k(B) & \mbox{if } k=4,5,6.
\end{dcases}
\] 
\end{theorem}
For $k>3$ the map $L^2\mathcal{H}^k(B)\ra H^k(B)$ is the obvious map that sends a closed form to its cohomology class. When $k<3$ every $L^2$--integrable closed $k$-form is exact outside a compact set; by radial integration the choice of primitive at infinity can be made canonical, \cf \cite[Lemma 2.11]{Karigiannis}, and therefore there is a well-defined map $L^2\mathcal{H}^k(B)\ra H^k_c(B)$. The borderline case $k=3$ is slightly more involved because $d+d^\ast$ fails to be Fredholm, \cf Proposition \ref{prop:Closed:Odd:Forms}. 

In fact we will need to work with forms on $B$ that decay at a much slower rate than what would be necessary to be $L^2$--integrable. For this reason we have to work with the weighted Sobolev and H\"older spaces introduced in Appendix \ref{Appendix:Analysic:AC}.

\begin{definition}\label{def:H^k_nu}
For $\nu\in\R$ outside of the discrete set of indicial roots of $d+d^\ast$ let $\mathcal{H}^k_\nu (B)$ denote the finite-dimensional space of closed and coclosed $k$-forms on $B$ in $C^\infty_\nu$.
\end{definition}
By weighted elliptic regularity and embeddings Theorems \ref{thm:Weighted:Embedding} and \ref{thm:Weighted:Regularity}, $\mathcal{H}^k_\nu (B)$ coincides with the space of closed and coclosed $k$-forms on $B$ of class $L^p_{l,\nu}$ or $C^{l,\alpha}_\nu$ for any $l\geq 0$, $p\geq 1$ and $\alpha\in (0,1)$.

\begin{definition}\label{def:W^k_nu}
For $\nu\in\R$ outside of the discrete set of indicial roots of $d+d^\ast$ let $\mathcal{W}^k_\nu$ be the $L^2_{\nu}$--orthogonal complement of $\mathcal{H}^k_\nu (B)$ in the space of $k$-forms on $B$ of class $L^2_{\nu}$. 
\end{definition}

\begin{remark}\label{rmk:W^k_nu}
Since $\nu$ is not an indicial root of $d+d^\ast$, we have $\mathcal{H}^k_\nu (B)\subset L^2_{\nu-\delta}$ for any sufficiently small $\delta >0$, \cf Proposition \ref{prop:Decay:Solutions}. Hence the $L^2_\nu$--inner product with an element in $\mathcal{H}^k_\nu (B)$ defines a linear map on the space of $k$-forms of class $C^{0,\alpha}_\nu$ for any $\alpha\in (0,1)$. We can therefore also define in an analogous way a subspace $\mathcal{W}^k_\nu$ of the space of $k$-forms of class $C^{0,\alpha}_\nu$. 
\end{remark}

\subsection{Functions and $1$-forms}

Assume now that $(B,\omega_0,\Omega_0)$ is an \emph{irreducible} AC Calabi--Yau $3$-fold, \ie there are no parallel $1$-forms on $B$. Since the Sasaki--Einstein manifold $\Sigma$ is smooth, this condition holds automatically under our standing assumption that the universal cover of $\Sigma$ is not isometric to the round $5$-sphere. An important consequence of the irreducibility of $B$ is the following vanishing result for decaying harmonic functions and $1$-forms.

\begin{lemma}\label{lem:AC:CY:Vanishing:Harmonic:0:1:Forms}
For any $\nu< 0$ there are no harmonic functions and $1$-forms on $B$ in $C^\infty_\nu$.
\proof
Let $u$ and $\gamma$ be a harmonic function and $1$-form in $C^\infty_{\nu}$. If $\nu< -2$ then Lemma \ref{lem:integration:parts} guarantees that we can integrate by parts:
\[
0=\langle \triangle u , u\rangle _{L^2} = \|\nabla u \|^2_{L^2}, \qquad 0=\langle \triangle \gamma, \gamma\rangle_{L^2} = \|\nabla\gamma\|^2_{L^2},
\]
where, since $B$ is Ricci-flat, we used the fact that $\triangle = \nabla^\ast\nabla$ on $1$-forms. We conclude that $u$ is constant and therefore vanishes since it decays at infinity and $\gamma=0$ since $B$ is irreducible. On the other hand, by Propositions \ref{prop:Harmonic:Functions:Cone} and \ref{prop:Harmonic:1Forms:Cone} there are no indicial roots for the Laplacian acting on functions and $1$-forms in the interval $[-2,0)$. 
\endproof 
\end{lemma}

A similar proof using Lemma \ref{lem:Gauge:fixing:operator:AC:CY} and Proposition \ref{prop:Gauge:Cone} instead of Propositions \ref{prop:Harmonic:Functions:Cone} and \ref{prop:Harmonic:1Forms:Cone} yields the following result.

\begin{lemma}\label{lem:AC:CY:Vanishing:Gauge:Fixing}
For $\nu< 0$ there exist no functions $g$ and $f$ and $1$-forms $\gamma$ on $B$ in $C^\infty_\nu$ such that
\[
\pi_1 dd^\ast \bigl( g\omega_0^2\bigr)=0,\qquad  \pi_{1\oplus 6}dd^\ast \bigl( f\omega_0 + \gamma^\sharp\lrcorner\Real\Omega_0\bigr)=0.
\]
\end{lemma}

An important role in our subsequent analysis will be played by the Dirac operator. Recall that under the isomorphism of the spinor bundle of $B$ with $\underline{\R}\oplus\underline{\R}\oplus T^\ast B$ the Dirac operator $\slashed{D}$ is identified with the first-order operator of Lemma \ref{lem:Dirac:cone}.

\begin{prop}\label{prop:Dirac:AC:CY}
Let $(B,\omega_0,\Omega_0)$ be an irreducible AC Calabi--Yau $3$-fold. The Dirac operator $\slashed{D}\co C^{k+1,\alpha}_{\nu+1}\ra C^{k,\alpha}_\nu$ is an isomorphism for all $\nu\in (-6,-1)$.
\proof
By Proposition \ref{prop:Dirac:cone} there are no indicial roots for the Dirac operator in the range specified and therefore $\slashed{D}\co C^{k+1,\alpha}_{\nu+1}\ra C^{k,\alpha}_\nu$ is Fredholm for all $\nu\in (-6,-1)$. Since $B$ is Ricci-flat, the Lichnerowicz and Weitzenb\"ock formulae imply that $\slashed{D}^2=\triangle$ as an operator from $\Omega^0\oplus\Omega^0\oplus\Omega^1$ into itself. Lemma \ref{lem:AC:CY:Vanishing:Harmonic:0:1:Forms} then implies that $\slashed{D}$ has no kernel and (by duality) cokernel in the range specified.
\endproof
\end{prop}

\begin{remark}\label{rmk:Dirac:AC:CY}
It will be useful to observe that if $(f,g,\gamma)\in C^{1,\alpha}_{\nu+1}$ for some $\nu<-1$ and $\slashed{D}(f,g,\gamma)=(u,v,\alpha)$ with $d^\ast\alpha=0$, then $f=0$. Indeed, by Lemma \ref{lem:Dirac:cone} we have $\alpha=\tu{curl}\,\gamma + df-Jdg$. Moreover $\tu{curl}\,\gamma$ and $Jdg$ are both coclosed since $\tu{curl}\,\gamma=\ast (d\gamma\wedge\Real\Omega_0)$ and $-2Jdg = \ast d\bigl( g\,\omega_0^2\bigr)$. Thus $f$ is a harmonic function of rate $\nu+1<0$ and therefore vanishes by Lemma \ref{lem:AC:CY:Vanishing:Harmonic:0:1:Forms}.
\end{remark}

Finally, we note that every AC Calabi--Yau manifold has finite fundamental group and therefore without loss of generality we can reduce to the case that $B$ is simply connected. This observation is useful when considering principal circle bundles on $B$, which will then be classified by their first Chern class. 
\begin{prop}\label{prop:AC:CY:Fundamental:group}
Let $B$ be an AC Calabi--Yau $3$-fold. Then $B$ has finite fundamental group.
\proof
Fix a large compact set $K\subset B$. Since $\Sigma$ has positive Ricci curvature, the fundamental group of the end $B\setminus K$ is finite. Moreover, $\pi_1 (B\setminus K)\ra \pi_1 (B)$ is surjective: otherwise by \cite[Lemma 2.18]{ACyl:CY} the cover of $B$ with characteristic group the image of $\pi_1 (B\setminus K)$ in $\pi_1 (B)$ would be a Ricci-flat manifold with at least two AC ends, which is impossible by the splitting theorem.
\endproof
\end{prop}

\subsection{Closed and coclosed $2$-forms and $3$-forms}

We can now describe the space of closed and coclosed $2$-forms and $3$-forms on $B$ with certain decay rates.
We will need the following well-known fact which follows from the behaviour of the exact sequence \eqref{eq:Exact:Sequence:Cohomology:Boundary} under Poincar\'e duality on $B$ and $\Sigma$:
see \cite[Corollary 4.63]{Karigiannis:Lotay} for a proof of the analogous result in $7$ dimensions.

\begin{lemma}\label{lem:Closed:2Forms:AC:CY} Identify cohomology classes on $\Sigma$ with their harmonic representatives. Then we have an $L^2$--orthogonal decomposition
\[
H^2(\Sigma) = \tu{im}\left( H^2(B)\ra H^2(\Sigma)\right) \oplus \ast_\Sigma\,\tu{im}\left( H^3(B)\ra H^3(\Sigma)\right).
\]
\end{lemma}

\begin{theorem}\label{thm:Closed:2Forms:AC:CY}
Let $B$ be an AC Calabi--Yau $3$-fold asymptotic to the Calabi--Yau cone $\tu{C}(\Sigma)$.
\begin{enumerate}
\item For all $\nu\in (-6,-2)$ there is a natural isomorphism $\mathcal{H}^2_{\nu}(B)\simeq H^2_{c}(B) \simeq L^2\mathcal{H}^2(B)$.
\item For all $\nu\in (-2,0)$ we have
\[
\dim\mathcal{H}^2_\nu(B) = \dim H^2_c(B) + \dim\textup{im}\left( H^2(B)\ra H^2(\Sigma)\right).
\]
Moreover, for every harmonic $2$-form $\tau$ on $\Sigma$ with $[\tau]\in \textup{im}\left( H^2(B)\ra H^2(\Sigma)\right)$ there exists $\sigma \in \mathcal{H}^2_{\nu}(B)$ for every $\nu>-2$ such that for some $\mu>0$ we have
\[
\sigma = \tau + O(r^{-2-\mu}).
\]
In particular, the natural map $\mathcal{H}^2_\nu (B)\ra H^2 (B)$ is an isomorphism.
\item For all sufficiently small $\delta>0$ there is a natural isomorphism
\[
\mathcal{H}^3_{-3-\delta}(B)\simeq \textup{im}\left( H^3_c(B)\ra H^3(B)\right) \simeq  L^2\mathcal{H}^3(B).
\]
\item For all sufficiently small $\delta>0$
every form $\rho\in \mathcal{H}^3_{-3+\delta}(B)$ has an expansion of the form
\[
\rho = -\eta\wedge\tau_1 + \tfrac{1}{r}dr\wedge\tau_2 + O(r^{-3-\mu})
\]
for some small $\mu>0$ and harmonic $2$-forms $\tau_1,\tau_2$ on $\Sigma$ such that $\ast_\Sigma\tau_i=-\eta\wedge\tau_i$ represents a cohomology class in the image of $H^3(B)\ra H^3(\Sigma)$. Moreover, the map $\rho\mapsto \left( [\eta\wedge\tau_1],[\eta\wedge\tau_2]\right)$ induces an isomorphism
\[
\mathcal{H}^3_{-3+\delta}(B)/\mathcal{H}^3_{-3-\delta}(B)\simeq \textup{im}\left( H^3(B)\ra H^3(\Sigma)\right)\oplus \textup{im}\left( H^3(B)\ra H^3(\Sigma)\right).
\]
 \end{enumerate}
\proof
The statement (iii) follows immediately from Theorem \ref{thm:L2:coho}.
Also by Theorem \ref{thm:L2:coho} there is a natural isomorphism $L^2\mathcal{H}^2(B)\simeq H^2_c(B)$. 
Note that $\mathcal{H}^2_\nu(B)\subseteq L^2\mathcal{H}^2(B)$ for $\nu<-3$. Moreover, by Proposition \ref{prop:Closed:Even:Forms:2} the dimension of $\mathcal{H}^2_\nu$ cannot jump in $(-6,-2)$. This proves (i).

To prove (ii) and (iv) we need to understand 
the change in $\mathcal{H}^2_\nu(B)$ and $\mathcal{H}^3_\nu (B)$ as we cross the indicial roots $-2$ and $-3$ respectively. 
We achieve this in the same way as \cite[Proposition 4.65]{Karigiannis:Lotay}, using Lemma \ref{lem:Closed:2Forms:AC:CY} instead of \cite[Corollary 4.63]{Karigiannis:Lotay}. 

To prove (ii) Proposition \ref{prop:Closed:Even:Forms:2} implies that $\mathcal{H}^2_\nu(B)$ remains constant for all $\nu\in (-2,0)$ and the statement that $\mathcal{H}^2_\nu (B)\ra H^2 (B)$ is an isomorphism for $\nu 
\in (-2,0)$ follows from the fact that $H^2(B) \simeq H^2_c(B)\oplus \tu{im}\left( H^2(B)\ra H^2(\Sigma)\right)$ (using $H^1(\Sigma)=0$
in the exact sequence \eqref{eq:Exact:Sequence:Cohomology:Boundary}).

To establish (iv), Proposition \ref{prop:Harmonic:3forms:Cone} implies that every $\rho\in \mathcal{H}^3_{-3+\delta}(B)$ has an expansion of the form
\[
\rho = -\eta\wedge\tau_1 + \tfrac{1}{r}dr\wedge\tau_2 + O(r^{-3-\mu})
\]
for some small $\mu>0$ and harmonic $2$-forms $\tau_1,\tau_2$ on $\Sigma$. By Remark \ref{rmk:Harmonic:2:forms:SE} both $\tau_i$ are basic primitive $(1,1)$ forms and therefore $\ast_\Sigma \tau_i=-\eta\wedge\tau_i$ by Lemma \ref{lem:Forms:Sasaki:Einstein} (iv). In particular, $\ast\rho$ has an expansion of the form
\[
\ast\rho = -\eta\wedge\tau_2 - \tfrac{1}{r}dr\wedge\tau_1 + O(r^{-3-\mu}).
\]
Hence $[-\eta\wedge\tau_1]$ and $[-\eta\wedge\tau_2]$ represent the images of, respectively, $[\rho]$ and $[\ast\rho]$ in $H^3(\Sigma)$ and the map $\rho\mapsto\left( [\eta\wedge\tau_1],[\eta\wedge\tau_2]\right)$ induces an injection of $\mathcal{H}^3_{-3+\delta}(B)/\mathcal{H}^3_{-3-\delta}(B)$ into $\textup{im}\left( H^3(B)\ra H^3(\Sigma)\right)^{\oplus 2}$. The fact that this map is also surjective follows from Lemma \ref{lem:Laplacian:2Forms:AC:CY} (iii) below.  
\endproof
\end{theorem}



\begin{lemma}\label{lem:Laplacian:2Forms:AC:CY}
Let $(B,\omega_0,\Omega_0)$ be an irreducible AC CY $3$-fold. For $\nu\in \R$ let $\ker{\triangle}^2_\nu$ denote the space of harmonic $2$-forms on $B$ of class $C^{\infty}_\nu$.
\begin{enumerate}
\item If $\nu<-2$, $\ker{\triangle}^2_\nu=\mathcal{H}^2_\nu (B)$, \ie every harmonic $2$-form of rate $\nu$ is closed and coclosed.
\item If $\nu<1$ every $\sigma\in \ker{\triangle}^2_\nu$ is coclosed. In particular, $\ast d\sigma\in \mathcal{H}^3_{\nu-1}(B)$ for every $\sigma\in \ker{\triangle}^2_\nu$.  
\item The composition of the map $\sigma\mapsto [\ast d\sigma]$ and the restriction map $H^3(B)\ra H^3(\Sigma)$ induces an isomorphism between $\ker{\triangle}^2_{-2+\mu}/\mathcal{H}^2_{-2+\mu}(B)$ and $\tu{im}\, \left(H^3(B)\ra H^3(\Sigma)\right)$ for every sufficiently small $\mu>0$. 
\end{enumerate}
\proof
Part (i) follows immediately from the integration by parts formula of Lemma \ref{lem:integration:parts} and part (ii) is a consequence of Lemma \ref{lem:AC:CY:Vanishing:Harmonic:0:1:Forms} since $d^\ast\sigma$ is a decaying harmonic $1$-form on $B$. The proof of part (iii) is more involved.

By Theorem \ref{thm:Closed:2Forms:AC:CY} the dimension of the space of closed and coclosed $2$-forms of rate $-2+\mu$ is $b^2_c(B)+\text{dim}\,\text{im}\left( H^2(B)\ra H^2(\Sigma)\right)$, where $b^2_c(B)=\dim H^2_c(B)$. On the other hand, let $d$ denote the dimension of the space of harmonic $2$-forms in $C^\infty_{-2+\mu}$. Proposition \ref{prop:Harmonic:2forms:cone}, Remark \ref{rmk:Log:Harmonic:2forms:Cone} and the index jump formula for the index of the Laplacian on $2$-forms as we cross the indicial root $-2$ show that $2\left( d-b^2_c(B) \right)=2b_2(\Sigma)$, \ie $d=b^2_c(B)+b_2(\Sigma)$. In particular, $\dim\, \ker{\triangle}^2_{-2+\mu}/\mathcal{H}^2_{-2+\mu}(B) = \dim\, \tu{im}\, H^3(B)\ra H^3(\Sigma)$ by Lemma \ref{lem:Closed:2Forms:AC:CY} and Theorem \ref{thm:Closed:2Forms:AC:CY} (ii).

Now, let $\sigma$ be a harmonic $2$-form of rate $-2+\mu$ for some small $\mu>0$. By part (ii) $\sigma$ is always coclosed and therefore $d\sigma$ is a closed and coclosed $3$-form. By Theorem \ref{thm:Closed:2Forms:AC:CY} (iii) if $d\sigma\in \mathcal{H}^3_{-3-\mu}(B)$ then we must have $d\sigma=0$. Hence $\sigma\mapsto \ast d\sigma$ induces an injective map $\ker{\triangle}^2_{-2+\mu}/\mathcal{H}^2_{-2+\mu}(B) \ra \mathcal{H}^3_{-3+\mu}(B)/\mathcal{H}^3_{-3-\mu}(B)$. Since the image of $[d\sigma]$ in $H^3(\Sigma)$ vanishes, the injectivity of the map in Theorem \ref{thm:Closed:2Forms:AC:CY} (iv) implies that the composition of the map $\sigma\mapsto [\ast d\sigma]$ and the restriction map $H^3(B)\ra H^3(\Sigma)$ induces a linear embedding of $\ker{\triangle}^2_{-2+\mu}/\mathcal{H}^2_{-2+\mu}(B)$ into $\tu{im}\, \left(H^3(B)\ra H^3(\Sigma)\right)$ for every sufficiently small $\mu>0$. The map is in fact an isomorphism for dimensional reasons.
\endproof
\end{lemma}

\begin{remark}\label{rmk:Laplacian:2Forms:AC:CY}
Identify $H^2(\Sigma)$ with the space $\mathcal{H}^2(\Sigma)$ of harmonic $2$-forms on $\Sigma$ and set $V_2 = \tu{im}\, H^2(B)\ra H^2(\Sigma)$. By Lemma \ref{lem:Closed:2Forms:AC:CY} the $L^2$--orthogonal complement of $V_2$ in $H^2(\Sigma)\simeq\mathcal{H}^2(\Sigma)$ is a subspace $V_3$ isomorphic to $\tu{im}\, H^3(B)\ra H^3(\Sigma)$ under the Hodge star operator  ${\ast}$ on $\Sigma$. Choose an $L^2$--orthonormal basis $\tau_1,\dots, \tau_b$ of harmonic $2$-forms on $\Sigma$ with $\tau_i\in V_2$ if $i\leq k$ and $\tau_i\in V_3$ if $i\geq k+1$. Here $b=b_2(\Sigma)$ and $k=\dim\tu{im}\, H^2(B)\ra H^2(\Sigma)$.
By Theorem \ref{thm:Closed:2Forms:AC:CY} (ii) and (iv) there are closed and coclosed $2$-forms $\sigma_1,\dots, \sigma_k$ on $B$ with $\sigma_i = \tau_i +O(r^{-2-\mu})$ and closed and coclosed $3$--forms $\rho_{k+1},\dots, \rho_b$ on $B$ with $\rho_i = \eta\wedge\tau_i + O(r^{-3-\mu})$ for some $\mu>0$ sufficiently small. To make these choices unique we use Theorem \ref{thm:Closed:2Forms:AC:CY} (i) and (iii) and require that $\ast\sigma_i$ and $\rho_i$ integrate to zero on the closed cycles which are Poincar\'e dual to the cohomology classes of closed and coclosed harmonic $L^2$--forms.
Moreover, for every $j=k+1,\dots, b$ there exists a harmonic $2$-form $\overline{\sigma}_j$ on $B$ coclosed but \emph{not} closed such that
\begin{equation}\label{eq:Laplacian:2Forms:AC:CY}
\overline{\sigma}_j = -\left(\log{r}\right)\,\tau_j + \sum_{i=k+1}^{b}{\alpha^i_j\, \tau_i} +O(r^{-2-\mu})
\end{equation}
for some $\alpha^i_j\in \R$ and $\mu>0$ sufficiently small. Indeed, by Lemma \ref{lem:Laplacian:2Forms:AC:CY} (iii) for all $j=k+1,\dots, b$ there exists a unique harmonic $2$-form $\overline{\sigma}_j$ modulo $\vspan\{ \sigma_1,\dots, \sigma_k\}$ and $L^2\mathcal{H}^2(B)$ such that $d\overline{\sigma}_j=\ast\rho_j$. In order to determine the asymptotic behaviour of $\overline{\sigma}_j$, we solve $d\overline{\sigma}_j=\ast\rho_j$ on the cone itself using the expansion $\ast\rho_j = -r^{-1}dr\wedge\tau_j + O(r^{-3-\mu})$ and deduce that we can choose $\overline{\sigma}_j$ satisfying \eqref{eq:Laplacian:2Forms:AC:CY} as claimed.
\end{remark}

Lemma \ref{lem:Laplacian:2Forms:AC:CY} shows that in general there are decaying harmonic $2$-forms on $B$ that are not closed. One way this fact causes complications compared to the compact setting is that the equation $\triangle\sigma = d^\ast\psi\in C^{0,\alpha}_{\nu-1}$ for a $2$-form $\sigma$ might have no solutions in $C^{2,\alpha}_{\nu+1}$. We will now show that for a certain range of rates $\nu$, this equation can always be solved with $d\sigma\in C^{1,\alpha}_\nu$, while $\sigma$ itself is allowed to decay slower than $r^{\nu+1}$.

\begin{prop}\label{prop:Laplacian:2Forms:AC:CY}
Let $(B,\omega_0,\Omega_0)$ be an irreducible AC CY $3$-fold. Fix $k\geq 1$, $\alpha\in (0,1)$, sufficiently small $\delta>0$ depending only on the rate of decay of $B$ to $\tu{C}(\Sigma)$ and the Sasaki--Einstein manifold $\Sigma$ and $\nu > -3-\delta$ so that $\nu+1$ is not an indicial root for the Laplacian acting on $2$-forms. Then there exists a constant $C>0$ such that for every $3$-form $\psi\in C^{k,\alpha}_{\nu}$ the equation 
\[\triangle\sigma = d^\ast\psi
\]
has a solution $\sigma$ with $d\sigma\in C^{k,\alpha}_{\nu}$ and satisfying
\[
\|d\sigma\|_{C^{k,\alpha}_{\nu}} \leq C\| d^\ast\psi\|_{C^{k-1,\alpha}_{\nu-1}}.
\]
\proof
By elliptic regularity it is enough to prove the result with $k=1$.

The obstructions to solve $\triangle\sigma = d^\ast\psi$ with $\sigma\in C^{2,\alpha}_{\nu+1}$ lie in the space of harmonic $2$-forms in $C^\infty_{-5-\nu}$ which are not closed. By Lemma \ref{lem:Laplacian:2Forms:AC:CY} (i), 
if $\nu> -3$ then there exists a unique $2$-form $\sigma\in C^{2,\alpha}_{\nu+1}$ which is $L^2_{\nu+1}$--orthogonal to the space of harmonic $2$-forms in $C^\infty_{\nu+1}$ and satisfies $\triangle\sigma = d^\ast\psi$. The estimate for $\| d\sigma\|_{C^{1,\alpha}_{\nu}}$ (even better, $\|\sigma\|_{C^{2,\alpha}_{\nu+1}}$) follows in a standard way. 

Similarly, if $\nu\in (-3-\delta,-3)$ and 
$H^2(B)\ra H^2(\Sigma)$ is surjective then every harmonic $2$-form of rate $-2+\mu$ is closed and coclosed by Lemma \ref{lem:Laplacian:2Forms:AC:CY} (iii) and therefore there are no obstructions to solve $\triangle\sigma=d^\ast\psi$ with $\sigma\in C^{2,\alpha}_{\nu+1}$. Moreover, the unique solution which is $L^2_{\nu+1}$--orthogonal to harmonic forms satisfies $\|\sigma\|_{C^{2,\alpha}_{\nu+1}}\leq C\| d^\ast\psi\|_{C^{0,\alpha}_{\nu-1}}$ for some constant $C>0$ independent of $\sigma$ and $\psi$. 

We have therefore reduced to the case where $\nu\in (-3-\delta,-3)$ and $H^2(B)\ra H^2(\Sigma)$ is \emph{not} surjective. In this case there are genuine obstructions to solve $\triangle\sigma=d^\ast\psi$ with $\sigma\in C^{2,\alpha}_{\nu+1}$. However, for every small $\mu>0$ we can always solve the equation $\triangle\sigma=d^\ast\psi$ for some $\sigma\in C^{2,\alpha}_{-2+\mu}$, unique up to the addition of a harmonic form. We will show that we can choose $\sigma$ so that $d\sigma\in C^{1,\alpha}_{\nu}$, \ie $d\sigma$ decays faster than expected. More precisely, Proposition \ref{prop:Harmonic:2forms:cone} and Remark \ref{rmk:Log:Harmonic:2forms:Cone} imply that $\sigma = \tau\log{r} + \tau' + O(r^{\nu+1}) $ for some $\tau,\tau'\in\mathcal{H}^2(\Sigma)$. We want to show that we can make $\tau=0$ by adding a harmonic form to $\sigma$, \ie there are no $\log{r}$ terms. This is the key point of the proof.

Consider the harmonic coclosed but not closed $2$-forms $\overline{\sigma}_{k+1}, \dots, \overline{\sigma}_b$ constructed in Remark \ref{rmk:Laplacian:2Forms:AC:CY} and note that $\{ \overline{\sigma}_{k+1}, \dots, \overline{\sigma}_b\}$ is a basis of the space of obstructions to solve $\triangle\sigma = d^\ast\psi$ with $\sigma\in C^{2,\alpha}_{\nu+1}$. Fix a cut-off function $\chi\equiv 1$ on $B\setminus K$ for some large compact set $K\subset B$. For all $j=k+1,\dots, b$ consider the $2$-form $\sigma'_j=\chi\, \tau_j$. Since $\tau_j$ is closed and coclosed on the cone $\tu{C}(\Sigma)$ (but not necessarily coclosed on $B$), note that $d\sigma'_j$ is compactly supported and $d^\ast\sigma'_j\in C^\infty_{\nu}$ if $\delta>0$ is chosen small enough.
On the other hand $\overline{\sigma}_h\in C^\infty_{-2+\mu}$ for any $\mu>0$ and, since $\nu<-3$, we can fix $\mu$ so that $\nu+(-2+\mu)<-5$. Then we can integrate by parts a first time appealing to Lemma \ref{lem:integration:parts}, but we need to calculate boundary terms explicitly when integrating by parts a second time: 
\[
\begin{aligned}
\langle \triangle \sigma'_j,\overline{\sigma}_h\rangle_{L^2} &= \langle d^\ast \sigma'_j,d^\ast\overline{\sigma}_h\rangle_{L^2} + \langle d \sigma'_j,d\overline{\sigma}_h\rangle_{L^2} = \langle d \sigma'_j,\ast \rho_h\rangle_{L^2}\\
&= -\lim_{R\ra\infty}{\int_{r\leq R}{d\left( \sigma'_j\wedge\rho_h\right)}} = -\lim_{R\ra\infty}{\int_{r= R}{\tau_j \wedge \left( \eta\wedge\tau_h + O(r^{-3-\mu})\right)}}\\
&=-\int_{\Sigma}{\tau_j\wedge\eta\wedge\tau_h}=\delta_{jh},
\end{aligned}
\]
for every $h=k+1,\dots,b$. Here we used the fact that $d^\ast\overline{\sigma}_h=0$, $d\overline{\sigma}_h=\ast\rho_h$ and $d\rho_h=0$.

Now, given $\psi\in C^{1,\alpha}_{\nu}$ set $a_i = \langle d^\ast\psi, \overline{\sigma}_i\rangle_{L^2}$. Note that $|a_i|\leq C\| d^\ast \psi\|_{C^{0,\alpha}_{\nu-1}}$ since $|d^\ast\psi| \leq \| d^\ast \psi\|_{C^{0,\alpha}_{\nu-1}} r^{\nu-1}$, $|\overline{\sigma}_i|\leq C r^{-2}\log{r}$ and $r^{\nu-3}\log{r}$ is integrable when $\nu<-3$. Consider $d^\ast\psi - \sum_{i=k+1}^b{a_i\, \triangle\sigma'_i}\in C^{0,\alpha}_{\nu-1}$. Since
\[
\langle d^\ast\psi - \sum_{i=k+1}^b{a_i\, \triangle\sigma'_i}, \overline{\sigma}_h\rangle_{L^2}=0
\]
for all $h=k+1,\dots, b$, there exists a unique $2$-form $\sigma'\in C^{2,\alpha}_{\nu+1}$ which is $L^2_{\nu+1}$--orthogonal to harmonic $2$-forms and satisfies $\triangle \sigma' = d^\ast\psi - \sum_{i=k+1}^b{a_i\, \triangle \sigma'_i}$. Moreover,
\[
\|\sigma'\|_{C^{2,\alpha}_{\nu+1}} \leq C\left( \| d^\ast\psi\|_{C^{0,\alpha}_{\nu-1}} + \sum_{i=k+1}^b{|a_i|}\right) \leq C\| d^\ast\psi\|_{C^{0,\alpha}_{\nu-1}}
\]
since $\|\triangle\sigma'_i\|_{C^{0,\alpha}_{\nu-1}}$ is uniformly bounded. We then set $\sigma = \sigma' + \sum_{i=k+1}^b{a_i\, \sigma'_i}$.
\endproof
\end{prop}

\subsection{Normal forms for exact $4$-forms}\label{sec:Normal:form:4:forms}

We now use the previous results to give ``normal forms'' for exact $4$-forms on $B$ that exploit the interplay between the mapping properties of the operator $d+d^\ast$ and the type-decomposition of differential forms on an AC Calabi--Yau $3$-fold $(B,\omega_0,\Omega_0)$. These results will be essential to understand the image of the linearisation of the Apostolov--Salamon equations; they are mostly direct consequences of algebraic identities involving the type-decomposition of differential forms and the mapping properties of the Dirac operator on $B$.

We begin with the following immediate corollary of Proposition \ref{prop:Dirac:AC:CY}.

\begin{lemma}\label{cor:3forms:AC:CY}
For all $\nu\in (-6,-1)$, $k\geq 0$ and $\alpha\in (0,1)$ there exists a constant $C>0$ such that for any $3$-form $\rho\in C^{k,\alpha}_{\nu}$ there exist unique $\rho_0\in C^{k,\alpha}_{\nu}\cap \Omega^3_{12}$ and $f,g,\gamma\in C^{k+1,\alpha}_{\nu+1}$ with
\[
\| (f,g,\gamma)\|_{C^{k+1,\alpha}_{\nu+1}}+ \| \rho_0\|_{C^{k,\alpha}_{\nu}} \leq C\| \rho\|_{C^{k,\alpha}_{\nu}}
\]
and
\[
\rho = d(f\omega_0 + \gamma^\sharp\lrcorner\Real\Omega_0) + d^\ast\bigl( \tfrac{1}{2}g\omega_0^2\bigr) + \rho_0.
\]
\proof
The statement follows immediately from Proposition \ref{prop:Dirac:AC:CY} since the operator $\Omega^2_{1\oplus 6}\oplus \Omega^4_1 \ra \Omega^3_{1\oplus 1\oplus 6}$ defined by $(\alpha, \beta)\mapsto \pi_{1\oplus 1\oplus 6}(d\alpha + d^\ast\beta)$ can be identified with the Dirac operator $\slashed{D}$ of $B$, \cf \cite[Equation (3.26)]{Nordstrom:Thesis}.
\endproof
\end{lemma}

Recall that $\mathcal{W}^3_\nu$ was defined in Definition \ref{def:W^k_nu} and Remark \ref{rmk:W^k_nu} as a complement of $\mathcal{H}^3_\nu (B)$ in $C^{0,\alpha}_\nu$.

\begin{prop}\label{prop:Exact:4form:AC:CY}
Let $(B,\omega_0,\Omega_0)$ be an irreducible AC Calabi--Yau $3$-fold asymptotic to the Calabi--Yau cone $\tu{C}(\Sigma)$. Fix $k\geq 1$, $\alpha\in (0,1)$, $\delta>0$ as in Proposition \ref{prop:Laplacian:2Forms:AC:CY} and $\nu\in (-3-\delta,-1)$ away from a discrete set of indicial roots. Then every exact $4$-form $\sigma=d\rho'$ with $\rho'\in C^{k,\alpha}_{\nu}$ can be written uniquely as
\[
\sigma = d\!\ast\! d\bigl( f\omega_0 + \gamma^\sharp\lrcorner\Real\Omega_0\bigr) + d\rho_0,
\]
where $f,\gamma\in C^{k+1,\alpha}_{\nu+1}$, $\rho_0\in C^{k,\alpha}_\nu\cap \mathcal{W}^{3}_{\nu}\cap \Omega^3_{12}$ with $d^\ast\rho_0=0$ and
\[
\| (f,\gamma)\| _{C^{k+1,\alpha}_{\nu+1}} + \|\rho_0\|_{C^{k,\alpha}_{\nu}} \leq C\|\rho'\|_{C^{k,\alpha}_{\nu}}
\]
for a constant $C>0$ independent of $\rho'$. Moreover, $f=0=\gamma$ if $\sigma\in\Omega^4_8$.
\proof
We first establish that we can write $\sigma = d\rho$ for some $3$-form $\rho\in C^{k,\alpha}_{\nu}$ with $d^\ast\rho=0$. If it exists, $\rho$ is uniquely defined up to the addition of a closed and coclosed $3$-form in $C^{k,\alpha}_{\nu}$. Moreover every such harmonic $3$-form lies in $\Omega^3_{12}$ by Lemma \ref{lem:AC:CY:Vanishing:Harmonic:0:1:Forms}.

By assumption, $\sigma=d\rho'$ for some $\rho'\in C^{k,\alpha}_{\nu}$. By Proposition \ref{prop:Laplacian:2Forms:AC:CY} there exists a $2$-form $\alpha$ such that $\triangle \alpha = -d^\ast\rho'$ and $d\alpha\in C^{k,\alpha}_{\nu}$. From the proof of that Proposition we know that $\alpha\in C^{k+1,\alpha}_{\nu+1}$ if $\nu\geq -3$ or if $H^2(B)\ra H^2(\Sigma)$ is surjective and $\alpha\in C^{k+1,\alpha}_{-2+\mu}$ for any small $\mu>0$ otherwise. Since $\nu<0$, $d^\ast \alpha$ is a decaying harmonic $1$-form and therefore vanishes by Lemma \ref{lem:AC:CY:Vanishing:Harmonic:0:1:Forms}. Thus $\rho=\rho'+d\alpha\in C^{k,\alpha}_{\nu}$ is coclosed and $\sigma = d\rho$. 

Using Lemma \ref{cor:3forms:AC:CY} we write
\[
\ast\rho = -d\bigl( f\omega_0 + \gamma^\sharp\lrcorner\Real\Omega_0\bigr) + \tfrac{1}{2}d^\ast (g\omega_0^2) + \ast \rho_0
\] 
with $\rho_0\in\Omega^3_{12}\cap C^{k,\alpha}_{\nu}$ and $f,g,\gamma\in C^{k+1,\alpha}_{\nu+1}$. Since $\rho$ is uniquely defined up to addition of a closed and coclosed $3$-form of type $\Omega^3_{12}$ we make $\rho$ unique by requiring that $\rho_0\in\mathcal{W}^3_{\nu}$.

Now, the closure of $\ast \rho$ is equivalent to the vanishing of the $4$-form $\tfrac{1}{2}dd^\ast (g\omega_0^2) + d\ast \rho_0$. In particular we have
\[
0=\pi_1 \bigl(\tfrac{1}{2}dd^\ast (g\omega_0^2) + d\!\ast\!\rho_0\bigr)=\pi_1 dd^\ast \bigl( \tfrac{1}{2}g\omega_0^2\bigr)
\]
since $d(\ast\rho_0)\wedge\omega_0 = d(\ast\rho_0\wedge\omega_0)=0$. Lemma \ref{lem:Gauge:fixing:operator:AC:CY} then shows that $g$ is a harmonic function and therefore vanishes since $\nu+1<0$. The co-closure of $\rho$ now forces $d^\ast\rho_0=0$.

In order to prove the last statement, assume moreover that $\sigma = d\ast d\bigl( f\omega_0 + \gamma^\sharp\lrcorner\Real\Omega_0\bigr) + d\rho_0\in \Omega^4_8$. Since $d^\ast\rho_0=0$ we have $d\rho_0 \in \Omega^4_8$ by Remark \ref{rmk:d:3forms:type6:1}. We therefore conclude that $d^\ast d\bigl( f\omega_0 + \gamma^\sharp\lrcorner\Real\Omega_0\bigr)\in \Omega^2_8$. By Lemma \ref{lem:Gauge:fixing:operator:AC:CY}, $f$ is therefore harmonic and $\gamma$ satisfies $dd^\ast\gamma + \tfrac{2}{3}d^\ast d\gamma=0$. Since $\nu+1<0$ Proposition \ref {lem:AC:CY:Vanishing:Gauge:Fixing}
shows that $f=0=\gamma$.
\endproof
\end{prop}

For our applications it will also be necessary to write every exact $4$-form $\sigma\in d\rho'$, $\rho'\in C^{k,\alpha}_\nu$, as $\sigma = d\rho$ for a different choice of $\rho$.

\begin{corollary}\label{cor:Exact:4form:AC:CY}
In the notation of Proposition \ref{prop:Exact:4form:AC:CY}, fix $\nu \in (-3-\delta,-1)$ away from a discrete set of indicial roots. For every exact $4$-form $\sigma=d\rho'$ with $\rho'\in C^{k,\alpha}_{\nu}$ there exist unique $\rho_0\in C^{k,\alpha}_\nu\cap\mathcal{W}^3_{\nu}\cap \Omega^3_{12}$ with $d^\ast\rho_0=0$ and $f,\gamma\in C^{k+1,\alpha}_{\nu+1}$ such that
\[
\sigma = d\!\ast\!d(f\omega_0) + d \bigl(\ast d(\gamma^\sharp\lrcorner\Real\Omega_0)-d(\gamma^\sharp\lrcorner\Imag\Omega_0)\bigr) + d\rho_0.
\]
Moreover,
\[
\ast d(\gamma^\sharp\lrcorner\Real\Omega_0)-d(\gamma^\sharp\lrcorner\Imag\Omega_0)+\rho_0 \in \Omega^3_{12}.
\] 
\proof
This is an immediate consequence of Proposition \ref{prop:Exact:4form:AC:CY} and Lemma \ref{lem:d:3forms:type6}.
\endproof
\end{corollary}

\section{Approximate solutions}
\label{sec:approx:solns}

With these preliminaries out of the way, we now return to the main goal of the paper. Let $(B,\omega_0,\Omega_0)$ be an AC Calabi--Yau $3$-fold asymptotic to the Calabi--Yau cone $\tu{C}(\Sigma)$ and let $M^7$ be a principal circle bundle over $B$. We want to construct $\tu{S}^1$--invariant ALC \gtmetric s on $M$ by constructing solutions to the Apostolov--Salamon equations starting from an abelian Hermitian Yang--Mills (HYM) connection on $(B,\omega_0,\Omega_0)$. We look for solutions of the rescaled equations \eqref{eq:GH:G2:collapsing:sequence} for small $\epsilon>0$. In this section we solve the linearised equations \eqref{eq:GH:G2:linearised} on $(B,\omega_0,\Omega_0)$ and therefore construct  a $1$-parameter family of closed ALC \gtstr s $\varphi^{(1)}_\epsilon$ on $M$ with torsion of order $O(\epsilon^2)$. In the next two sections we will then show that for $\epsilon$ sufficiently small any such approximate solution can be perturbed to a solution of the Apostolov--Salamon equations \eqref{eq:GH:G2:collapsing:sequence}.

Following the discussion of Section \ref{sec:GH:G2}, we look for solutions of \eqref{eq:GH:G2:linearised} on an AC Calabi--Yau $3$-fold $(B,\omega_0,\Omega_0)$ with vanishing $\sigma$ and $h$: the resulting coupled linear system for a $U(1)$--connection $\theta$ and a $3$-form $\rho$ on $B$ is
\begin{equation}
\label{eq:GH:G2:linearised:HYM}
\begin{gathered}
d\theta\wedge\Imag\Omega_0=0, \quad d\theta\wedge\omega_0^2=0,\qquad
d\rho = -d\theta\wedge\omega_0,\quad
d\hat{\rho}=0,\\
\omega_0\wedge\left( \rho + i\hat{\rho}\right) =0,\quad
\Real\Omega_0\wedge\hat{\rho}+\rho\wedge\Imag\Omega_0 = 0.
\end{gathered}
\end{equation}

By Lemma \ref{lem:identities:2forms} the first two equations imply that $d\theta$ is a primitive $(1,1)$--form, \ie $\theta$ is a HYM connection. Equivalently, $\ast d\theta=-d\theta\wedge\omega_0$ by Lemma \ref{lem:identities:2forms}.

In order to understand the equations for $\rho$, recall that $\hat{\rho}$ is given explicitly in Proposition \ref{prop:Linearisation:Hitchin:dual} in terms of the type decomposition 
of $\rho$. 
Using this one sees that the algebraic constraints in \eqref{eq:GH:G2:linearised:HYM} force $\rho$ to be of the form $\rho = f\Imag\Omega_0 + \rho_0$ for a function $f$ and a $3$-form $\rho_0\in\Omega^3_{12}$. Given a HYM connection $\theta$ we will look for a $3$-form $\rho$ that satisfies \eqref{eq:GH:G2:linearised:HYM} by solving instead the inhomogeneous linear \emph{elliptic} system
\begin{equation}\label{eq:GH:G2:linearised:HYM:3form}
d\rho = -d\theta \wedge \omega_0=\ast d\theta, \qquad \quad d^\ast\rho=0.
\end{equation}
It will turn out that in our situation any such $3$-form $\rho$ will be of pure type $\Omega^3_{12}$, and therefore $\hat{\rho}=\!-\!\ast\!\rho$ by Proposition \ref{prop:Linearisation:Hitchin:dual}. 
Hence  any pair $(\theta,\rho)$ consisting of any HYM connection $\theta$ and $\rho$ a solution of \eqref{eq:GH:G2:linearised:HYM:3form} gives rise to  a solution of the linearised equations \eqref{eq:GH:G2:linearised:HYM} 
(in particular the algebraic constraints in \eqref{eq:GH:G2:linearised:HYM} are automatically satisfied).

Therefore to solve the coupled linear system  \eqref{eq:GH:G2:linearised:HYM} for the pair $(\theta,\rho)$ we proceed in two steps: 
first we find a HYM connection $\theta$ on the circle bundle $M \to B$; 
second we solve the inhomogeneous linear elliptic system \eqref{eq:GH:G2:linearised:HYM:3form} with source term 
determined by the HYM connection $\theta$ constructed 
in the first step.
To approach both these steps analytically the results about the space of closed and coclosed $2$-forms on $B$ with certain decay rates obtained in the previous section are crucial.

The following theorem is the basic existence result for solutions of \eqref{eq:GH:G2:linearised:HYM}. 

\begin{theorem}\label{thm:GH:G2:linearised:HYM}
Let $(B,\omega_0,\Omega_0)$ be a simply connected irreducible AC Calabi--Yau $3$-fold and let $M\ra B$ a principal circle bundle. Assume that
\begin{equation}\label{eq:sigma_tau:orthogonality:1}
c_1(M)\cup[\omega_0]=0\in H^4(B).
\end{equation}
Fix $\nu=-1+\delta$ for some small $\delta>0$. Then there exists a unique solution $(\theta,\rho)$ of \eqref{eq:GH:G2:linearised:HYM} with $d\theta\in C^\infty_{-2}$ and $\rho\in C^\infty_{-1}\cap \mathcal{W}^3_\nu\cap \Omega^3_{12}$.
\end{theorem}

\begin{remark*}
Given any solution $(\theta,\rho)$ of \eqref{eq:GH:G2:linearised:HYM} the addition of any closed and coclosed $3$-form in $\Omega^3_{12}$ to $\rho$ 
yields another solution 
(any such form corresponds to an infinitesimal Calabi--Yau deformation of $B$); insisting that $\rho\in\mathcal{W}^3_\nu$ fixes this ambiguity. 
\end{remark*}

\proof
By Theorem \ref{thm:Closed:2Forms:AC:CY} (ii) we can represent $c_1 (M)\in H^2(B)$ by a closed and coclosed $\kappa\in\mathcal{H}^2_{-2+\mu}(B)$ for any sufficiently small $\mu>0$. Since $c_1 (M)$ is an integral class, $\kappa = d\theta$ for a connection $\theta$ on $M\ra B$ unique up to gauge transformations. Note that $\kappa = \tau + O(r^{-2-\mu})$ for some $\mu>0$, where $\tau\in\mathcal{H}^2(\Sigma)$ is the harmonic representative of the image of $c_1 (M)$ in $H^2(\Sigma)$. Moreover, Lemma \ref{lem:AC:CY:Vanishing:Harmonic:0:1:Forms} guarantees that $\kappa\in\Omega^2_8$, \ie $\theta$ is HYM. Indeed, since $\triangle\kappa=0$ and the Laplacian preserves the type decomposition, $\pi_1\kappa$ and $\pi_6\kappa$ are forced to vanish. Hence \eqref{eq:sigma_tau:orthogonality:1} can be rewritten as
\begin{equation}\label{eq:sigma_tau:orthogonality}
[\ast \kappa]=0\in H^4(B).
\end{equation}

Now, the fact that $\ast\kappa$ is exact implies that $\kappa$ is $L^2$--orthogonal to $\mathcal{H}^2_\nu(B)$, $\nu\in (-6,-2)$. Indeed, first of all note that the $L^2$--inner product $\langle \kappa, \sigma\rangle_{L^2}$ for $\sigma\in \mathcal{H}^2_\nu(B)$, $\nu\in (-6,-2)$, is well defined since $\sigma$ decays at least as fast as $r^{-6+\mu}$ for every $\mu>0$ by Proposition \ref{prop:Closed:Even:Forms:2}. Moreover, the isomorphism $\mathcal{H}^2_{\nu}(B)\simeq H^2_{c}(B)$, $\nu \in (-6,-2)$, can be made explicit as follows: by integration in the radial direction outside a compact set as in \cite[Lemma 2.11]{Karigiannis}, every $\sigma\in \mathcal{H}^2_\nu(B)$, $\nu\in (-6,-2)$, can be written as
\[
\sigma = \sigma_c + d\zeta
\] 
with $\sigma_c$ closed and compactly supported and $\zeta \in C^{\infty}_{-5+\mu}$ for every $\mu>0$. Then if $\ast\kappa = dv$ for some $3$-form $v$ we have
\[
\langle \kappa, \sigma\rangle_{L^2} = \lim_{R\ra \infty}\int_{r\leq R}{ dv \wedge (\sigma_c + d\zeta)} = \lim_{R\ra \infty}\int_{r\leq R}{d\left( v\wedge\sigma_c + \ast\kappa\wedge\zeta\right) }=0.
\]

In order to solve \eqref{eq:GH:G2:linearised:HYM} we consider the elliptic equation \eqref{eq:GH:G2:linearised:HYM:3form} for a $3$-form $\rho$. We look for an odd-degree form $\rho$ such that $(d+d^\ast)\rho = \ast\kappa$ with $\rho\in C^{1,\alpha}_{\nu}$ for $\nu>-1$ arbitrarily close to $-1$. Note that if $\rho$ exists then it must be of pure degree $3$. Indeed, any solution is harmonic and Lemma \ref{lem:AC:CY:Vanishing:Harmonic:0:1:Forms} guarantees that the components of degree $1$ and $5$ of $\rho$ vanish. For the same reason, if a solution exists then it must be of type $\Omega^3_{12}$. By Proposition \ref{prop:Linearisation:Hitchin:dual}, we then conclude that, given $\kappa$, solving \eqref{eq:GH:G2:linearised:HYM:3form} is equivalent to solving \eqref{eq:GH:G2:linearised:HYM}.

Now, the obstructions to solve $(d+d^\ast)\rho=\ast\kappa$ lie in the space of closed and coclosed even-degree forms of rate $-5-\nu<-4$. These are in particular $L^2$--integrable and by Theorem \ref{thm:L2:coho} are parametrised by a subset of $H^2_c(M) \oplus H^4(M)$. Note that since we are in the $L^2$--range, each pure-degree component is individually closed and coclosed. Thus the only possibly non-vanishing obstructions to solve $(d+d^\ast)\rho=\ast\kappa$ arise from inner products $\langle \kappa,\sigma\rangle_{L^2}$ for $\sigma\in\mathcal{H}^2_{-5-\nu}$. We have already seen that these $L^2$--inner products all vanish by \eqref{eq:sigma_tau:orthogonality}.

We conclude that if $c_1 (M)$ satisfies \eqref{eq:sigma_tau:orthogonality:1} then there exists a harmonic $3$-form $\rho$ on $B$ that solves \eqref{eq:GH:G2:linearised:HYM:3form}. The $3$-form $\rho$ is uniquely defined up to the addition of a closed and coclosed $3$-form in $\mathcal{H}^3_\nu (B)$. We fix the choice of $\rho$ by requiring that $\rho\in \mathcal{W}^3_{\nu}$. By elliptic regularity $\rho$ is smooth.

Finally, we derive the more precise asymptotic behaviour of $\rho$ at infinity claimed in the statement. We already know that $\rho\in C^{1,\alpha}_{-1+\delta}$ for every $\delta>0$: we want to show that $\rho\in C^\infty_{-1}$ as claimed. In order to understand the asymptotic behaviour of $\rho$, we solve \eqref{eq:GH:G2:linearised:HYM:3form} on the cone $\tu{C}(\Sigma)$ itself with $\tau$ in place of $\kappa$: note that $\ast\tau = rdr\wedge\eta\wedge\tau$ on $\tu{C}(\Sigma)$ and therefore $\rho_\infty = \tfrac{1}{2}r^2\eta\wedge\tau$ satisfies $d\rho_\infty =\ast \tau$ and $d^\ast\rho_\infty=0$ on $\tu{C}(\Sigma)$. Taking into account Proposition \ref{prop:Closed:Odd:Forms} we conclude that 
\begin{equation}\label{eq:Asymptotics:rho:tau}
\rho = \tfrac{1}{2}r^2\eta\wedge\tau + rdr\wedge\alpha +\tfrac{1}{2}r^2d\alpha + O(r^{-1-\delta})
\end{equation}
for some $\delta>0$ and a coclosed $2$-form $\alpha$ on $\Sigma$ such that $\triangle\alpha = 4\alpha$ (by Proposition \ref{prop:Spectrum:Laplacian:Regular:SE}, $\alpha=0$ if $\Sigma$ is a regular Sasaki--Einstein $5$-manifold). 
\endproof

\begin{remark*}
If $c_1 (M)=0$ then $M = B\times \tu{S}^1$ and $\theta$ is the trivial connection. Moreover, $\rho\in\Omega^3_{12}$ is a solution of \eqref{eq:GH:G2:linearised:HYM} with $d\theta=0$ if and only if $\rho$ is closed and coclosed. Hence $\rho=0$ if $\rho\in \mathcal{W}^3_\nu$.	
\end{remark*}

Fix a principal circle bundle $M\ra B$ satisfying \eqref{eq:sigma_tau:orthogonality:1} and let $(\theta,\rho)$ be the solution of \eqref{eq:GH:G2:linearised:HYM} given by Theorem \ref{thm:GH:G2:linearised:HYM}. For all $\epsilon>0$ sufficiently small we now define a $1$-parameter family
\begin{equation}\label{eq:Approximate:GH:G2:1}
\varphi^{(1)}_\epsilon = \epsilon\, \theta \wedge \omega_0 + \Real\Omega_0 + \epsilon\rho
\end{equation}
of highly collapsed \emph{closed} ALC \gtstr s on $M$ with torsion of order $O(\epsilon^2)$. Here $\epsilon$ should be chosen small enough to ensure that $\Real\Omega_0 + \epsilon\rho$ is a stable $3$-form on $B$.

We now would like to deform $\varphi^{(1)}_\epsilon$ to a torsion-free \gtstr. Joyce gave a general and flexible result that allows one to deform closed \gtstr s with small torsion to torsion-free \gtstr s on \emph{compact} manifolds. Provided that one develops a good Fredholm theory for operators on ALC manifolds, it is relatively straightforward to adapt Joyce's results to the non-compact ALC setting when one not only assumes that the torsion is small,  but also that it decays fast enough. The approximate solutions $\varphi^{(1)}_\epsilon$ that we constructed do not satisfy this fast decay condition and therefore it would be necessary to correct them by hand as a preliminary step. In fact, since we are looking to deform $\varphi^{(1)}_\epsilon$ as an $\tu{S}^1$--invariant torsion-free \gtstr, we find it more straightforward to solve directly the Apostolov--Salamon equations \eqref{eq:GH:G2:collapsing:sequence} on the AC manifold $B$
(bypassing the need to adapt Joyce's results to the ALC setting). This is the goal of the next two sections.

\section{The linearisation of the Apostolov--Salamon equations}
\label{sec:lin:AS}

In this section we develop the requisite tools to solve the Apostolov--Salamon equations as a power series in $\epsilon$ for $\epsilon>0$ sufficiently small. Understanding the mapping properties of the linearisation of the Apostolov--Salamon system and showing that these equations can be solved to all orders in $\epsilon$ requires some effort. As an indication that this must be expected, note that when $d\theta=0$ and $\rho$ is therefore a closed and coclosed $3$-form on $B$, then the problem reduces to showing that deformations of complex structure of AC Calabi--Yau $3$-folds that preserve the asymptotic cone at infinity are unobstructed. 

Our proof is divided into various steps. Firstly, we show how to work transversely to the obvious kernel of the linearisation arising from the action of diffeomorphisms and the deformations of the AC Calabi--Yau structure on the base. As a consequence of Proposition \ref{prop:Torsion:SU(3):structures}, which gives constraints on the torsion of an {\suthreestr} on a $6$-manifold, we then show that one can add additional free parameters to the Apostolov--Salamon equations. These additional parameters are exploited to understand the mapping properties of the linearisation of the Apostolov--Salamon equations. The ``normal forms'' for exact $4$-forms obtained in Section \ref{sec:Normal:form:4:forms} will also be essential in our analysis. In the next section we apply the results of this section to deform $\varphi^{(1)}_\epsilon$ to an $\tu{S}^1$--invariant torsion-free {\gtstr} on $M$.

\subsection{Gauge-fixing and Calabi--Yau deformations}\label{sec:Gauge:Fixing}

Non-uniqueness of solutions of the linearisation \eqref{eq:GH:G2:linearised} of the Apostolov--Salamon equations arises from two obvious sources: the action of diffeomorphisms and a (possibly) non-trivial moduli space of AC Calabi--Yau structures on the collapsed limit $B$. In this preliminary section we explain how to work transversally to these sources of non-uniqueness.

Consider the sequence of symplectic manifolds $(B,\omega_\epsilon)$ which converges to $(B,\omega_0)$. Without loss of generality we can assume that $\omega_\epsilon$ represents the fixed cohomology class $[\omega_0]\in H^2(B)$. Indeed, by Conlon--Hein \cite[Theorem 2.4]{Conlon:Hein:I} we know that every K\"ahler class on $B$ can be represented by a unique K\"ahler Ricci-flat metric $\omega$ such that $\tfrac{1}{6}\omega^3=\tfrac{1}{4}\Real\Omega_0\wedge\Imag\Omega_0$. On the other hand, every cohomology class close to $[\omega_0]$ is a K\"ahler class on the given complex manifold $(B,\Omega_0)$. Indeed, by Theorem \ref{thm:Closed:2Forms:AC:CY} (ii) every cohomology class on $B$ can be represented by a closed and coclosed $2$-form $\sigma$ with respect to the metric induced by $(\omega_0,
\Omega_0)$ of rate $\leq -2$. Every such harmonic representative must be a primitive $(1,1)$--form on $(B,\omega_0,\Omega_0)$ by Lemma \ref{lem:AC:CY:Vanishing:Harmonic:0:1:Forms} and therefore $\omega_0+\delta\sigma$ is a K\"ahler form on $(B,\Omega_0)$ for $\delta$ sufficiently small.

If $[\omega_\epsilon]=[\omega_0]$, then the next lemma and Moser's deformation argument for a path of symplectic structures imply that up to diffeomorphism we can assume that $\omega_\epsilon \equiv \omega_0$ for all $\epsilon$ sufficiently small.

\begin{lemma}\label{lem:AC:CY:Exact:2:forms}
Let $\sigma$ be an exact $2$-form on $(B,\omega_0,\Omega_0)$ in $C^{k,\alpha}_\nu$ for some $k\geq 1$, $\alpha\in (0,1)$ and $\nu\in (-5,-1)$. Then there exists a unique coclosed $1$-form $\gamma\in C^{k+1,\alpha}_{\nu+1}$ such that $\sigma = d\gamma$.
\proof
We first solve $\triangle \gamma = d^\ast \sigma\in C^{k-1,\alpha}_{\nu-1}$. By Lemma \ref{lem:AC:CY:Vanishing:Harmonic:0:1:Forms}, since $-5-\nu<0$ there are no obstructions to solve this equation and since $\nu+1<0$ the solution is unique.

Now, $d^\ast\gamma$ is a harmonic function of rate $\nu<0$ and therefore vanishes by Lemma \ref{lem:AC:CY:Vanishing:Harmonic:0:1:Forms}. On the other hand $d\gamma-\sigma$ is a closed and coclosed $2$-form of rate $\nu<0$ representing the trivial cohomology class. Since $\mathcal{H}^2_{\nu}\ra H^2(B)$ is injective by Theorem \ref{thm:Closed:2Forms:AC:CY} (ii) we conclude that $\sigma = d\gamma$.
\endproof
\end{lemma}
Now, if $[\omega_\epsilon]=[\omega_0]$ then the Lemma shows that $\omega_\epsilon=\omega_0 + d\gamma$ for a \emph{decaying} $\gamma$ (since $\nu+1<0$). The decay of $\gamma$ is crucial for two reasons: firstly, when exponentiating $\gamma^\sharp$ to define a time dependent flow connecting $\omega_\epsilon$ to $\omega_0$ we do not have to worry about the non-compactness of $B$; secondly, the diffeomorphism so defined preserves the AC asymptotics of the Calabi--Yau structure. 

Making the assumption $\omega_\epsilon\equiv\omega_0$ allows us to break the diffeomorphism invariance up to vector fields (with decay conditions) preserving $\omega_0$. It remains to understand how to work transversally to such vector fields. Consider then a vector field $X\in C^{k+1,\alpha}_{\nu+1}$ with $\nu\in (-5,-1)$ such that $d(JX^\flat)=0$. The action of $X$ on $\Real\Omega_0$ is by Lie derivative $\mathcal{L}_X\Real\Omega_0 = d(X\lrcorner\Real\Omega_0)$. Proposition \ref{prop:differential:2forms} (iv) and (v) and Remark \ref{rmk:d:3forms:type6} imply that
\begin{equation}\label{eq:Gauge:fixing:transverse:CY:def}
d(X \lrcorner\Real\Omega_0)  = \tfrac{1}{2}(d^\ast\! JX^\flat)\Imag\Omega_0 + \rho_0
\end{equation}
for some $\rho_0\in \Omega^3_{12}$. Moreover, by Lemma \ref{lem:AC:CY:Vanishing:Harmonic:0:1:Forms}, $X=0$ if $d^\ast\! JX^\flat=0=d(JX^\flat)$, \ie $X$ is uniquely determined by $d^\ast\! JX^\flat$. On the other hand every $3$-form $f\Imag\Omega_0$ with $f\in C^{k,\alpha}_{\nu}$, $\nu>-6$, can be written as
\[
f\Imag\Omega_0 = d\left( (J\nabla u)\lrcorner\Real\Omega_0\right) + \rho_0
\] 
for some $\rho_0\in C^{k,\alpha}_{\nu}$ and $u\in C^{k+2,\alpha}_{\nu+2}$. Indeed, it is enough to look for $u$ such that $\triangle u = 2f$ and then use \eqref{eq:Gauge:fixing:transverse:CY:def} with $X=J\nabla u$. It follows that for any $3$-form $\rho$ of class $C^{k,\alpha}_\nu$ there exists a unique $X\in C^{k+1,\alpha}_\nu$ with $d(JX^\flat)=0$ such that $\rho+\mathcal{L}_X\Real\Omega_0$ contains no component of the form $f\Imag\Omega_0$, \ie $\left( \rho+\mathcal{L}_X\Real\Omega_0\right) \wedge\Real\Omega_0=0$. We will therefore work transversally to all diffeomorphisms of class $C^{k+1,\alpha}_{\nu+1}$ by requiring that
\[
\omega_\epsilon = \omega_0, \qquad \Real\Omega_\epsilon\wedge\Real\Omega_0=0.
\]

We summarise our discussion in the following proposition.

\begin{prop}\label{prop:Gauge:fixing:transverse:CY:def}
Fix $\nu\in (-5,-1)$, $k\geq 1$ and $\alpha\in (0,1)$ and let $(\sigma,\rho)\in C^{k,\alpha}_{\nu}$ be an infinitesimal deformation of the {\suthreestr} $(\omega_0,\Real\Omega_0)$ on $B$ with $[\sigma]=0\in H^2(B)$. Then there exists a vector field $X\in C^{k+1,\alpha}_{\nu+1}$ such that $\sigma + \mathcal{L}_{X}\omega_0 =0$ and $\rho'=\rho + \mathcal{L}_{X}\Real\Omega_0$ satisfies $\rho'\wedge\Real\Omega_0=0$. Moreover $(0,\rho')$ is an infinitesimal Calabi--Yau deformation of $(B,\omega_0,\Omega_0)$ if and only if $\rho'$ is a closed and coclosed $3$-form in $\Omega^3_{12}$.
\proof
Only the last statement requires proof: $(0,\rho)$ is an infinitesimal Calabi--Yau deformation if and only if
\[
d\rho = 0= d\hat{\rho}, \qquad \rho\wedge\omega_0=0=\Real\Omega_0\wedge\hat{\rho} + \rho\wedge\Imag\Omega_0.
\]
If moreover $\rho\wedge\Real\Omega_0=0$ then $\rho\in\Omega^3_{12}$ and $d\hat{\rho}=0$ is equivalent to $d^\ast\rho=0$ by Proposition \ref{prop:Linearisation:Hitchin:dual}. 
\endproof
\end{prop}

\subsection{The deformation problem}

The next step is to rewrite the Apostolov--Salamon system in a way that makes it easier to identify the image of its linearisation. The strategy to achieve this is to exploit Proposition \ref{prop:Torsion:SU(3):structures}, which states that if $(\omega,\Omega)$ is an {\suthreestr} on a $6$-manifold then there are relations between the differentials of the defining differential forms. 

\begin{prop}\label{prop:Modified:GH:G2}
Let $\Lie{c}_0 =(\omega_0,\Omega_0)$ be an AC Calabi--Yau structure on a $6$-manifold $B$ and denote by $g_0$ and $\nabla_0$ the induced metric and Levi--Civita connection. Fix $k\geq 1$, $\alpha\in (0,1)$ and $\nu<-1$. Then there exists a constant $\epsilon_0>0$ such that the following holds. Let $\Lie{c} = (\omega,\Omega)$ be a second {\suthreestr} on $B$ such that $r|\Lie{c}-\Lie{c}_0| + r^2 |\nabla_0 (\Lie{c}-\Lie{c}_0)|<\epsilon_0$. Suppose that there exists a function $h$ and an integral closed $2$-form $\kappa=d\theta$ on $B$ such that
\[
\kappa\wedge\omega^2=0, \qquad \tfrac{1}{2}dh\wedge\omega^2 = h^{\frac{1}{4}}\kappa\wedge\Imag\Omega.
\]
Moreover assume the existence of functions $u,v$ and a vector field $X$ in $C^{k+1,\alpha}_{\nu+1}$ such that
\[d\omega=0, \quad d\bigl( h^{\frac{3}{4}}\Real\Omega\bigr) + \kappa\wedge\omega = d\!\ast\! d(u\,\omega),\quad
d\bigl( h^{\frac{1}{4}}\Imag\Omega\bigr)  = d\!\ast\!d\left( X\lrcorner\Real\Omega + v\,\omega\right).
\]
Here the Hodge $\ast$ is computed with respect to the metric induced by $\Lie{c}_0$.

Then $u=v=X=0$, \ie $(\omega,\Omega,h,\theta)$ is a solution of the Apostolov--Salamon equations \eqref{eq:GH:G2}.
\proof
We first prove that $u=0$. Integrating the equation satisfied by $\Real\Omega$ against $u'\in L^2_{0,\mu}$ with $\mu<-7-\nu$ yields
\[
\langle d^\ast d(u\,\omega),u'\omega\rangle_{L^2} = -\int{h^{\frac{3}{4}}\left( d\Real\Omega + \tfrac{3}{4}h^{-1}dh\wedge\Real\Omega + h^{-\frac{3}{4}}\kappa\wedge\omega\right) \wedge u'\omega}.
\]
Since $\omega$ is closed, the assumption $\kappa\wedge\omega^2=0$ and Proposition \ref{prop:Torsion:SU(3):structures} imply that the integral on the right-hand-side vanishes. We therefore conclude that $u$ lies in the kernel of $\pi^\Lie{c}_1 d^\ast d (u\,\omega)$, where $\pi^\Lie{c}_1$ is the pointwise adjoint of $u'\mapsto u' \omega$ with respect to the metric induced by $\Lie{c}_0$.

Now, by the assumptions on the closeness of $\Lie{c}$ to $\Lie{c}_0$, the operator $C^{k+1,\alpha}_{\nu+1}\ra C^{k-1,\alpha}_{\nu-1}$ defined by $u\mapsto \pi^\Lie{c}_1 d^\ast d (u\,\omega)$ differs from $u\mapsto \pi_1 d^\ast d (u\,\omega_0)$ by a bounded operator of norm controlled by $\epsilon_0$. Lemma \ref{lem:Gauge:fixing:operator:AC:CY} and the assumption $\nu<-1$ show that $\pi_1 d^\ast d (u\,\omega_0)=0$ implies $u=0$. If $\epsilon_0$ is sufficiently small we conclude that $u=0$ even if $\Lie{c}_0$ is replaced with $\Lie{c}$.

We will now consider the equation satisfied by $\Imag\Omega$ and deduce that $X=0=v$. Using Lemmas \ref{lem:identities:1forms} (v) and \ref{lem:identities:2forms} (ii) one can show that the equation $\tfrac{1}{2}dh\wedge\omega^2 = h^{\frac{1}{4}}\kappa\wedge\Imag\Omega$ is equivalent to $\pi_6 \kappa = -\tfrac{1}{2}h^{-\frac{1}{4}}(J\nabla h)\lrcorner\Real\Omega$. Algebraic manipulations of $d\bigl( h^{\frac{3}{4}}\Real\Omega\bigr) + \kappa\wedge\omega =0$ using Lemmas \ref{lem:Hodge:star} (ii) and \ref{lem:identities:2forms} (i) then yield $\pi_6 \left( d\Real\Omega\right) = -\tfrac{1}{4}h^{-1}dh\wedge\Real\Omega$. Combined with the fact that $d\omega=0$, Proposition \ref{prop:Torsion:SU(3):structures} implies that $h^{\frac{1}{4}}\ast\bigl( d\Imag\Omega + \frac{1}{4}h^{-1}dh\wedge\Imag\Omega\bigr)$ is a primitive $(1,1)$--form with respect to $(\omega,\Omega)$. Thus a similar integration by parts as above shows that $(X,v)$ lies in the kernel of the operator $(X,v)\mapsto \pi_{1\oplus 6}d^\ast d (X\lrcorner\Real\Omega + v\,\omega)$. Closeness of $\Lie{c}$ to $\Lie{c}_0$, the assumption $\nu<-1$ and Lemma \ref{lem:Gauge:fixing:operator:AC:CY} finally imply that $X=0=v$.
\endproof
\end{prop}

\subsection{The linearised problem}

We now exploit Proposition \ref{prop:Exact:4form:AC:CY} and Corollary \ref{cor:Exact:4form:AC:CY} to study the mapping properties of the linearisation of the Apostolov--Salamon equations \eqref{eq:GH:G2}, reformulated as in Proposition \ref{prop:Modified:GH:G2}.

\begin{theorem}\label{thm:Linearisation:GH:G2}
Let $(B,\omega_0,\Omega_0)$ be an AC Calabi--Yau $3$-fold. Fix $k\geq 1$, $\alpha\in (0,1)$, $\delta>0$ as in Proposition \ref{prop:Laplacian:2Forms:AC:CY} and $\nu\in (-3-\delta,-1)$ away from a discrete set of indicial roots. Then there exists a constant $C>0$ such that the following holds.

Let $\alpha_0$ be a function in $C^{k,\alpha}_{\nu}$, $\alpha_1$ a closed $5$-form in $C^{k-1,\alpha}_{\nu-1}$ and $\alpha_2=d\beta_2$ and $\alpha_3=d\beta_3$ exact $4$-forms with $\beta_2, \beta_3\in C^{k,\alpha}_{\nu}$. Then there exist a unique function $h$, $1$-form $\gamma$, $3$-form $\rho$ of the form $\tfrac{1}{2}\alpha_0\Real\Omega_0 + \Omega^3_{12}$, all in $C^{k,\alpha}_{\nu}$, and functions $f_1,f_2$ and a vector field $X$ in $C^{k+1,\alpha}_{\nu+1}$ such that
\[
\begin{gathered}
d^\ast\gamma =0, \qquad d\gamma\wedge\omega_0^2=0, \qquad \tfrac{1}{2}dh\wedge\omega_0^2-d\gamma\wedge\Imag\Omega_0 = \alpha_1,\\
d\rho + \tfrac{3}{4}dh\wedge\Real\Omega_0 + d\gamma\wedge\omega_0 + d\!\ast\! d(f_1\omega_0)=\alpha_2,\\
d\hat{\rho}+\tfrac{1}{4}dh\wedge\Imag\Omega_0+d\!\ast\!d(X\lrcorner\Real\Omega_0 + f_2\omega_0) =\alpha_3.
\end{gathered} 
\]
Moreover,
\[
\| (h,\gamma,\rho)\| _{C^{k,\alpha}_{\nu}} + \| (f_1,f_2,X)\|_{C^{k+1,\alpha}_{\nu+1}} \leq C\left( \| (\alpha_0, \beta_2,
\beta_3)\|_{C^{k,\alpha}_{\nu}} + \|\alpha_1\|_{C^{k-1,\alpha}_{\nu-1}}\right).
\]
\proof
We start by looking for $(h,\gamma)$. After recalling Lemma \ref{lem:Dirac:cone} we notice that the three equations that only involve the pair $(h,\gamma)$ can be interpreted as the inhomogeneous Dirac equation
\[
\slashed{D}(0,h,-J\gamma) = (0,0,\ast\alpha_1).
\]
By Proposition \ref{prop:Dirac:AC:CY} and Remark \ref{rmk:Dirac:AC:CY} there exists a unique solution $(h,\gamma)$ with
\[
\| (h,\gamma)\| _{C^{k,\alpha}_{\nu}}\leq C \|\alpha_1\| _{C^{k-1,\alpha}_{\nu-1}}.
\]

Once $h$ and $\gamma$ are given, by Corollary \ref{cor:Exact:4form:AC:CY} we can write
\[
\alpha_2 - \tfrac{1}{2}d\left( \alpha_0\Real\Omega_0\right) -\tfrac{3}{4}dh\wedge\Real\Omega_0 - d\gamma\wedge\omega_0  = d\ast d(u_1\omega_0) + d\left( \ast d( Y_1\lrcorner\Real\Omega_0) -d(Y_1\lrcorner\Imag\Omega_0)\right) +d\rho_0
\]
for unique $(u_1,Y_1,\rho_0)$ with $\rho_0\in \Omega^3_{12}\cap\mathcal{W}^3_{\nu}$ and $d^\ast\rho_0=0$. Moreover, the $C^{k,\alpha}_{\nu}$--norm of $\rho_0$ and the $C^{k+1,\alpha}_{\nu+1}$--norm of $(u_1,Y_1)$ are uniformly controlled by $\|(\alpha_0,\beta_2)\|_{C^{k,\alpha}_{\nu}}$ and $\| (h,\gamma)\| _{C^{k,\alpha}_{\nu}}$.

We now take $f_1=u_1$ and look for $\rho$ of the form
\[
\rho = \tfrac{1}{2}\alpha_0\Real\Omega_0 + \ast d( Y_1\lrcorner\Real\Omega_0) -d(Y_1\lrcorner\Imag\Omega_0) +\rho_0 + \ast\rho'_0 
\]
for some $\rho'_0\in \Omega^3_{12}\cap\mathcal{W}^3_{\nu}$ with $d^\ast\rho'_0=0$. Moreover, Lemma \ref{lem:d:3forms:type6} implies that the sum of the second and third term lies in $\Omega^3_{12}$. Using Proposition \ref{prop:Linearisation:Hitchin:dual} we calculate
\[
\hat{\rho} = \tfrac{1}{2}\alpha_0\Imag\Omega_0 + d( Y_1\lrcorner\Real\Omega_0) +\ast d(Y_1\lrcorner\Imag\Omega_0) -\!\ast\rho_0 + \rho'_0.
\]
By Proposition \ref{prop:Exact:4form:AC:CY} we furthermore write in a unique way
\[
\alpha_3 -\tfrac{1}{4}dh\wedge\Imag\Omega_0-d\left( \tfrac{1}{2}\alpha_0\Imag\Omega_0 + d( Y_1\lrcorner\Real\Omega_0) +\ast d(Y_1\lrcorner\Imag\Omega_0)\right) = d\!\ast\!d\left( Y_2\lrcorner\Real\Omega_0 + u_2\omega_0\right) + d\rho''_0
\]
with $d^\ast\rho''_0 = 0$ and $\rho''_0\in \Omega^3_{12}\cap\mathcal{W}^3_{\nu}$. We then set $f_2=u_2$, $X=Y_2$ and $\rho'_0=\rho''_0$.
\endproof
\end{theorem}

\section{Existence of highly collapsed circle-invariant torsion-free \gtstr s}
\label{sec:g2:existence}

Return now to the setting of Theorem \ref{thm:GH:G2:linearised:HYM}. Let $(B,\omega_0,\Omega_0)$ be a simply connected AC Calabi--Yau $3$-fold and let $M\ra B$ be a non-trivial principal $\tu{S}^1$--bundle satisfying \eqref{eq:sigma_tau:orthogonality:1}. Theorem \ref{thm:GH:G2:linearised:HYM} guarantees the existence of a HYM connection $\theta$ on $M$ and a $3$-form $\rho\in\Omega^3_{12}$ such that \eqref{eq:GH:G2:linearised:HYM} are satisfied where $\hat{\rho}=-\!\ast\!\rho$. In other words there exists $\epsilon_0>0$ such that for all $\epsilon \in (0,\epsilon_0)$
\[
\varphi^{(1)}_\epsilon = \epsilon\,\theta\wedge\omega_0 + \Real\Omega_0 + \epsilon\rho
\]
is a \emph{closed} $\tu{S}^1$-invariant {\gtstr} on $M$ with torsion of order $O(\epsilon^2)$. Moreover, $d\theta = O(r^{-2})$ and $\rho = O(r^{-1})$ as $r\ra\infty$ with analogous estimates on all higher derivatives and therefore the metric on $M$ induced by $\varphi^{(1)}_\epsilon$ is ALC. We now aim to perturb $\varphi^{(1)}_\epsilon$ to an $\tu{S}^1$--invariant \emph{torsion-free} ALC \gtstr. In fact, in view of the $\tu{S}^1$--invariance we work directly on the $6$-manifold $B$ and aim to perturb $(\theta,\rho)$ to a solution of the Apostolov--Salamon equations \eqref{eq:GH:G2:collapsing:sequence} for $\epsilon$ sufficiently small.

\begin{theorem}\label{thm:GH:G2:Existence}
Let $(B,\omega_0,\Omega_0)$ be a simply connected irreducible AC Calabi--Yau $3$-fold and let $M\ra B$ a non-trivial principal circle bundle that satisfies
\[
c_1(M)\cup[\omega_0]=0 \in H^4(B).
\]
Let $(\theta,\rho)$ be the solution of \eqref{eq:GH:G2:linearised:HYM} with $d\theta\in C^\infty_{-2}$ and $\rho\in C^\infty_{-1}$ given by Theorem \ref{thm:GH:G2:linearised:HYM}.

Fix $l\geq 1$, $\alpha\in (0,1)$ and $\nu\in (-2,-1)$ away from a discrete set of indicial roots. Then there exist $\epsilon_0, C>0$ such that for all $\epsilon\in (0,\epsilon_0)$ there exists a unique solution
\[
h_\epsilon = 1 + h', \qquad \theta_\epsilon = \theta + \theta', \qquad \omega_\epsilon=\omega_0, \qquad \Real\Omega_\epsilon = \Real\Omega_0 + \epsilon \rho+ \rho',
\]
of the Apostolov--Salamon equations \eqref{eq:GH:G2:collapsing:sequence} with $d^\ast\theta'\!=0$ and 
$\rho'$ of the form $\frac{1}{2}\alpha_0 \Real\Omega_0 + \Omega^3_{12}$. Moreover,
\[
\| (h',\epsilon\, \theta',\rho')\|_{C^{l,\alpha}_\nu} \leq C\epsilon^2.
\]

\end{theorem}
Our main existence result Theorem \ref{thm:Main:Theorem:technical} in the Introduction follows immediately from Theorems \ref{thm:GH:G2:linearised:HYM} and \ref{thm:GH:G2:Existence}.

\begin{proof}[Proof of Theorem \ref{thm:Main:Theorem:technical}]
By Lemma \ref{lem:GH:G2} every solution $(h_\epsilon,\theta_\epsilon,\omega_\epsilon,\Omega_\epsilon)$ of the Apostolov--Salamon equations \eqref{eq:GH:G2:collapsing:sequence} on $B$ corresponds to an $\tu{S}^1$--invariant torsion-free {\gtstr} on $M$. The estimates in Theorems \ref{thm:GH:G2:linearised:HYM} and \ref{thm:GH:G2:Existence} guarantee that, with $\nu<-1$,
\[
h_\epsilon = 1 +O(\epsilon^2r^{\nu}), \qquad \Omega_\epsilon = \Omega_0 +O(\epsilon r^{-1}),
\]
and $\theta_\epsilon$ approaches a connection $\theta_\infty$ on $\tu{C}(\Sigma)$ up to terms of order $O(\epsilon r^{\nu})$, where all estimates hold in H\"older spaces. It is then clear that, if $g_\epsilon$ denotes the metric induced by $\varphi_\epsilon$, $(M,g_\epsilon)$ is an ALC manifold and $g_\epsilon$ can be made arbitrarily $C^{l,\alpha}$--close to the Riemannian submersion $g_0 + \epsilon ^2 \theta^2$ as $\epsilon\ra 0$. It is left to prove that the restricted holonomy group of $(M,g_\epsilon)$ is necessarily the whole \gtwo.

First of all note that $M$ has finite fundamental group since $c_1 (M)\neq 0$. Hence without loss of generality we can assume that $M$ is simply connected. Then the holonomy of $(M,g_\epsilon)$ reduces to a strict subgroup of \gtwo~if and only if $(M,g_\epsilon)$ carries non-trivial parallel $1$-forms \cite[Lemma 1]{Bryant:1987}.

Let then $\Gamma$ be a parallel $1$-form on $(M,g_\epsilon)$. We have to show that $\Gamma =0$. An argument exploiting the adiabatic limit seems possible, even though the results of the papers  \cite{Dai,Mazzeo:Melrose,Forman}
cannot be immediately applied to our situation because of the non-compactness of $M$ and $B$. We sketch instead a proof that exploits the ALC structure and holds more generally even when $\epsilon$ is not small. This argument requires the extension of the analytic results of Appendix \ref{Appendix:Analysic:AC} from the AC to the ALC setting. 
We are preparing a paper that develops such a systematic theory in the ALC setting, however the results of Hausel--Hunsicker--Mazzeo \cite{HHM} (which hold more generally for fibred boundary metrics) already suffice for our current purposes.

Now, since $\Gamma$ is parallel it is in particular closed,  coclosed and bounded. By \cite[Proposition 16]{HHM} and Proposition \ref{prop:Harmonic:1Forms:Cone}, outside of a compact set
\[
\Gamma = a\,\theta_\infty +O(r^{\delta})
\]
for some $a\in\R$ and $\delta <0$. Moreover $\delta \leq-4$ if $a=0$. In this case $\Gamma\in L^2$. Now, the $L^2$--cohomology of $M$ is topological \cite[Corollary 1]{HHM}: $L^2\mathcal{H}^1(M)\simeq H^1_c(M)$. We can then use a Gysin sequence to identify $H^1_c(M)$ with $H^1_c(B)\simeq H^1(B)$ and then conclude that $L^2\mathcal{H}^1(M)=0$ using the simply-connectedness of $B$. As a consequence, the map $\Gamma\mapsto a$ is injective and therefore there exists at most one parallel $1$-form $\Gamma$ up to scale. Moreover, since the metric is $\tu{S}^1$--invariant, the $\tu{S}^1$ action must preserve the space of parallel $1$-forms and therefore $\Gamma$ is $\tu{S}^1$--invariant. Hence $\Gamma = u\,\theta + \gamma$ where $u,\gamma$ are pulled-back from $B$, $\gamma =O(r^\delta)$ for some $\delta<0$ and $u$ is a bounded function on $B$ that approaches the constant function $1$ at infinity with rate $\delta$. A straightforward calculation shows that the equation $d\,\Gamma=0$ forces $u\equiv 1$ and $d\gamma = -d\theta$. However this is impossible since $[d\theta]=c_1(M)\neq 0\in H^2(B,\Z)$.
\end{proof}

\begin{remark*}\label{rmk:Symmetries}
Note that in Theorem \ref{thm:Main:Theorem:technical} we start from an AC Calabi--Yau structure $(\omega_0,\Omega_0)$ on $B$ and not only the induced K\"ahler Ricci-flat metric $g_0$. Different choices of {\suthreestr} $(\omega_0,\Omega_0)$ inducing the same metric $g_0$ give rise to different ALC \gtmetric s on $M$ (possibly related by a diffeomorphism). Since we are assuming that $B$ is irreducible, and therefore $(B,g_0)$ has full holonomy $\sunitary{3}$, the space of Calabi--Yau structures on $B$ inducing the same metric $g_0$ is $\tu{S}^1\rtimes \Z_2$, where $\Z_2$ is generated by complex conjugation $(\omega_0,\Omega_0)\mapsto (-\omega_0,\overline{\Omega}_0)$ and $e^{it}\in \tu{S}^1$ corresponds to a change of phase $\Omega_0\mapsto e^{it}\Omega_0$ of the complex volume form. Note that the involution $(\theta,h,\omega,\Omega)\mapsto (-\theta,h,-\omega,\overline{\Omega})$ is a symmetry of the Apostolov--Salamon equations \eqref{eq:GH:G2} and therefore for every $\epsilon$ sufficiently small applying Theorem \ref{thm:Main:Theorem:technical} to $M\to (B,\omega_0,\Omega_0)$ and to its dual $M^\vee\ra (B,-\omega_0,\overline{\Omega}_0)$ yields the same ALC {\gtmetric} $g_\epsilon$ on $M$. Note also that in the examples considered in the next section the AC Calabi--Yau $3$-fold $B$ is a crepant resolution of a Calabi--Yau cone: in this case there is a diffeomorphism of $B$ (generated by the lift of the Reeb vector field of the cone, which has linear growth) that relates Calabi--Yau structures differing by a phase of the complex volume form.   
\end{remark*}

\vbox{The rest of this section contains the proof of Theorem \ref{thm:GH:G2:Existence}. Our strategy is to construct a solution of the Apostolov--Salamon equations as a power series in $\epsilon$ by solving  \eqref{eq:GH:G2:collapsing:sequence} iteratively to all orders in $\epsilon$. We first use Theorem \ref{thm:Linearisation:GH:G2} to construct a formal power series solution and then prove that the series converges in weighted H\"older spaces for $\epsilon$ sufficiently small.}

\subsection{Existence of formal power series solutions}
\label{sec:as:formal}

Fix $l\geq 1$, $\alpha\in (0,1)$ and $\nu\in (-2,-1)$ away from a discrete set of indicial roots. We look for an {\suthreestr} $(\omega_\epsilon,\Omega_\epsilon)$ and a positive function $h_\epsilon$ on $B$ and a connection $\theta_\epsilon$ on $M\ra B$ expressed as power series in $\epsilon$
\begin{subequations}
\begin{equation}\label{eq:Power:Series:Data}
\begin{gathered}
\omega_\epsilon = \omega_0,\qquad \quad \Real\Omega_\epsilon = \Real\Omega_0 + \epsilon\,\rho + \sum_{k\geq 2}{\epsilon^k \rho_k},\\
h_\epsilon = 1 + \sum_{k\geq 2}{\epsilon^k\, h_k},\qquad  \quad \epsilon\,\theta_\epsilon = \epsilon\,\theta + \sum_{k\geq 2}{\epsilon^k\, \gamma_k},
\end{gathered}
\end{equation}
with $\rho_k, h_k, \gamma_k\in C^{l,\alpha}_\nu$ for all $k\geq 2$. Working transversally to gauge transformations we require  further that $d^\ast \gamma_k=0$ and, in view of Proposition \ref{prop:Gauge:fixing:transverse:CY:def}, that $\rho_k = \frac{1}{2}\alpha_{0,k}\Real\Omega_0 + \pi_{12}\,\rho_k$ for some function $\alpha_{0,k}$. It will be convenient to set $h_1 =0$, $\rho_1 = \rho$ and $\gamma_1 = \theta$.

For all $k\geq 2$ there exists a map $Q_k = Q_k (\rho_1, \rho_2,\dots, \rho_{k-1})$ depending real-analytically on its arguments such that
\begin{equation}\label{eq:Power:Series:Dual}
\Imag\Omega_\epsilon = \Imag\Omega_0 + \epsilon\, \hat{\rho} + \sum_{k\geq 2}{\epsilon^k\left( \hat{\rho}_k + Q_k \right)}.
\end{equation}
\end{subequations}
Here $\rho_k\mapsto \hat{\rho}_k$ is the linearisation of Hitchin's duality map for stable $3$-forms on $6$-manifolds. In fact, if $\rho_i$ is given degree $i$ for all $i\geq 1$, then $Q_k$ is a weighted homogeneous polynomial of degree $k$ with values in the space of $3$-forms. We set $Q_1=0$.

In order for $(\omega_\epsilon,\Omega_\epsilon)$ to be an {\suthreestr} we have to impose the algebraic constraints \eqref{eq:SU(3):structure:Constraints}. While the condition $\omega_\epsilon\wedge\Real\Omega_\epsilon=0$ is automatically satisfied with our choice of $\rho_k$, the condition $\frac{1}{4}\Real\Omega_\epsilon\wedge\Imag\Omega_\epsilon = \frac{1}{6}\omega^3_\epsilon$ implies that $\alpha_{0,k} = \alpha_{0,k}(\rho_1, \rho_2,\dots, \rho_{k-1})$ is a degree-$k$ weighted homogeneous polynomial in its arguments. Indeed, imposing $\Real\Omega_\epsilon\wedge\Imag\Omega_\epsilon = \Real\Omega_0\wedge\Imag\Omega_0$ to all orders in $\epsilon$ defines $\alpha_{0,k}$ uniquely as a function of $\rho_1, \rho_2,\dots, \rho_{k-1}$ by
\begin{equation}\label{eq:GH:G2:Power:Series:k:RHS:0}
\alpha_{0,k}\Real\Omega_0 \wedge\Imag\Omega_0 + \Real\Omega_0\wedge Q_k + \sum_{m=1}^{k-1}{\rho_{k-m}\wedge\left( \hat{\rho}_m + Q_m\right)}=0.
\end{equation}
Here we used the fact that $\rho_k\wedge\Imag\Omega_0 + \Real\Omega_0\wedge\hat{\rho}_k=\alpha_{0,k}\Real\Omega_0\wedge\Imag\Omega_0$ with our definition of $\rho_k$.

Finally there are open conditions to be satisfied: $\Real\Omega_\epsilon$ must be a stable $3$-form and $h_\epsilon>0$; equivalently, $\Real\Omega_\epsilon-\Real\Omega_0$ and $h_\epsilon-1$ are sufficiently small in $C^0$--norm. In fact we will need to require that $(\omega_\epsilon,\Omega_\epsilon)$ is close to $(\omega_0,\Omega_0)$ in the weighted $C^1$--norm of Proposition \ref{prop:Modified:GH:G2}. In this subsection we aim to show that we can construct successive polynomial approximations to a solution of the Apostolov--Salamon equations \eqref{eq:GH:G2:collapsing:sequence}. Hence we can always assume that these open conditions are satisfied by taking $\epsilon$ sufficiently small. More precisely, for every $k\geq 1$ there exists $\epsilon_k>0$ such that the degree-$k$ truncations of $\Real\Omega_\epsilon$ and $h_\epsilon$ are close to $\Real\Omega_0$ and $1$ in the relevant norms for all $\epsilon \in (0,\epsilon_k)$. In the next subsection we will show that the sequence $\epsilon_k$ is bounded away from zero. 

Fix $k\geq 2$. Suppose that $(h_i,\gamma_i,\rho_i)\in C^{l,\alpha}_\nu$ have been determined for $i=1,\dots, k-1$ and that the Apostolov--Salamon equations \eqref{eq:GH:G2:collapsing:sequence} are satisfied up to terms of order $O(\epsilon^k)$. We now show that we can choose $(h_k,\gamma_k,\rho_k)$ so that the Apostolov--Salamon equations are satisfied up to terms of order $O(\epsilon^{k+1})$. The vanishing of the degree-$k$ component of the Apostolov--Salamon system is equivalent to the linear system
\begin{equation}\label{eq:GH:G2:Power:Series:k}
\begin{gathered}
d^\ast\gamma_k =0, \qquad d\gamma_k\wedge\omega_0^2=0, \qquad \tfrac{1}{2}dh_k\wedge\omega_0^2-d\gamma_k\wedge\Imag\Omega_0 = \alpha_{1,k},\\
d\rho_k + \tfrac{3}{4}dh_k\wedge\Real\Omega_0 + d\gamma_k\wedge\omega_0 =\alpha_{2,k},\qquad d\hat{\rho}_k+\tfrac{1}{4}dh_k\wedge\Imag\Omega_0=\alpha_{3,k},
\end{gathered}
\end{equation}
where
\begin{equation}\label{eq:GH:G2:Power:Series:k:RHS}
\begin{aligned}
\alpha_{1,k}& = \sum_{m=1}^{k-1}{(\epsilon\, d\theta_\epsilon)_{[m]} \wedge \bigl( h_\epsilon^{\frac{1}{4}}\Imag\Omega_\epsilon\bigr)_{[k-m]}},\\
\alpha_{2,k}&=-d\left( \sum_{m=1}^{k-1}{\bigl( h_\epsilon^{\frac{3}{4}}\bigr)_{[m]}\left( \Real\Omega_\epsilon\right)_{[k-m]}}\right),\\
 \alpha_{3,k}&=-d\left( Q_k + \sum_{m=1}^{k-1}{\bigl( h_\epsilon^{\frac{1}{4}}\bigr)_{[m]}\left( \Imag\Omega_\epsilon\right)_{[k-m]}}\right).
 \end{aligned}
\end{equation}
Here the subscript ${}_{[m]}$ denotes the coefficient of $\epsilon^m$ in the power series expansion of the given quantity. By the inductive hypothesis $\alpha_{1,k}$ is closed and $\alpha_{2,k},\alpha_{3,k}$ are manifestly exact.


Now, Proposition \ref{prop:Modified:GH:G2} can be used order by order in $\epsilon$ to show that we can in fact rewrite \eqref{eq:GH:G2:Power:Series:k} in the form of Theorem \ref{thm:Linearisation:GH:G2} by adding functions $f_1,f_2$ and a vector field $X$ as additional free parameters. More precisely, apply Theorem \ref{thm:Linearisation:GH:G2} to find triples $(h_k,\gamma_k,\rho_k)\in C^{l,\alpha}_\nu$ and $(f_1,f_2,X)\in C^{l+1,\alpha}_{\nu+1}$ satisfying
\begin{equation}\label{eq:GH:G2:Power:Series:k:Primed}
\begin{gathered}
d^\ast\gamma_k =0, \qquad d\gamma_k\wedge\omega_0^2=0, \qquad \tfrac{1}{2}dh_k\wedge\omega_0^2-d\gamma_k\wedge\Imag\Omega_0 = \alpha_{1,k},\\
d\rho_k + \tfrac{3}{4}dh_k\wedge\Real\Omega_0 + d\gamma_k\wedge\omega_0 + d\!\ast\!d(f_1\omega_0) =\alpha_{2,k},\\
d\hat{\rho}_k+\tfrac{1}{4}dh_k\wedge\Imag\Omega_0 + d\!\ast\!d(X\lrcorner\Real\Omega_0 + f_2\omega_0)=\alpha_{3,k}.
\end{gathered}\tag{\ref*{eq:GH:G2:Power:Series:k}${}^\prime$}	
\end{equation}
and all the additional algebraic constraints. Now, if $\epsilon$ is small enough the $3$-form $\Real\Omega'$ obtained by truncating the power series for $\Real\Omega_\epsilon$ in \eqref{eq:Power:Series:Data} to order $k$ is a stable $3$-form with Hitchin dual $\Imag\Omega'$ that coincides modulo $\epsilon^{k+1}$ with the truncation to order $k$ of the power series in \eqref{eq:Power:Series:Dual}. Define in a similar way a function $h'$ and a closed integral $2$-form $\kappa'$ truncating the power series of $h_\epsilon$ and $\epsilon\, d\theta_\epsilon$ in \eqref{eq:Power:Series:Data} to order $k$. Now, since $\alpha_{0,k}, \dots, \alpha_{3,k}$ are given by \eqref{eq:GH:G2:Power:Series:k:RHS:0}, the inductive hypothesis and \eqref{eq:GH:G2:Power:Series:k:Primed} imply that $(\omega_0,\Omega',h',\kappa')$ and $(u,v,X)=(-\epsilon^k\,f_1,-\epsilon^k\,f_2,-\epsilon^k\,X)$ satisfy all the assumptions of Proposition \ref{prop:Modified:GH:G2} (including the algebraic constraints for $(\omega_0,\Omega')$ to be an \sunitary{3}--structure) modulo $\epsilon^{k+1}$. In particular, the torsion constraints of Proposition \ref{prop:Torsion:SU(3):structures} also hold modulo $\epsilon^{k+1}$ and the calculations in the proof of Proposition \ref{prop:Modified:GH:G2} show that necessarily $f_1,f_2,X=0$.

\subsection{Convergence}
\label{sec:ac:convergence}

Given the linear estimate of Theorem \ref{thm:Linearisation:GH:G2}, the proof that the formal power series $\varphi_\epsilon$ constructed in the previous subsection converges for $\epsilon$ small enough follows the arguments of Kodaira--Nirenberg--Spencer in their proof of the existence of analytic deformations of complex structures \cite[\S 5]{Kodaira:Nirenberg:Spencer}, \cf also \cite[\S 5.3.(c)]{Kodaira}. Here we sketch the proof that the power series $\varphi_\epsilon$ converges in $C^{1,\alpha}_\nu$. The only slight complication with respect to Kodaira--Nirenberg--Spencer's argument is that the first-order term of the expansion $(h_1,\gamma_1,\rho_1)\notin C^{1,\alpha}_\nu$ (but any non-linear expression involving these quantities decays at least as fast as $O(r^{-2})$ and therefore does lie in $C^{1,\alpha}_\nu$---this is the reason for choosing $\nu>-2$).

Set $\mathcal{C}=\Omega^0 \times\Omega^1\times\Omega^3_{12}$ and for $k\geq 2$ write $\varphi_k$ for the triple $(h_k,\gamma_k,\rho_k)\in \mathcal{C}$. The iteration procedure described in the previous subsection can be summarised schematically as follows: for each $k\geq 1$ we can iteratively determine $\varphi_{k+1}$ from $\varphi_1,\dots, \varphi_{k}$ in a unique way by solving an equation of the form
\[
\mathcal{L}\left( \varphi_{k+1} \right) = \Theta_{k+1},
\]
where $\mathcal{L}$ is a linear first-order operator and $\Theta_{k+1}=\Theta_{k+1}(\varphi_1,\dots,\varphi_{k})$ is an algebraic expression in $\varphi_1,\dots,\varphi_{k}$. Moreover, by Theorem \ref{thm:Linearisation:GH:G2} there exists a uniform constant $C>0$ such that
\begin{equation}\label{eq:Linearisation:GH:G2:1}
\| \varphi_{k+1}\|_{C^{1,\alpha}_\nu} \leq C \| \Theta_{k+1} \|_{C^{1,\alpha}_\nu}.
\end{equation}

In fact $\Theta_{k+1}(\varphi_1,\dots,\varphi_{k})$ is the coefficient of $\epsilon^{k+1}$ in the Taylor series of $\Theta (\epsilon\,\varphi_1 + \dots +\epsilon^{k}\varphi_{k})$, where $\Theta\co \mathcal{C}\ra\mathcal{C}$ is a pointwise real analytic map vanishing at $0$ together with its derivative ($\Theta$ is a composition of the non-linear part of Hitchin's duality map for stable $3$-forms on $6$-manifolds, maps like $f\mapsto (1+f)^a-1$, $a\in\R$, and wedging together differential forms). We therefore conclude that there exists a power series $\sum_{m\geq 2}{C_m\, x^m}$ with radius of convergence $R>0$ such that
\[
|\Theta (\varphi)| \leq \sum_{m\geq 2}{C_m |\varphi|^m}, \qquad |\nabla\Theta (\varphi)| \leq |\nabla\varphi|\sum_{m\geq 2}{m\, C_m |\varphi|^{m-1}}
\]
for all $\varphi\in\mathcal{C}$ with $|\varphi| + |\nabla\varphi|<R$. Using Theorem \ref{thm:Weighted:Embedding} (v), we can therefore estimate, for all $k\geq 1$,
\begin{equation}\label{eq:Power:Series:Estimate:Theta}
\| \Theta_{k+1}\|_{C^{1,\alpha}_\nu} \leq Q\sum_{m=2}^{k+1}{
C_m\left(
\sum_{I\in \mathcal{I}_{m,k}}{
\|\varphi_1\|_{C^{1,\alpha}_{-1}}^{i_1}\|\varphi_2\|_{C^{1,\alpha}_{\nu}}^{i_2}\dots \|\varphi_{k}\|_{C^{1,\alpha}_{\nu}}^{i_{k}}
}
\right)
}
\end{equation}
for some uniform constant $Q>0$. Here $\mathcal{I}_{m,k}$ is the set of multi-indices $I=(i_1,\dots, i_{k})$ such that $i_1+\dots+i_k=m$ and $i_1 + 2\,i_2 + \dots+ k\,i_k=k+1$. In order to see why \eqref{eq:Power:Series:Estimate:Theta} is true, first note that $\|\varphi_1\|_{C^{1,\alpha}_{-1}}<\infty$ since $|\nabla ^k\varphi_1|=O(r^{-1-k})$ for all $k\geq 0$ by \eqref{eq:Asymptotics:rho:tau}. Moreover, in order to justify the application of Theorem \ref{thm:Weighted:Embedding} (v) we need to check that for every $I\in \mathcal{I}_{m,k}$ we have
\[
-i_1 + (i_2+\dots+i_k)\nu = -i_1 + (m-i_1)\nu\leq \nu.
\]
This is the case, since $-i_1 + (m-i_1)\nu =m\nu\leq 2\nu$ if $i_1=0$, $-i_1 + (m-i_1)\nu=-1 + (m-1) \nu\leq -1+\nu$ if $i_1=1$ and $-i_1 + (m-i_1)\nu\leq -2$ if $i_1\geq 2$. Here the choice $\nu >-2$ is crucial (a weak inequality would suffice, but the strict inequality avoids indicial roots).

For constants $b,c>0$ to be determined later consider the power series
\[
A(\epsilon) = \frac{b}{16 c}\sum_{m=1}^\infty{\frac{c^m}{m^2} \epsilon^m}.
\]
It has the crucial property \cite[Equation (5.116)]{Kodaira}
\begin{equation}\label{eq:A(epsilon):property}
\left( A(\epsilon)^m\right)_{[k]} \leq \left( \frac{b}{c}\right) ^{m-1} A_k
\end{equation}
for all integers $m,k\geq 1$. Here $\left( A(\epsilon)^m\right)_{[k]}$ denotes the coefficient of $\epsilon^k$ in the power series $A(\epsilon)^m$. We now introduce the notation $\|\varphi_1\| = \|\varphi_{1}\|_{C^{1,\alpha}_{-1}}$ and $\|\varphi_k\| = \|\varphi_{k}\|_{C^{1,\alpha}_\nu}$ for all $k\geq 2$. We are going to prove that $\|\varphi_k\|\leq A_k$ by induction on $k\geq 1$. Since $A(\epsilon)$ has a positive radius of convergence this will show that the power series $\varphi(\epsilon) = \sum_{k\geq 2}{\varphi_k \epsilon^k}$ converges in $C^{1,\alpha}_\nu$ for $\epsilon$ small enough.

For the base of the induction, choose $b$ so that
\[
\|\varphi_1\|_{C^{1,\alpha}_{-1}} \leq A_1 = \frac{b}{16}.
\]
For $k\geq 1$, \eqref{eq:Linearisation:GH:G2:1}, \eqref{eq:Power:Series:Estimate:Theta} and the induction hypothesis yield
\[
\|\varphi_{k+1}\| \leq C\|\Theta_{k+1}\|_{C^{1,\alpha}_\nu} \leq CQ \sum_{m=2}^{k+1}{C_m \left( \sum_{I\in\mathcal{I}_{m,k}}{A_1^{i_1}\dots A_k^{i_k}}\right)} = CQ \sum_{m=2}^{k+1}{C_m \left( A(\epsilon)^m\right)_{[k+1]}}.
\]
Property \eqref{eq:A(epsilon):property} therefore implies
\[
\|\varphi_{k+1}\| \leq CQ \left( \sum_{m=2}^{k+1}{C_m \left( \frac{b}{c}\right)^{m-1}}\right) A_{k+1}.
\] 
In order to conclude that $\|\varphi_{k+1}\| \leq A_{k+1}$ we now choose $c$ sufficiently large: the choice of $c$ can be made independent of $k$ since the power series $\sum_{m\geq 2}{C_m x^{m-1}}$ has a positive radius of convergence.

The proof of Theorem \ref{thm:GH:G2:Existence} with $l=1$ is now complete. The case $l>1$ is similar.

\begin{remark*}
In Theorem \ref{thm:GH:G2:Existence} we choose $\nu\in (-2,-1)$ while in Theorem \ref{thm:Linearisation:GH:G2} we are allowed to take $\nu \in (-3-\delta,-1)$ for some $\delta>0$ sufficiently small. In fact, pushing the linear analysis to the case $\nu<-3$, \ie the $L^2$--range, required most of the work in the proof of Proposition \ref{prop:Laplacian:2Forms:AC:CY}. The full strength of Theorem \ref{thm:Linearisation:GH:G2} could be used to derive sharper asymptotics for the metrics produced by Theorem \ref{thm:Main:Theorem:technical}. A finite number of iterations of our recursive method to solve the Apostolov--Salamon equations as a power series in $\epsilon$ makes the torsion of the approximate solution decay faster than $r^{-4}$. Using Theorem \ref{thm:Linearisation:GH:G2} with $\nu \in (-3-\delta,-3)$ then shows that all subsequent corrections to the approximate solution decay faster than $r^{-3}$. 
\end{remark*}

\begin{remark*}
Note that the proof of Theorem \ref{thm:GH:G2:Existence} did not use the fact that $c_1 (M)\neq 0$. If $M=B \times{S}^1$ is endowed with the trivial connection $\theta$ and $\rho\in \Omega^3_{12}$ is a closed and coclosed $3$-form of rate $\nu\leq -1$ then our proof shows that for $\epsilon$ sufficiently small $(\omega_0, \Real\Omega_0 + \epsilon \rho)$ can be perturbed to an AC Calabi--Yau structure.
\end{remark*}

\section{Examples from crepant resolutions of Calabi--Yau cones}
\label{sec:examples}

In this final section we apply Theorem \ref{thm:Main:Theorem:technical} to produce examples of complete ALC \gtmfd s. A rich theory of AC Calabi--Yau manifolds has been developed in recent years thanks to the work of many authors. Theorem \ref{thm:Main:Theorem:technical} allows us to exploit these advances in Calabi--Yau geometry to obtain non-trivial results about \gtmfd s.

\subsection{Crepant resolutions of Calabi--Yau cones}

Let $\tu{C}(\Sigma)$ be a Calabi--Yau cone with conical Calabi--Yau structure $(\omega_\tu{C},\Omega_\tu{C})$. In the wide spectrum of AC Calabi--Yau manifolds asymptotic to $\tu{C}(\Sigma)$, two extreme classes of examples arise from affine smoothings (by \cite[Theorem 3.1]{VanCoevering:Examples} every K\"ahler cone is biholomorphic to an affine variety) and crepant resolutions of $\tu{C}(\Sigma)$. Here a crepant resolution $\pi\co B\ra\tu{C}(\Sigma)$ is a resolution of singularity with trivial canonical bundle; in other words, $B$ is a smooth complex $3$-fold isomorphic to $\tu{C}(\Sigma)$ outside the exceptional locus $\pi^{-1}(o)$, where $o$ is the vertex of $\tu{C}(\Sigma)$, and $\pi^\ast\Omega_{\tu{C}}$ extends to a holomorphic complex volume form $\Omega_0$ on $B$. In view of Theorem \ref{thm:Main:Theorem:technical}, we look for non-trivial circle bundles over AC Calabi--Yau $3$-folds. In general affine smoothings of $3$-dimensional isolated singularities have vanishing second cohomology group. For example, this is always the case for hypersurface singularities by a well-known result of Milnor \cite[Theorem 6.5]{Milnor} and for complete intersections by \cite[Satz 1.7 (iv)]{Hamm}. In this section we therefore consider the class of AC Calabi--Yau manifolds given by crepant resolutions of Calabi--Yau cones.

The following theorem summarises the known existence and uniqueness results for AC Calabi--Yau structures on crepant resolutions of Calabi--Yau cones. We state the result in complex dimension $3$ but an analogous statement holds in every complex dimension.  

\begin{theorem}\label{thm:AC:CY:crepant:resolutions}
Let $\tu{C}(\Sigma)$ be a $3$-dimensional Calabi--Yau cone with Calabi--Yau cone structure $(\omega_{\tu{C}},\Omega_{\tu{C}})$ and metric $g_\tu{C}$. Let $\pi\co B\ra\tu{C}(\Sigma)$ be a crepant resolution with complex volume form $\Omega_0$ extending $\pi^\ast\Omega_{\tu{C}}$. Then in every cohomology class containing K\"ahler metrics there exists a unique AC K\"ahler Ricci-flat metric $\omega_0$ on $B$ with $\frac{1}{6}\omega_0^3 = \frac{1}{4}\Real\Omega_0 \wedge\Imag\Omega_0$. Moreover, $(B,\omega_0,\Omega_0)$ is asymptotic to the Calabi--Yau cone $\tu{C}(\Sigma)$ with rate $-6$ if the K\"ahler class $[\omega_0]$ is compactly supported and with rate $-2$ otherwise. 
\end{theorem}

The theorem is the combination of work of various people. The existence in the special case of ALE manifolds was established by Joyce in \cite{Joyce:ALE:CY}, in the case of compactly supported K\"ahler classes by van Coevering \cite[Theorem 1.2]{VanCoevering:Existence} and in the general case by Goto \cite[Theorem 5.1]{Goto:Crepant}. The optimal uniqueness statement in Theorem \ref{thm:AC:CY:crepant:resolutions} is due to Conlon--Hein \cite[Theorem 3.1]{Conlon:Hein:I}.

Theorem \ref{thm:AC:CY:crepant:resolutions} reduces the problem of constructing AC Calabi--Yau metrics to the construction of Calabi--Yau cones and the classification of their K\"ahler crepant resolutions. For example, the existence of Calabi--Yau cone metrics in the toric setting was completely understood by Futaki--Ono--Wang \cite{Futaki:Ono:Wang} and toric crepant resolutions of a toric singularity correspond to nonsingular subdivisions of the fan of the singular toric variety. Exploiting his result on the existence of AC K\"ahler Ricci-flat metrics with compactly supported K\"ahler classes, van Coevering constructs infinitely many examples of toric AC Calabi--Yau $3$-folds \cite[Theorem 1.3]{VanCoevering:Examples}, in fact at least one such example asymptotic to each Gorenstein toric K\"ahler cone with an isolated singularity.

Existence of Calabi--Yau cone metrics, or equivalently, Sasaki--Einstein manifolds, is a difficult problem. In the regular and quasi-regular cases, the problem is equivalent to the existence of K\"ahler--Einstein (orbifold) metrics with positive scalar curvature. Many examples were obtained exploiting calculations of $\alpha$--invariants, \cf for example \cite{Boyer:Galicki:JDG} and \cite{Boyer:Galicki:Kollar}. The first known examples of irregular Sasaki--Einstein manifolds were constructed by Gauntlett--Martelli--Sparks--Waldram \cite{GMSW} and are completely explicit. We refer to the survey paper \cite{Sparks:SE} and the monograph \cite{Boyer:Galicki}, in particular Chapter 11, for further details on all these constructions.

Very recently, Collins--Sz\'ekelyhidi \cite{Collins:Szekelyhidi:GT} proved that the existence of a Calabi--Yau cone metric on $\tu{C}(\Sigma)$ is implied by K-stability. In the conical setting, K-stability is an algebro-geometric notion for the affine variety $\tu{C}(\Sigma)\subset \C^N$ with an isolated singularity at the origin together with a holomorphic $(\C^\ast)^m$--action generated by a vector field $\xi$ which acts with positive weights on the coordinate functions of $\C^N$. The definition of K-stability involves all possible degenerations of $(\tu{C}(\Sigma),\xi)$ and is therefore difficult to check in practice. However, in cases with a large automorphism group only \emph{equivariant} degenerations need be considered and therefore checking K-stability can be reduced to combinatorial calculations. For example, in the toric case every degeneration must be isomorphic to $\tu{C}(\Sigma)$ itself and the result of \cite{Futaki:Ono:Wang} is recovered in this way. Collins--Sz\'ekelyhidi also consider certain new examples with a complexity-one torus action, \ie $n$-dimensional examples acted upon by an $(n-1)$--torus. In the next subsection we use these new examples of $3$-dimensional Calabi--Yau cones to construct infinitely many ALC \gtmfd s. 

\subsection{Examples from small resolutions}

Given a Calabi--Yau cone $(\tu{C}(\Sigma),\omega_\tu{C},\Omega_\tu{C})$, let $(B,\omega_0,\Omega_0)$ be a crepant resolution of $\tu{C}(\Sigma)$ endowed with an AC Calabi--Yau structure as in Theorem \ref{thm:AC:CY:crepant:resolutions}. Let $M\ra B$ be a non-trivial circle bundle. In general it is not straightforward to check condition \eqref{eq:sigma_tau:orthogonality:1}, \ie $c_1 (M)\cup [\omega_0]=0\in H^4(B)$. In this subsection we consider examples where \eqref{eq:sigma_tau:orthogonality:1} is automatically satisfied because $H^4(B)=0$. By \cite[Theorem 5.2]{Caibar} this is the case if and only if $B\ra \tu{C}(\Sigma)$ is a \emph{small} resolution, \ie the exceptional set of the resolution has no divisorial component.
 
\begin{example}\label{eg:Small:res:conifold}
The conifold $\tu{C}(\Sigma) = \{ z_1^2+\dots +z_4^2=0\}\subset \C^4$ admits an explicit Calabi--Yau cone metric: the corresponding Sasaki--Einstein $5$-manifold $\Sigma$ is $S^2\times S^3 = \sunitary{2}\times \sunitary{2}/\triangle \unitary{1}$ endowed with a homogeneous Sasaki--Einstein metric. The conifold admits two small resolutions related by a flop and both isomorphic to the rank $2$ complex vector bundle $\mathcal{O}(-1)\oplus \mathcal{O}(-1)\ra \CP^1$. Let $B$ be one such small resolution. Note that $H^2(B)\simeq H^2(\Sigma)$ is $1$-dimensional and $H^4(B)=0$. Candelas--de la Ossa \cite{Candelas:delaOssa} constructed an explicit cohomogeneity one AC Calabi--Yau structure $(\omega_0,\Omega_0)$ on $B$, unique up to scale. Hence Theorem \ref{thm:Main:Theorem:technical} yields the existence of a $1$-parameter family of ALC \gtmetric s up to scale (and orientation) on the unique simply connected circle bundle $M\ra B$. Because of the uniqueness modulo diffeomorphisms of our construction, the ALC \gtmetric s produced by Theorem \ref{thm:Main:Theorem:technical} in this case must in fact admit an isometric cohomogeneity one action of $\sunitary{2}\times\sunitary{2}\times\unitary{1}$ and therefore they must belong to the $\mathbb{D}_7$ family of cohomogeneity one ALC \gtmfd s studied numerically by Cveti\v{c}--Gibbons--L\"u--Pope \cite{CGLP:M:Conifolds}.
\end{example}

Let $\tu{C}(\Sigma)$ be a $3$-dimensional Calabi--Yau  cone admitting a crepant resolution. By \cite[Proposition 3.5]{VanCoevering:Examples}, $\tu{C}(\Sigma)$ is then a Gorenstein canonical singularity. If $\tu{C}(\Sigma)$ admits a \emph{small} resolution then $\tu{C}(\Sigma)$ must be a terminal singularity and therefore \cite[Main Theorem (I)]{Reid:Pagoda} a so-called \emph{compound Du Val} singularity: $\tu{C}(\Sigma)$ is a $3$-fold hypersurface singularity of the form $\{ f(x,y,z)+tg(x,y,z,t)=0\} \subset \C^4$, where $\{ f(x,y,z)=0\}\subset\C^3$ defines a Du Val singularity. Equivalently, the generic hyperplane section of $\tu{C}(\Sigma)$ is a Du Val singularity. If this Du Val singularity is of type $W_n=A_n$, $D_n$ or $E_6, E_7$ or $E_8$ we say that $\tu{C}(\Sigma)$ is a compound $W_n$ singularity, or $cW_n$ singularity for short. In \cite[\S 8]{Pinkham} Pinkham observes that it is possible to construct small resolutions of compound Du Val singularities as deformations of partial resolutions of Du Val singularities. This approach is pursued by Katz \cite{Katz}, who gives a complete classification of the $cA_n$ and $cD_4$ singularities that admit small resolutions. In the $cA_n$ case partial results were in fact known to various other authors, \cf for example Friedman \cite[\S 2]{Friedman} and, in very special cases, Brieskorn \cite[Satz 0.2]{Brieskorn}.

\begin{theorem}\label{thm:cAp}
For $p\geq 1$ consider the compound $A_p$ singularity $X_p\subset \C^4$ defined by the equation
\begin{equation}\label{eq:cAp}
x^2 + y^2 + z^{p+1} - w^{p+1}=0.
\end{equation}
\begin{enumerate}
\item $X_p$ admits a Calabi--Yau cone metric with Reeb vector field $\xi$ acting on $\C^4$ with weights
\[
\tfrac{3}{4}\left( p+1, p+1, 2, 2\right).
\]
Thus $X_p = \tu{C}(\Sigma)$ is the Calabi--Yau cone over a quasi-regular Sasaki--Einstein structure on $\Sigma \simeq \#_{p} \left( S^2\times S^3\right)$.
\item $X_p$ admits a K\"ahler small resolution $B$. The exceptional set is a chain of $p$ rational curves meeting transversely. Moreover, $B$ is simply connected and its Betti numbers are
\[
b_0 (B)=1, \qquad b_1 (B)=b_3(B)=b_4(B)=b_5(B)=b_6 (B)=0, \qquad b_2 (B) = p.
\]
In particular $B$ admits a $p$-parameter family of Calabi--Yau structures $(\omega_0,\Omega_0)$ asymptotic to the Calabi--Yau cone $X_p = \tu{C}(\Sigma)$ with rate $-2$.
\end{enumerate}
\proof
The existence of a Calabi--Yau cone metric on $X_p$ with the given Reeb vector field follows from \cite[Theorem 7.1 (I)]{Collins:Szekelyhidi:GT}: Collins-Sz\'ekelyhidi consider more generally Brieskorn--Pham singularities of the form
\[
x^2 + y^2 + z^h - w^k=0
\]
for integers $2\leq h \leq k$. They show that, with the most obvious choice of Reeb vector field \cite[\S 3]{GMSY}, this Brieskorn--Pham singularity is K-stable (and therefore admits a Calabi--Yau metric by their existence result \cite[Theorem 1.1]{Collins:Szekelyhidi:GT}) if and only if $k < 2h$. In fact, the necessity of this condition follows from the Lichnerowicz obstruction of Gauntlett--Martelli--Sparks--Yau \cite[\S 3.2]{GMSY}. However, the existence of a Calabi--Yau cone metric on these singularities was previously known only when $h=k=2$ (the conifold) and $h=2, k=3$ (thanks to work of Li--Sun \cite[\S 5.2]{Li:Sun} on K\"ahler--Einstein metrics with cone singularities along divisors). The fact that the link $\Sigma$ of the singularity $X_p$ is diffeomorphic to $\#_{p} \left( S^2\times S^3\right)$ is well known, \cf \cite[Table B.4.4]{Boyer:Galicki}.

Brieskorn \cite[Satz 0.2]{Brieskorn} has shown that a singularity of the form $x^2 + y^2 + z^{h} - w^{k}=0$, $2\leq h\leq k$, has a small resolution if and only if $k$ is a multiple of $h$. More generally, Katz \cite[Theorem 1.1]{Katz} has shown that a $cA_p$ singularity admits a small resolution if and only if the singularity is given by $x^2 + y^2 + g(z,w)=0$ with $g(z,w)=0$ an isolated plane curve singularity with exactly $p+1$ smooth branches. Note that $z^{p+1} - w^{p+1} = \prod_{j=0}^{p}(z-\zeta^{j}w)$, where $\zeta = e^{\frac{2i\pi}{p+1}}$. A small resolution $B$ of $X_p$ is then obtained by blowing-up $X_p$ along $p$ planes of the form $x+iy=0=z-\zeta^j w$ (or $x-iy=0=z-\zeta^j w$). This shows that $B$ is K\"ahler since it can be realised as the complement of a divisor in a blow-up of a weighted projective space. 

Katz has shown that the exceptional locus of the resolution $B\ra X_p$ is a chain of $p$ rational curves meeting transversely \cite[Theorem 1.1]{Katz}. The topology of $B$ can be determined using \cite[Theorems 5.2 and 7.2]{Caibar} or more directly exploiting the fact that the $\C^\ast$--action on $X_p$ induces a retraction of $B$ onto the exceptional set of the resolution.

Let $\Omega_0$ be the holomorphic complex volume form on any such small resolution $B$ obtained by pulling-back the conical complex volume form $\Omega_\tu{C}$ on $X_p$. By Theorem \ref{thm:AC:CY:crepant:resolutions}, $B$ admits a family of AC Calabi--Yau structures $(\omega_0,\Omega_0)$ parametrised by the space of classes in $H^2(B)$ containing a K\"ahler metric. In Section \ref{sec:Gauge:Fixing} we have already noted that the ``K\"ahler cone'' of an AC Calabi--Yau $3$-fold is an open subset of the second cohomology. Since $H^2_c (B)\simeq H_4 (B)=0$, the K\"ahler class of an AC Calabi--Yau structure $(\omega_0,\Omega_0)$ on any small resolution is never compactly supported and therefore 
by the final statement in Theorem \ref{thm:AC:CY:crepant:resolutions},
$(\omega_0,\Omega_0)$ decays to the conical Calabi--Yau structure on $X_p=\tu{C}(\Sigma)$ with rate $-2$.
\endproof
\end{theorem}
 
\begin{corollary}\label{cor:cAp:ALC:G2}
Let $B$ be a small resolution of the compound Du Val singularity \eqref{eq:cAp} for $p\geq 2$. Let $M\ra B$ be a principal circle bundle. By passing to a finite cover we can assume that $c_1 (M)$ is a primitive element in the lattice $H^2(B;\Z)$ so that $M$ is simply connected.

Then $M$ carries a $p$-dimensional family of complete ALC \gtmetric s up to scale. In particular, there exist families of ALC \gtmfd s of arbitrarily high dimension. Moreover, for $p,p'\geq 2$ with $p\neq p'$ the ALC \gtmfd s $M$ and $M'$ constructed in this way are not diffeomorphic. In particular, there exists infinitely many diffeomorphism types of simply connected complete non-compact \gtmfd s.
\proof
The existence of complete ALC \gtmetric s follows by applying Theorem \ref{thm:Main:Theorem:technical} to the AC Calabi--Yau manifolds produced by Theorem \ref{thm:cAp} (ii). The fact that $H^4(B)=0$ implies that the constraint $c_1 (M)\cup [\omega_0]=0\in H^4(B)$ is trivially satisfied.

Since $H^3(B)=0=H^4(B)$, the Gysin sequence for the circle fibration $M\ra B$ shows that integration along the circle fibres yields an isomorphism between $H^3(M)$ and $H^2(B)$. Since the image of $[\varphi_\epsilon]$ in $H^2(B)$ is $\epsilon [\omega_0]$ we see that ALC \gtmetric s arising in Theorem \ref{thm:Main:Theorem:technical} from AC Calabi--Yau structures on $B$ with different K\"ahler classes cannot be equivalent under a diffeomorphism isotopic to the identity.

In order to show that different choices of $p, p'$ give rise to non-diffeomorphic $7$-manifolds, we use the computation of the Betti numbers of $B$ in Theorem \ref{thm:cAp} (ii) together with the Gysin sequence of the circle fibration $M\ra B$. If $M$ is a non-trivial circle bundle over a small resolution of the $cA_p$ singularity $X_p$ then $b_0 (M)=1$
\[
b_2 (M) = p-1, \qquad b_3 (M)=p,
\]
and all other Betti numbers vanish.
\endproof
\end{corollary}

\begin{remark*}\label{rmk:Topology:end:ALC:small:res}
While it might be possible to determine fully the diffeomorphism type of $M$, here we content ourselves with the determination of the diffeomorphism type of its end. The end of $M$ is diffeomorphic to $(R,\infty)\times N$, where $N$ is the total space of a non-trivial circle bundle over $\Sigma \simeq \#_{p} \left( S^2\times S^3\right)$. The Gysin sequence for this fibration shows that $b_1(N)=0$, $b_2 (N)=p-1$ and $b_3(N)=2p$. In order to determine fully the diffeomorphism class of $N$ we make the additional assumption that the image of $c_1 (M)$ in $\text{im}\,H^2 (B,\Z)\ra H^2(\Sigma,\Z)$ is primitive (an assumption that can always be achieved by passing to a finite cover). Then $N$ has no torsion in cohomology. Moreover, since $N$ is a hypersurface in a {\gtmfd} it is spin. Hence we can use the classification of such manifolds given by Wall \cite[Theorem 5]{Wall}: the diffeomorphism class of $N$ is determined by the symmetric trilinear map $\mu\co H^2(N,\Z)\times H^2(N,\Z) \times H^2(N,\Z) \ra \Z$ given by cup product and the first Pontrjagin class, which can be regarded as a linear map $p_1\co H^2(N,\Z) \ra \Z$ via cup product. Now, the Gysin sequence for the circle bundle $\pi\co N\ra \Sigma$ shows that $H^2(N)\simeq \pi^\ast H^2(\Sigma)$, hence $\mu=0$ since $\Sigma$ is $5$-dimensional. On the other hand, since $N$ is a principal circle bundle over $\Sigma$ and therefore the vertical tangent bundle of $\pi$ is trivial, we have $TN = T\Sigma \oplus \underline{\R}$. Hence $p_1(N)=\pi^\ast p_1 (\Sigma)=0$ since $H^4(\Sigma)=0$. Wall's result then implies that $N$ is diffeomorphic to the connected sum of $p-1$ copies of $S^2\times S^4$ and $p$ copies of $S^3\times S^3$.
\end{remark*}

\subsection{Other examples}

Let $(B,\omega_0,\Omega_0)$ be an AC Calabi--Yau structure on a crepant resolution of a Calabi--Yau cone $\tu{C}(\Sigma)$ and assume that $H^4(B)\neq 0$. Set $d=\tu{dim}\,H^2_{c}(B)$ and $b=\tu{dim}\,H^2(\Sigma)$. As pointed out in \cite[Equation (60)]{VanCoevering:Examples}, we have an exact sequence $0\ra H^2_c(B)\ra H^2(B)\ra H^2(\Sigma)\ra 0$ and therefore, given a choice of bases, the topological condition $c_1(M)\cup [\omega_0]=0$ in Theorem \ref{thm:Main:Theorem:technical} is a system of $d$ linear equations in $d+b$ variables and we are interested in its integer solutions. Note that by Poincar\'e duality the coefficients of the system are determined by triple intersections of divisors on $B$ (compact and non-compact ones, with at least one compact divisor in each triple). Toric examples show that the system $\kappa\cup [\omega_0]=0\in H^4(B)$ can be highly degenerate, but we expect that Theorem \ref{thm:Main:Theorem:technical} can be applied to a wide class of crepant resolutions of Calabi--Yau cones yielding a plethora of new complete \gtmetric s, many of which are potentially defined on the same underlying smooth $7$-manifold. In fact, an extension of the construction of this paper to the case where $B$ is an AC Calabi--Yau \emph{orbifold} yields infinitely many distinct families of ALC \gtmetric s on a manifold as simple as $S^3\times \R^4$ \cite[Theorem 4.12]{Foscolo:ALC:Spin7}.

We conclude the paper by considering a concrete example of an AC Calabi--Yau $B$ with $b_4(B)=1$. Set $D=\CP^1\times\CP^1$ and consider the canonical line bundle $B=K_D$. Blowing down the zero-section exhibits $B$ as a crepant resolution of the Calabi--Yau cone $\tu{C}(\Sigma)$ over a free $\Z_2$--quotient $\Sigma$ of the homogeneous Sasaki--Einstein structure on $S^2\times S^3$. By Theorem \ref{thm:AC:CY:crepant:resolutions} the complex manifold $B$ carries a $2$-parameter family of AC Calabi--Yau structures. In fact these metrics are all invariant under a cohomogeneity one action of $\sunitary{2}\times\sunitary{2}$ and 
have been determined explicitly \cite[\S C]{Zayas:et:al}.

Now, $H^2(B)\simeq H^2(D)$ is generated over the integers by the classes $[\omega_1]$ and $[\omega_2]$ of the Fubini--Study metrics of the two $\CP^1$--factors in $D$. We therefore denote cohomology classes on $B$ by pairs of real numbers. A cohomology class $(a_1,a_2)$ is a K\"ahler class if and only if $a_1, a_2 >0$. If $c_1 (M) = (\alpha_1,\alpha_2)\in \Z^2$ 
and $[\omega_0]=(a_1,a_2)$ then $c_1(M)\cup [\omega_0]=0$ if and only if
\[
a_1 \alpha_2 + a_2\alpha_1 =0.
\]
Therefore Theorem \ref{thm:Main:Theorem:technical} yields the following result.

\begin{theorem}\label{thm:P1xP1}
Given a pair of coprime integers $m,n$ consider the simply connected $7$-manifold $M$ which is the total space of the principal circle bundle over $K_{\CP^1\times\CP^1}$ with first Chern class $(m,-n)$. Then provided $mn> 0$, for every $a\in \R_+$ the $7$-manifold $M$ carries a $1$-parameter family $\{ g_\epsilon\}_{0<\epsilon<\epsilon_0}$ of ALC \gtmetric s that collapses to $K_{\CP^1\times\CP^1}$ endowed with the AC Calabi--Yau metric with K\"ahler class $a(|m|,|n|)$ as $\epsilon\ra 0$.
\end{theorem}

Note that scaling acts simultaneously on $\epsilon$ and $a$ and therefore Theorem \ref{thm:P1xP1} yields a $1$-parameter family of ALC \gtmetric s up to scale on $M$. By uniqueness in our construction the metrics $g_\epsilon$ must in fact be invariant under the cohomogeneity one action of $\sunitary{2}\times\sunitary{2}\times\unitary{1}$ on $M$. In \cite{FHN:Coho1:ALC} we use cohomogeneity one methods to study these families away from the highly collapsed regime of Theorem \ref{thm:P1xP1}. Based on numerical experiments, Cveti\v{c}--Gibbons--L\"u--Pope \cite[\S 4]{CGLP:C7:tilde} argue for the existence of a continuous $2$-parameter family up to scale of cohomogeneity one ALC \gtmetric s collapsing to $K_{\CP^1\times\CP^1}$. This contradicts Theorem \ref{thm:P1xP1}, which only guarantees the existence of countably many $1$-parameter families up to scale of such metrics. However, we note that the authors of \cite{CGLP:C7:tilde} are not very careful about the finite group stabiliser of the principal orbits: the fact that this is non-trivial seems to force some of the parameters determining initial conditions for their numerical solutions to vanish.   

\appendix

\section{Homogeneous harmonic forms on Riemannian cones}
\label{Appendix:homogenous:harmonic}

In this first appendix we study harmonic forms on a Riemannian cone. We give general statements that work in every dimension. These results were first obtained by Cheeger \cite{Cheeger:PNAS,Cheeger:cones, Cheeger:unpublished} as degenerate limiting cases of his more general analysis of eigenforms of the Laplacian on Riemannian cones. For completeness we give a direct proof of the results about harmonic forms that we use in the paper.

Let $\tu{C}=\tu{C}(\Sigma)$ be an $n$-dimensional Riemannian cone over a smooth compact Riemannian $(n-1)$-manifold $\Sigma$. The following lemma is the result of a straightforward computation.

\begin{lemma}\label{lem:log:homogeneous:harmonic}
Let $\gamma = r^{\lambda+k}\left( \frac{dr}{r}\wedge\alpha + \beta\right)$ be a $k$-form on $\tu{C}$ homogeneous of order $\lambda$. For every function $u=u(r)$ we have $\triangle (u\gamma) = r^{\lambda+k-2}\left( \tfrac{dr}{r}\wedge A + B\right)$, where
\[
\begin{gathered}
A = u\Big( \triangle\alpha - (\lambda+k-2)(\lambda+n-k)\alpha -2d^\ast\beta \Big) -r\dot{u}\left( 2\lambda +n-1\right) \alpha - r^2 \ddot{u}\,\alpha,\\
B = u\Big( \triangle\beta - (\lambda+n-k-2)(\lambda+k)\beta -2d\alpha \Big) -r\dot{u}\left( 2\lambda +n-1\right) \beta - r^2 \ddot{u}\,\beta.
\end{gathered}
\]
\end{lemma}

\begin{theorem}\label{thm:Harmonic:forms:cone:n}
Let $\gamma = r^{\lambda+k}\left( \frac{dr}{r}\wedge\alpha + \beta\right)$ be a harmonic $k$-form on $\tu{C}$ homogeneous of order $\lambda$. Then $\gamma$ decomposes into the sum of homogeneous harmonic forms $\gamma = \gamma_1 + \gamma_2 + \gamma_3 + \gamma_4$ where $\gamma_i = r^{\lambda+k}\left( \frac{dr}{r}\wedge\alpha_i + \beta_i\right)$ satisfies the following conditions.
\begin{enumerate}
\item $\beta_1=0$ and $\alpha_1$ satisfies $d\alpha_1=0$ and $\triangle\alpha_1 = (\lambda+k-2)(\lambda+n-k)\alpha_1$.
\item $(\alpha_2,\beta_2) \in\Omega^{k-1}_{coexact}\times\Omega^k_{exact}$ satisfies the first-order system
\[
d\alpha_2 = (\lambda+k)\beta_2, \qquad d^\ast\beta_2 = (\lambda+n-k)\alpha_2.
\]
In particular, if $(\alpha_2,\beta_2)\neq 0$ then $\lambda+k\neq 0\neq \lambda+n-k$ and the pair $(\alpha_2,\beta_2)$ is uniquely determined by either of the two factors, which is a coexact/exact eigenform of the Laplacian with eigenvalue $(\lambda+k)(\lambda+n-k)$.
\item $(\alpha_3,\beta_3) \in\Omega^{k-1}_{coexact}\times\Omega^k_{exact}$ satisfies the first-order system
\[
d\alpha_3 + (\lambda+n-k-2)\beta_3=0=d^\ast\beta_3 + (\lambda+k-2)\alpha_3.
\]
In particular, if $(\alpha_3,\beta_3)\neq 0$ then $\lambda+k-2\neq 0\neq \lambda+n-k-2$ and the pair $(\alpha_3,\beta_3)$ is uniquely determined by either of the two factors, which is a coexact/exact eigenform of the Laplacian with eigenvalue $(\lambda+k-2)(\lambda+n-k-2)$.
\item $\alpha_4=0$ and $\beta_4$ satisfies $d^\ast \beta_4=0$ and $\triangle\beta_4 = (\lambda+n-k-2)(\lambda+k)\beta_4$.
\end{enumerate}
The decomposition $\gamma = \gamma_1 + \gamma_2 + \gamma_3 + \gamma_4$ is unique, except when $\lambda = -\frac{n-2}{2}$; in that case forms of type \textup{(ii)} and \textup{(iii)} coincide, and there is a unique decomposition $\gamma = \gamma_1 + \gamma_2 + \gamma_4$.
\proof
By Lemma \ref{lem:log:homogeneous:harmonic}, $\gamma = r^{\lambda+k}\left( \frac{dr}{r}\wedge\alpha + \beta\right)$ is harmonic if and only if
\begin{equation}\label{eq:Harmonic:k:form}
\begin{cases}
\triangle \alpha = (\lambda+k-2)(\lambda+n-k)\alpha + 2d^\ast\beta, & \\
\triangle\beta = (\lambda+n-k-2)(\lambda+k)\beta + 2d\alpha.
\end{cases}
\end{equation}

By Hodge theory on the compact manifold $\Sigma$ write $\alpha=\alpha_c + \alpha_{ce}$ with $\alpha_c$ closed and $\alpha_{ce}$ coexact and $\beta = \beta_{cc}+\beta_e$ with $\beta_{cc}$ coclosed and $\beta_{e}$ exact. Then $\alpha_c,\beta_{cc}$ satisfy decoupled equations
\begin{align*}
\triangle \alpha_{c} &= (\lambda+k-2)(\lambda+n-k)\alpha_c, \\
\triangle\beta_{cc} &= (\lambda+n-k-2)(\lambda+k)\beta_{cc}.
\end{align*}
Hence for any such $\alpha_c$ and $\beta_{cc}$ the pairs $(\alpha_c,0)$ and $(0,\beta_{cc})$ are solutions to \eqref{eq:Harmonic:k:form} and give rise to harmonic forms of type (i) and (iv) respectively.

In the rest of the proof we can therefore assume that $\alpha=\alpha_{ce}$ and $\beta=\beta_e$. Note that \eqref{eq:Harmonic:k:form} can be rewritten as
\begin{equation}\label{eq:Harmonic:k:form:1}
\begin{cases}
\triangle \alpha - (\lambda+k)(\lambda+n-k)\alpha = 2\Big( d^\ast\beta - (\lambda+n-k)\alpha\Big), & \\
\triangle\beta - (\lambda+n-k)(\lambda+k)\beta = 2\Big( d\alpha - (\lambda+k)\beta\Big).
\end{cases}
\end{equation}
Moreover, a straightforward computation using $d^\ast\alpha=0=d\beta$, \eqref{eq:Harmonic:k:form} and the equations obtained by taking $d$ and $d^\ast$ of the first and second equation in \eqref{eq:Harmonic:k:form}, respectively, shows that $d^\ast\beta - (\lambda+n-k)\alpha$ and $d\alpha - (\lambda+k)\beta$ are eigenforms of the Laplacian with eigenvalue $(\lambda+k-2)(\lambda+n-k-2)$.

We now solve \eqref{eq:Harmonic:k:form:1} using the $L^2$--orthogonal decompositions of $\Omega^{k-1}_{coexact}$ and $\Omega^{k}_{exact}$ into eigenspaces for the Laplacian: we write $(\alpha,\beta) = \sum_{\mu}{(\alpha_\mu,\beta_\mu)}$ where $\alpha_\mu$ and $\beta_\mu$ are eigenforms for the Laplacian with eigenvalue $\mu$. Since $d$ and $d^\ast$ commute with the Laplacian, note that $d^\ast\beta_\mu$ and $d\alpha_\mu$ are eigenforms for the Laplacian with eigenvalue $\mu$. Since eigenspaces corresponding to different eigenvalues are $L^2$--orthogonal, $(\alpha_\mu,\beta_\mu)$ must be an independent solution to \eqref{eq:Harmonic:k:form:1} for every eigenvalue $\mu$. On the other hand, since we know that the right-hand side of \eqref{eq:Harmonic:k:form:1} lies in the eigenspace of the Laplacian with eigenvalue $(\lambda+k-2)(\lambda+n-k-2)$ we conclude that $(\alpha_\mu,\beta_\mu)=0$ unless $\mu = (\lambda+k)(\lambda+n-k)$ or $\mu=(\lambda+k-2)(\lambda+n-k-2)$. The two cases correspond to harmonic forms $\gamma_2$ and $\gamma_3$ of type (ii) and (iii) respectively.

The modified statement in the special case $\lambda=-\frac{n-2}{2}$ is clear since then $\lambda+k=-(\lambda +n-k-2)$ and $\lambda+n-k=-(\lambda+k-2)$.  
\endproof
\end{theorem}

\begin{remark}\label{rmk:Harmonic:forms:cone:n}
A straightforward computation shows that $\gamma$ is closed if and only if $d\alpha - (\lambda+k)\beta=0=d\beta$ and coclosed if and only if $d^\ast\beta - (\lambda+n-k)\alpha=0=d^\ast\alpha$. We conclude that, when $\gamma_i\neq 0$,
\begin{enumerate}
\item $\gamma_1$ is always closed, and coclosed if and only if $\lambda+n-k=0$ (in which case $\alpha_1$ is harmonic);
\item $\gamma_2$ is always closed and coclosed (in fact exact and coexact);
\item $\gamma_3$ is closed and coclosed when $\lambda=-\frac{n-2}{2}$, \ie the degenerate case that forms of type (ii) and (iii) coincide, and neither closed nor coclosed otherwise;
\item $\gamma_4$ is always coclosed, and closed if and only if $\lambda+k=0$ (in which case $\beta_4$ is harmonic).
\end{enumerate}
Moreover, $\gamma$ is (co)closed if and only if each component $\gamma_i$ is.
\end{remark}

We also need to consider harmonic $k$-forms on $\tu{C}$ that are expressed as polynomials in $\log{r}$ with coefficients in the space of $k$-forms on the cone homogeneous of order $\lambda$. The structure of the logarithmic terms in the following lemma is consistent with those that appear in Cheeger's discussion of harmonic forms on cones \cite[Equation (2.17)$^-$]{Cheeger:unpublished}.

\begin{prop}\label{prop:log:homogeneous:harmonic}
Let $\gamma = \sum_{j=0}^m{\gamma_j\, (\log{r})^j}$ be a polynomial in $\log{r}$ with coefficients in the space of $k$-forms on $\tu{C}$ homogeneous of order $\lambda$. If $\triangle\gamma =0$ then either $m=0$ (and $\triangle\gamma_0=0$) or $m=1$, $\lambda = -\frac{n-2}{2}$ and $\gamma =\gamma_1\, \log{r} + \gamma_0$ with $\triangle\gamma_1=0=\triangle\gamma_0$.
\proof
Let $\mathcal{L}\co \Omega^{k-1}(\Sigma)\oplus\Omega^k(\Sigma)\ra \Omega^{k-1}(\Sigma)\oplus\Omega^k(\Sigma)$ be the linear elliptic operator defined by
\[
\mathcal{L}(\alpha, \beta) = \Big( \triangle\alpha - (\lambda+k-2)(\lambda+n-k)\alpha -2d^\ast\beta, \triangle\beta - (\lambda+n-k-2)(\lambda+k)\beta -2d\alpha \Big).
\]
Note that $\mathcal{L}$ is self-adjoint.

For all $j=0,\dots,m$ write $\gamma_j = r^{\lambda+k}\left( \frac{dr}{r}\wedge\alpha_j + \beta_j\right)$. By abuse of notation we use the same symbol $\gamma_j$ to denote the pair $(\alpha_j,\beta_j)\in \Omega^{k-1}(\Sigma)\oplus\Omega^k(\Sigma)$. We set $\gamma_j =0$ if $j>m$ or $j<0$.

Now, using Lemma \ref{lem:log:homogeneous:harmonic} with $u=(\log{r})^j$ it is not difficult to see that $\triangle\gamma=0$ if and only if
\begin{equation}\label{eq:log:homogeneous:harmonic}
\mathcal{L}(\gamma_j) = (j+1)(2\lambda +n-2)\,\gamma_{j+1} + (j+1)(j+2)\,\gamma_{j+2}
\end{equation}
for all $j=0,\dots, m$. The choice $j=m$ in this equation shows that $\gamma_m\neq 0$ lies in the kernel of $\mathcal{L}$.

Assume that $m\geq 1$ and consider the equation \eqref{eq:log:homogeneous:harmonic} with $j=m-1$:
\[
\mathcal{L}(\gamma_{m-1}) = m (2\lambda +n-2)\,\gamma_m.
\]
Since $\mathcal{L}$ is self-adjoint, the right-hand-side (that is simultaneously in the image and the kernel of $\mathcal{L}$) must vanish and therefore $2\lambda +n-2=0$. Assume this is the case and that $m\geq 2$. Then \eqref{eq:log:homogeneous:harmonic} with $j=m-2$ is
\[
\mathcal{L}(\gamma_{m-2}) = m(m-1)\,\gamma_m
\]
and therefore leads to a contradiction.
\endproof 
\end{prop}

\begin{prop}\label{cor:log:homogeneous:harmonic}
Let $\gamma = \sum_{j=0}^m{\gamma_j(\log{r})^j}$ be a polynomial in $\log{r}$ with coefficients in the space of differential forms on $\tu{C}$ homogeneous of order $\lambda$. If $(d+d^\ast)\gamma =0$ then $m=0$.
\proof
Since each pure-degree component of $\gamma$ is harmonic, Proposition \ref{prop:log:homogeneous:harmonic} and a straightforward computation imply that either $m=0$ or $m=1$, $\lambda = -\frac{n-2}{2}$ and $\gamma = \gamma_1 \log{r} + \gamma_0$ with harmonic $\gamma_0,\gamma_1$ satisfying the first-order system $(d+d^\ast)\gamma_1=0=(d+d^\ast)\gamma_0 + \frac{1}{r}dr\wedge\gamma_1 -\frac{1}{r}\partial_r\lrcorner\gamma_1$.

Now, since $\gamma_0,\gamma_1$ are homogeneous of order $\lambda$, we can think of them as elements of $\Omega^\ast (\Sigma)\otimes \R^2$. Under this identification $d+d^\ast$ and $\frac{1}{r}dr\wedge\,\cdot\, -\frac{1}{r}\partial_r\lrcorner\,\cdot\,$ are operators $D$ and $S$ from $\Omega^\ast (\Sigma)\otimes \R^2$ into itself with the property that $S^{-1}D$ is self-adjoint. Indeed, one can write
\[
S^{-1}D\, (\alpha, \beta) = (d+d^\ast) (\beta,\alpha) + \Phi (\alpha,\beta)
\]
where $\Phi$ is a constant multiple of the identity on each subspace $\Lambda^k(T^\ast\Sigma)\oplus\{0\}$ and $\{0\}\oplus \Lambda^k(T^\ast\Sigma)$. Here $d+d^\ast$ on the right-hand side is an operator on $\Sigma$. Since $\gamma_1$ is both in the image and kernel of $S^{-1}D$ we conclude that it must vanish.
\endproof
\end{prop}

\section{Analysis on asymptotically conical manifolds}\label{Appendix:Analysic:AC}

In this appendix we collect basic facts about analysis on asymptotically conical manifolds. Weighted Banach spaces are a standard tool to work on such manifolds. While their use is widespread by now, we still feel it is useful to collect here the statements of the main results of the theory to make the paper more self-contained. We refer the reader to \cite{Lockhart:McOwen,Lockhart, Melrose} and \cite[\S 4.3]{Marshall} for proofs and a more extensive treatment.

We begin with the formal definition of an asymptotically conical Riemannian manifold.

\begin{definition}\label{def:AC}
Let $(B^n,g)$ be a complete Riemannian manifold. We say that $(B,g)$ is asymptotically conical (AC) of rate $\mu<0$ if there exists a compact set $K\subset B$, $R>0$, a Riemannian cone $\left( \tu{C}(\Sigma), dr^2 + r^2 g_\Sigma\right)$ over a smooth compact Riemannian $(n-1)$-manifold $(\Sigma,g_\Sigma)$ and a diffeomorphism $f\co (R,\infty)\times\Sigma\ra B\setminus K$ such that
\[
|\nabla_{g_\tu{C}}^j\left( f^\ast g - g_\tu{C}\right)|_{g_\tu{C}} = O(r^{\mu-j})
\]
for all $j\geq 0$.
\end{definition}
Since we are interested in Ricci-flat AC manifolds, in view of the Cheeger--Gromoll Splitting Theorem we will always assume that $\Sigma$ is connected. In this case $B$ has only one end.

The vector field $r\partial_r$ generates an $\R^+$--action on the Riemannian cone $(\tu{C}(\Sigma), dr^2 +r^2 g_\Sigma)$ by dilations $r\mapsto \lambda r$. Let $E_\infty\ra\tu{C}(\Sigma)$ be a vector bundle equipped with a lift of this $\R^+$--action.
Since dilations are diffeomorphisms of $\tu{C}(\Sigma)$, they act naturally on the frame bundle of $\tu{C}(\Sigma)$ and when $E_\infty$ is a vector bundle associated with the frame bundle via a representation of $\tu{GL}(n,\R)$ we always assume the lift of the $\R^+$--action to $E_\infty$ is the induced one. Moreover, in this case the restriction of the $\tu{GL}(n,\R)$--action to diagonal matrices induces a scaling action on $E_\infty$: if we assume for simplicity that $E_\infty$ is associated with an irreducible representation of $\tu{GL}(n,\R)$ then $\lambda\in \R^+$ acts via $\lambda^w \,\tu{id}_{E_\infty}$, where $w$ is the conformal weight of $E_\infty$, \cf for example \cite[\S I.1]{Gauduchon}. When $E_\infty$ is not associated with the frame bundle we set $w=0$. 
We say that a section $s$ of $E_\infty$ is \emph{$0$--homogeneous} if $r^w s$ is dilation invariant. Similarly, a triple $(E_\infty,h_\infty,\nabla_\infty)$ of a bundle, bundle metric and metric connection on $\tu{C}(\Sigma)$ is said to be \emph{$0$--homogeneous} if $E_\infty$ is equipped with a lift of dilations as above, $h_\infty$ is a $0$--homogeneous section of $\tu{Sym}^2(E_\infty^\ast)$ and $r\nabla_\infty\co \Gamma (E_\infty)\ra \Gamma (E_\infty\otimes T^\ast\tu{C}(\Sigma))$ preserves $0$--homogeneous sections.    

\begin{definition}\label{def:AC:admissible:connection}
Let $(B,g)$ be an AC manifold asymptotic to the cone $\tu{C}(\Sigma)$ with rate $\mu<0$. Let $(E,h,\nabla)\ra B$ be a bundle $E$ together with a bundle metric $h$ and a metric connection $\nabla$. We say that $(E,h,\nabla)$ is admissible if, under the identification $f\co (R,\infty)\times\Sigma\ra B\setminus K$ of Definition \ref{def:AC} there exists a bundle isomorphism $\Phi\co f^\ast E\ra E_\infty$ such that $\Phi^\ast h =h_\infty + h'$ and $\Phi^\ast\nabla = \nabla_\infty + a$, where the triple $(E_\infty, h_\infty,\nabla_\infty)$ on $\tu{C}(\Sigma)$ is $0$--homogeneous and $(h',a)$ satisfy
\[
|\nabla_\infty ^j h'|_{g_\tu{C}\otimes h_\infty}=O(r^{\mu-j}), \qquad |\nabla_\infty ^j a|_{g_\tu{C}\otimes h_\infty}=O(r^{\mu-1-j}).
\]
\end{definition}

We will mostly be interested in (sub)bundles of $\bigotimes^rTB\otimes\bigotimes^s T^\ast B$. By Definition \ref{def:AC}, any such bundle together with the metric induced by $g$ and the connection induced by the Levi--Civita connection of $g$ is admissible.

\begin{definition}\label{def:AC:weighted:spaces}
Let $(E,h,\nabla)$ be an admissible bundle. For all $p\geq 1, k\in\N_0, \alpha\in (0,1)$ and $\nu\in\R$ we define the weighted Sobolev space $L^p_{k,\nu}$ and the weighted H\"older space $C^{k,\alpha}_\nu$ of sections of $E$ as the completion of $C^\infty _c(B;E)$ with respect to the norms
\[
\| u\|_{L^p_{k,\nu}} = \left( \sum_{j=0}^k{\| r^{-\frac{n}{p}-\nu+j}\nabla^j u\|^p_{L^p}}\right)^{\frac{1}{p}}, \qquad \| u\|_{C^{k,\alpha}_\nu} = \sum_{j=0}^k{\|r^{-\nu+j}\nabla^j u\|_{C^0} + [r^{-\nu+k}\nabla^k u]_\alpha}.
\]
Here $[r^{-\nu+k}\nabla^k u]_\alpha$ is the H\"older seminorm defined using parallel transport of $\nabla$ to identify  fibres of $E$ along minimising geodesics in a small neighbourhood of each point in $B$. By dropping the H\"older seminorm $[r^{-\nu+k}\nabla^k u]_\alpha$ in the definition of the $C^{k,\alpha}_\nu$--norm, we obtain the definition of the space of sections of $E$ of class $C^k_\nu$. Finally, set $C^\infty_\nu = \bigcap_{k\geq 0}{C^k_\nu}$.
\end{definition}

\begin{theorem}\label{thm:Weighted:Embedding}
Let $B$ be an $n$-dimensional AC manifold.
\begin{enumerate}
\item If $k\geq h\geq 0$, $k-\frac{n}{p}\geq h-\frac{n}{q}$, $p\leq q$ and $\nu\leq \nu'$ there is a continuous embedding $L^p_{k,\nu}\subset L^q_{h,\nu'}$. Moreover, if $k>h$, $k-\frac{n}{p}> h-\frac{n}{q}$ and $\nu<\nu'$ then the embedding is compact.
\item If $k\geq h\geq 0$, $k-\frac{n}{p}\geq h-\frac{n}{q}$, $p> q$ and $\nu< \nu'$ there is a continuous embedding $L^p_{k,\nu}\subset L^q_{h,\nu'}$. Moreover, if $k>h$ and $k-\frac{n}{p}> h-\frac{n}{q}$ then the embedding is compact.
\item If $\nu<\nu'$ and $k-\frac{n}{p}\geq h+\alpha$ then there are continuous embeddings $L^p_{k,\nu} \subset C^{h,\alpha}_\nu \subset L^q_{h,\nu'}$ for any $q$.
\item If $\nu\leq \nu'$ and $k+\alpha \geq h+\beta$ then there are continuous embeddings $C^{k+1}_\nu\subset C^{k,\alpha}_\nu\subset C^{h,\beta}_{\nu'}\subset C^h_{\nu'}$. Moreover, if $\nu<\nu'$ the embedding $C^{k,\alpha}_\nu\subset C^h_{\nu'}$ is compact.
\item If $\nu_1 + \nu_2 \leq \nu$ then the product $C^{k,\alpha}_{\nu_1}\times C^{k,\alpha}_{\nu_2}\ra C^{k,\alpha}_\nu$ is continuous.
\end{enumerate}
\end{theorem}


\begin{definition}\label{def:Admissible:Operators}
Let $P\co \Gamma(E)\ra \Gamma (F)$ be an elliptic operator of order $k$ between sections of admissible vector bundles over an AC manifold $B$. Let $f\co (R,\infty)\times\Sigma\ra B\setminus K$ be the identification of Definition \ref{def:AC}. Let $P_\infty\co \Gamma (f^\ast E)\ra \Gamma (f^\ast F)$ be an elliptic operator such that $r^k P_\infty$ preserves $0$--homogeneous sections. Assume that there exists $\mu<0$ such that for every $l\geq 0$
\[
|\nabla _\infty^l \left( f^\ast (Pu) - P_\infty f^\ast u\right) |_{h_\infty}=O(r^{-k+\mu-l}) 
\]
for every smooth section $u$ of $E$ on $B\setminus K$. Then we say that $P$ is an admissible operator asymptotic to $P_\infty$.  
\end{definition}


By Definition \ref{def:AC:admissible:connection}, if $P\co \Gamma (E)\ra \Gamma (F)$ is an elliptic operator of order $k$ between admissible vector bundles defined as the composition of $\nabla ^k\co \Gamma (E)\ra \Gamma \left( \bigotimes ^k T^\ast B\otimes E\right)$ with a constant coefficient bundle map $\bigotimes ^k T^\ast B\otimes E\ra F$, then $P$ is admissible. In particular, the Dirac operator, the Laplacian and $d+d^\ast$ acting on spinors and differential forms on an AC manifold are admissible operators.

\begin{lemma}\label{lem:integration:parts}
Let $P\co \Gamma (E)\ra \Gamma (F)$ be an admissible operator of order $1$ and let $P^\ast$ be its formal adjoint. Then for every $u\in L^2_{1,\nu}$ and $v\in L^2_{1,\nu'}$ with $\nu+\nu'\leq -n+1$ we have
\[
\langle Pu,v\rangle_{L^2} = \langle u, P^\ast v\rangle_{L^2}.
\]
\end{lemma}

Many local estimates on domains in $\R^n$ extend to weighted estimates on AC manifolds. One of the basic techniques to obtain these estimates is the following scaling argument. On a fixed compact set $K\subset B$ estimates are proved as on compact manifolds. Identify instead $M\setminus K$ with an exterior region $\{ r \geq R\}$ in $\tu{C} (\Sigma)$. For $R$ (and therefore $K$) large enough, up to small errors we can work with the model operator $P_\infty$. We apply the scaling technique of \cite[Theorem 1.2]{Bartnik}. Decompose the region $\{ r\geq R \}$ in $\tu{C}(\Sigma)$ into the union of annuli $\{ 2^k R \leq r \leq 2^{k+1}R\}$. Up to a factor of $(2^k R)^{-\nu}$, on each annulus the weighted Sobolev/H\"older norms are equivalent (with constants independent of $R$ and $k$) to the standard Sobolev/H\"older norms on the fixed annulus $\{ 1\leq r\leq 2\}$. The required estimates can then be proved by applying standard estimates on these rescaled annuli, rescaling back and summing/taking supremums over $k\in\Z_{\geq 0}$. For instance, consider the standard elliptic estimates in Sobolev spaces, the Schauder estimates and the local estimate
\[
\| u \|_{C^{l+k,\alpha}(B)} \leq C \left( \| Pu\|_{C^{l,\alpha}(2B)} + \| u \|_{L^2 (2B)}\right)
\]
for an elliptic operator $P$. Here $B\subset \R^n$ is a ball and $2B$ is a ball of twice the radius. The scaling argument we have sketched yields the following weighted estimates.

\begin{theorem}\label{thm:Weighted:Regularity}
Let $P\co \Gamma (E)\ra \Gamma (F)$ be an admissible operator of order $k$. Then for every $l\geq 0$, $p\geq 1$, $\alpha\in (0,1)$ and $\nu\in \R$ there exists $C>0$ such that
\[
\begin{gathered}
\| u \|_{L^p_{l+k,\nu+k}} \leq C\left( \| Pu\|_{L^p_{l,\nu}} + \| u \|_{L^p_{0,\nu+k}}\right) ,\qquad \| u \|_{C^{l+k,\alpha}_{\nu+k}} \leq C\left( \| Pu\|_{C^{l,\alpha}_{\nu}} + \| u \|_{C^{0,\alpha}_{\nu+k}}\right),\\
\| u \|_{C^{l+k,\alpha}_{\nu+k}} \leq C \left( \| Pu\|_{C^{l,\alpha}_{\nu}} + \| u \|_{L^2_{\nu+k}}\right)
\end{gathered}
\]
for all $u\in C^\infty_{c}$.
\end{theorem}

In order to proceed further it is necessary to study in more detail the mapping properties of the model operator $P_\infty$. This can be done explicitly by separation of variables. Note that the natural $\R^+$--action on $\tu{C}(\Sigma)$ generated by the vector field $r\partial_r$ allows us to talk of homogeneous functions on $\tu{C}(\Sigma)$.

\begin{definition}\label{def:indicial:roots}
Let $P\co \Gamma (E)\ra \Gamma (F)$ be an admissible operator asymptotic to the model operator $P_\infty \co \Gamma (f^\ast E)\ra \Gamma (f^\ast F)$. We say that a section $u$ of $f^\ast E\ra \tu{C}(\Sigma)$ is homogeneous of rate $\lambda$ if $|u|_{h_\infty}$ is a homogeneous function of rate $\lambda$. We say that $\lambda$ is an \emph{indicial root} for $P_\infty$ if there exists a homogeneous section $u$ of rate $\lambda$ such that  $P_\infty u=0$. Let $\mathcal{D}(P_\infty)$ denote the set of indicial roots for $P_\infty$. For each $\lambda\in \mathcal{D}(P_\infty)$ 
let $d(\lambda)$ denote the dimension of the space of $u\in \ker P_\infty$ of the form $u = \sum_{j=0}^m{u_j (\log{r})^j}$ with $u_0,\dots, u_m$ homogeneous sections of $E_\infty$ of rate $\lambda$.
\end{definition}

In order to understand kernel and cokernel of admissible operators acting between weighted H\"older spaces, the following regularity result is necessary.

\begin{prop}\label{prop:Weighted:Regularity}
Let $P\co \Gamma (E)\ra \Gamma (F)$ be an admissible operator of order $k$. Fix $\alpha\in (0,1)$, $\nu\in \R$ and choose $\nu'>\nu$ so that $[\nu,\nu']$ does not contain any indicial root for $P$. Then there exists $C>0$ such that
\[
\| u \|_{C^{k,\alpha}_{\nu}} \leq C\left( \| Pu\|_{C^{0,\alpha}_{\nu-k}} + \| u \|_{L^2_{\nu'}}\right).
\]
\proof
For fixed compact sets $K'\subset K$ we have the standard estimate $\| u \|_{C^{k,\alpha}} \leq C \left( \| Pu\|_{C^{0,\alpha}} + \| u \|_{L^2}\right)$. One can then use these estimate on a fixed compact set $K_0\subset B$ and over annuli $\{ R\leq r \leq 2R\}$ to show that
\[
\| u \|_{C^{k,\alpha}_{\nu'}} \leq C \left( \| Pu\|_{C^{0,\alpha}_{\nu'-k}} + \| u \|_{L^2_{\nu'}}\right).
\]
Finally, since there are no indicial roots of $P$ in the interval $[\nu,\nu']$, by solving the boundary value problem equation $Pv=Pu$  on $B\setminus K_0$, $v=u$ on $\partial K_0$ one can show that 
\[
\| u \|_{C^{k,\alpha}_{\nu}} \leq C\left( \| Pu\|_{C^{0,\alpha}_{\nu-k}} + \| u \|_{C^{k,\alpha}_{\nu'}}\right).\qedhere
\]
\end{prop}

\begin{corollary}
Let $P\co \Gamma (E)\ra \Gamma (F)$ be an admissible operator of order $k$ and fix $\alpha\in (0,1)$ and $\nu\in\R$. Then for every $f\in C^{0,\alpha}_{\nu-k}$ such that
\[
\langle f , \overline{u}\rangle _{L^2}=0
\]
for all $\overline{u}\in \ker P^\ast \cap C^\infty_{-n-\nu+k}$, there exists $u\in C^{k,\alpha}_\nu$ such that $Pu=f$ and
\[
\| u \|_{C^{k,\alpha}_{\nu}} \leq C\| f\|_{C^{0,\alpha}_{\nu-k}} .
\]
\proof
Fix $\nu'>\nu$ so that $[\nu,\nu']$ does not contain any indicial root for $P$. Note that $f\in L^2_{\nu'-k}$ by Theorem \ref{thm:Weighted:Embedding} (iii) since $\nu'>\nu$ and that $\ker P^\ast \cap C^\infty_{-n-\nu+k} = \ker P^\ast \cap L^2_{-n-\nu'+k}$ since $[\nu,\nu']$ does not contain any indicial root. By Lemma \ref{lem:integration:parts}, $f$ is $L^2$--orthogonal to the cokernel of $P\co L^2_{k,\nu'}\ra L^2_{\nu'-k}$ and therefore we can solve the equation $Pu=f$ with $u\in L^2_{k,\nu'}$. If $u$ is $L^2_{\nu'}$--orthogonal to the kernel of $P$ in $L^2_{\nu'}$ then we also have an estimate
\[
\| u\|_{L^2_{\nu'}} \leq C \|f\|_{L^2_{\nu'-k}}\leq C' \| f\|_{C^{0,\alpha}_{\nu-k}}
\]
for constants $C,C'>0$ independent of $u$ and $f$. Now apply Proposition \ref{prop:Weighted:Regularity}.
\endproof
\end{corollary}

The following theorem contains the main statement about admissible operators between weighted H\"older spaces, their Fredholm property and index.

\begin{theorem}\label{thm:Fredholm}
Let $P\co \Gamma (E)\ra \Gamma (F)$ be an admissible operator of order $k$ and fix $l\geq 0$, $\alpha\in (0,1)$ and $\nu,\nu'\in\R$ with $\nu<\nu'$.
\begin{enumerate}
\item If $\nu \in \R \setminus \mathcal{D}(P_\infty)$ then there exists a compact set $K\subset B$ and a constant $C>0$ such that
\[
\| u \| _{C^{l+k,\alpha}_{\nu}} \leq C \left( \| Pu \| _{C^{l,\alpha}_{\nu-k}} + \| u \| _{L^2 (K)}\right).
\]
In particular, $P\co C^{l+k,\alpha}_{\nu}\ra C^{l,\alpha}_{\nu-k}$ is a Fredholm operator.
\item Assume that $\nu,\nu'\notin\mathcal{D}(P_\infty)$ and denote by $i(\nu)$ and $i(\nu')$ the indexes of $P\co C^{k,\alpha}_{\nu}\ra C^{0,\alpha}_{\nu-k}$ and $P\co C^{k,\alpha}_{\nu'}\ra C^{0,\alpha}_{\nu'-k}$ respectively. Then 
\[
i(\nu')-i(\nu)=N(\nu,\nu'),
\]
where $N(\nu, \nu') = \sum_{\lambda\in \mathcal{D}(P_\infty)\cap (\nu,\nu')}{d(\lambda)}$.
\end{enumerate}
\end{theorem}

Finally, the following result about the asymptotic decay of solutions to $Du=f$ is used in the proof of the index jump formula in Theorem \ref{thm:Fredholm} and is as useful as the Theorem itself. 

\begin{prop}\label{prop:Decay:Solutions}
Let $P\co \Gamma (E)\ra \Gamma (F)$ be an admissible operator of order $k$ and fix $l\geq 0$, $\alpha\in (0,1)$ and $\nu,\nu'\in\R$ with $\nu<\nu'$ and $\nu,\nu'\notin\mathcal{D}(P_\infty)$.

Set $N=N(\nu,\nu')$. Let $u_1,\dots, u_N$ be a basis of the space of $u\in \ker P_\infty$ such that there exist $\lambda\in (\nu,\nu')$ and homeogeneous sections $\overline{u}_0,\dots, \overline{u}_m$ of $E_\infty$ of rate $\lambda$ such that $u = \sum_{j=0}^m{\overline{u}_j (\log{r})^j}$.

Then there exists a compact set $K\subset B$ such that for every $f\in C^{0,\alpha}_{\nu-k}$ with $f=Du'$ for some $u'\in C^{k,\alpha}_{\nu'}$ there exist $a=(a_1,\dots,a_N)\in\R^N$ and $u\in C^{k,\alpha}_{\nu}(B\setminus K)$ such that $u'|_{B\setminus K}=u+\sum_{i=1}^N{a_i\,u_i}$. Moreover, there exists a constant $C>0$ independent of $f,u,u',a$ such that
\[
\| u \|_{C^{k,\alpha}_{\nu}(B\setminus K)} + \|a\| \leq C\left( \| f\| _{C^{0,\alpha}_{\nu-k}} + \| u' \|_{C^{k,\alpha}_{\nu'}}\right).
\]
\end{prop}

\bibliographystyle{amsinitial}
\bibliography{Def_ALC}

\providecommand{\bysame}{\leavevmode\hbox to3em{\hrulefill}\thinspace}
\providecommand{\MR}{\relax\ifhmode\unskip\space\fi MR }
\providecommand{\MRhref}[2]{%
  \href{http://www.ams.org/mathscinet-getitem?mr=#1}{#2}
}
\providecommand{\href}[2]{#2}
\begin{thebibliography}{10}

\bibitem{Acharya:sym}
B.~S. {Acharya}, \emph{{On Realising N=1 Super Yang-Mills in M theory}}, 2000,
  hep-th/0011089.

\bibitem{Ammann:Bar}
B.~Ammann and C.~B{\"a}r, \emph{The {D}irac operator on nilmanifolds and
  collapsing circle bundles}, Ann. Global Anal. Geom. \textbf{16} (1998),
  no.~3, 221--253.

\bibitem{Apostolov:Salamon}
V.~Apostolov and S.~Salamon, \emph{K\"ahler reduction of metrics with holonomy
  {$G_2$}}, Comm. Math. Phys. \textbf{246} (2004), no.~1, 43--61.

\bibitem{Atiyah:Maldacena:Vafa}
M.~Atiyah, J.~Maldacena, and C.~Vafa, \emph{An {M}-theory flop as a large {$N$}
  duality}, J. Math. Phys. \textbf{42} (2001), no.~7, 3209--3220.

\bibitem{Bartnik}
R.~Bartnik, \emph{The mass of an asymptotically flat manifold}, Comm. Pure
  Appl. Math. \textbf{39} (1986), no.~5, 661--693.

\bibitem{Friedrich:al}
H.~Baum, T.~Friedrich, R.~Grunewald, and I.~Kath, \emph{Twistors and {K}illing
  spinors on {R}iemannian manifolds}, Teubner-Texte zur Mathematik, vol. 124,
  B. G. Teubner Verlagsgesellschaft mbH, Stuttgart, 1991.

\bibitem{Bogoyavlenskaya}
O.~A. Bogoyavlenskaya, \emph{On a new family of complete {R}iemannian metrics
  on {$S^3\times\Bbb R^4$} with holonomy group {$G_2$}}, Sibirsk. Mat. Zh.
  \textbf{54} (2013), no.~3, 551--562.

\bibitem{Bourguignon:Gauduchon}
J.-P. Bourguignon and P.~Gauduchon, \emph{Spineurs, op\'erateurs de {D}irac et
  variations de m\'etriques}, Comm. Math. Phys. \textbf{144} (1992), no.~3,
  581--599.

\bibitem{Boyer:Galicki:JDG}
C.~P. Boyer and K.~Galicki, \emph{New {E}instein metrics in dimension five}, J.
  Differential Geom. \textbf{57} (2001), no.~3, 443--463.

\bibitem{Boyer:Galicki}
\bysame, \emph{Sasakian geometry}, Oxford Mathematical Monographs, Oxford
  University Press, Oxford, 2008.

\bibitem{Boyer:Galicki:Kollar}
C.~P. Boyer, K.~Galicki, and J.~Koll\'ar, \emph{Einstein metrics on spheres},
  Ann. of Math. (2) \textbf{162} (2005), no.~1, 557--580.

\bibitem{Brandhuber}
A.~Brandhuber, \emph{{$G_2$} holonomy spaces from invariant three-forms},
  Nuclear Phys. B \textbf{629} (2002), no.~1-3, 393--416.

\bibitem{BGGG}
A.~Brandhuber, J.~Gomis, S.~S. Gubser, and S.~Gukov, \emph{Gauge theory at
  large {$N$} and new {$G_2$} holonomy metrics}, Nuclear Phys. B \textbf{611}
  (2001), no.~1-3, 179--204.

\bibitem{Brieskorn}
E.~Brieskorn, \emph{Die {A}ufl\"osung der rationalen {S}ingularit\"aten
  holomorpher {A}bbildungen}, Math. Ann. \textbf{178} (1968), 255--270.

\bibitem{Bryant:1987}
R.~L. Bryant, \emph{Metrics with exceptional holonomy}, Ann. of Math. (2)
  \textbf{126} (1987), no.~3, 525--576.

\bibitem{Bryant:Salamon}
R.~L. Bryant and S.~M. Salamon, \emph{On the construction of some complete
  metrics with exceptional holonomy}, Duke Math. J. \textbf{58} (1989), no.~3,
  829--850.

\bibitem{Caibar}
M.~Caib\u{a}r, \emph{Minimal models of canonical 3-fold singularities and their
  {B}etti numbers}, Int. Math. Res. Not. (2005), no.~26, 1563--1581.

\bibitem{Candelas:delaOssa}
P.~Candelas and X.~C. de~la Ossa, \emph{Comments on conifolds}, Nuclear Phys. B
  \textbf{342} (1990), no.~1, 246--268.

\bibitem{Cheeger:unpublished}
J.~Cheeger, \emph{On the spectral geometry of spaces with cone-like
  singularities}, 1978.

\bibitem{Cheeger:PNAS}
\bysame, \emph{On the spectral geometry of spaces with cone-like
  singularities}, Proc. Nat. Acad. Sci. U.S.A. \textbf{76} (1979), no.~5,
  2103--2106.

\bibitem{Cheeger:cones}
\bysame, \emph{Spectral geometry of singular {R}iemannian spaces}, J.
  Differential Geom. \textbf{18} (1983), no.~4, 575--657 (1984).

\bibitem{CDS}
X.~Chen, S.~Donaldson, and S.~Sun, \emph{K\"{a}hler-{E}instein metrics on
  {F}ano manifolds. {I}: {A}pproximation of metrics with cone singularities},
  J. Amer. Math. Soc. \textbf{28} (2015), no.~1, 183--197.

\bibitem{Chiossi:Salamon}
S.~Chiossi and S.~Salamon, \emph{The intrinsic torsion of {$\rm SU(3)$} and
  {$G_2$} structures}, Differential geometry, {V}alencia, 2001, World Sci.
  Publ., River Edge, NJ, 2002, pp.~115--133.

\bibitem{Collins:Szekelyhidi:JDG}
T.~C. Collins and G.~Sz\'{e}kelyhidi, \emph{K-semistability for irregular
  {S}asakian manifolds}, J. Differential Geom. \textbf{109} (2018), no.~1,
  81--109.

\bibitem{Collins:Szekelyhidi:GT}
\bysame, \emph{Sasaki-{E}instein metrics and {K}-stability}, Geom. Topol.
  \textbf{23} (2019), no.~3, 1339--1413.

\bibitem{Conlon:Hein:I}
R.~J. Conlon and H.-J. Hein, \emph{Asymptotically conical {C}alabi-{Y}au
  manifolds, {I}}, Duke Math. J. \textbf{162} (2013), no.~15, 2855--2902.

\bibitem{CGLP:Almost:Special:Holonomy}
M.~Cveti\v{c}, G.~Gibbons, H.~L\"u, and C.~Pope, \emph{Almost special holonomy
  in type {IIA} and {M}-theory}, Nuclear Physics B \textbf{638} (2002),
  no.~1-2, 186--206.

\bibitem{CGLP:C7:tilde}
\bysame, \emph{A {$G_2$} unification of the deformed and resolved conifolds},
  Physics Letters B \textbf{534} (2002), no.~1-4, 172--180.

\bibitem{CGLP:M:Conifolds}
\bysame, \emph{M-theory conifolds}, Phys. Rev. Lett. \textbf{88} (2002),
  no.~12, no. 121602, 4.

\bibitem{CGLP:ALC:Spin7}
\bysame, \emph{New complete noncompact {S}pin(7) manifolds}, Nuclear Phys. B
  \textbf{620} (2002), no.~1-2, 29--54.

\bibitem{Dai}
X.~Dai, \emph{Adiabatic limits, nonmultiplicativity of signature, and {L}eray
  spectral sequence}, J. Amer. Math. Soc. \textbf{4} (1991), no.~2, 265--321.

\bibitem{Forman}
R.~Forman, \emph{Spectral sequences and adiabatic limits}, Comm. Math. Phys.
  \textbf{168} (1995), no.~1, 57--116.

\bibitem{def:nK}
L.~Foscolo, \emph{Deformation theory of nearly {K}\"{a}hler manifolds}, J.
  Lond. Math. Soc. \textbf{95} (2017), no.~2, 586--612.

\bibitem{Foscolo:ALC:Spin7}
\bysame, \emph{Complete non-compact {$\tu{Spin}(7)$}--manifolds from self-dual
  {E}instein {$4$}--orbifolds}, 2019, arXiv:1901.04074. To appear in Geometry
  and Topology.

\bibitem{FHN:Coho1:ALC}
L.~Foscolo, M.~Haskins, and J.~Nordstr\"om, \emph{Infinitely many new families
  of complete cohomogeneity one {$\gtwo$}-manifolds: {$\gtwo$} analogues of the
  {T}aub-{NUT} and {E}guchi-{H}anson spaces}, 2018, arXiv:1805.02612. To appear
  in J. Eur. Math. Soc. (JEMS).

\bibitem{Friedman}
R.~Friedman, \emph{Simultaneous resolution of threefold double points}, Math.
  Ann. \textbf{274} (1986), no.~4, 671--689.

\bibitem{Friedrich:Kath}
T.~Friedrich and I.~Kath, \emph{Einstein manifolds of dimension five with small
  first eigenvalue of the {D}irac operator}, J. Differential Geom. \textbf{29}
  (1989), no.~2, 263--279.

\bibitem{Futaki:Ono:Wang}
A.~Futaki, H.~Ono, and G.~Wang, \emph{Transverse {K}\"ahler geometry of
  {S}asaki manifolds and toric {S}asaki-{E}instein manifolds}, J. Differential
  Geom. \textbf{83} (2009), no.~3, 585--635.

\bibitem{Gauduchon}
P.~Gauduchon, \emph{Structures de {W}eyl et th\'{e}or\`emes d'annulation sur
  une vari\'{e}t\'{e} conforme autoduale}, Ann. Scuola Norm. Sup. Pisa Cl. Sci.
  (4) \textbf{18} (1991), no.~4, 563--629.

\bibitem{GMSW}
J.~P. Gauntlett, D.~Martelli, J.~Sparks, and D.~Waldram,
  \emph{Sasaki-{E}instein metrics on {$S^2\times S^3$}}, Adv. Theor. Math.
  Phys. \textbf{8} (2004), no.~4, 711--734.

\bibitem{GMSY}
J.~P. Gauntlett, D.~Martelli, J.~Sparks, and S.-T. Yau, \emph{Obstructions to
  the existence of {S}asaki-{E}instein metrics}, Comm. Math. Phys. \textbf{273}
  (2007), no.~3, 803--827.

\bibitem{Goto:Crepant}
R.~Goto, \emph{Calabi-{Y}au structures and {E}instein-{S}asakian structures on
  crepant resolutions of isolated singularities}, J. Math. Soc. Japan
  \textbf{64} (2012), no.~3, 1005--1052.

\bibitem{Hamm}
H.~Hamm, \emph{Lokale topologische {E}igenschaften komplexer {R}\"aume}, Math.
  Ann. \textbf{191} (1971), 235--252.

\bibitem{ACyl:CY}
M.~Haskins, H.-J. Hein, and J.~Nordstr\"om, \emph{Asymptotically cylindrical
  {C}alabi-{Y}au manifolds}, J. Differential Geom. \textbf{101} (2015), no.~2,
  213--265.

\bibitem{HHM}
T.~Hausel, E.~Hunsicker, and R.~Mazzeo, \emph{Hodge cohomology of gravitational
  instantons}, Duke Math. J. \textbf{122} (2004), no.~3, 485--548.

\bibitem{Hein:Sun}
H.-J. Hein and S.~Sun, \emph{Calabi-{Y}au manifolds with isolated conical
  singularities}, Publ. Math. Inst. Hautes \'{E}tudes Sci. \textbf{126} (2017),
  73--130.

\bibitem{Hitchin:3forms}
N.~Hitchin, \emph{The geometry of three-forms in six dimensions}, J.
  Differential Geom. \textbf{55} (2000), no.~3, 547--576.

\bibitem{Hitchin:Stable:forms}
\bysame, \emph{Stable forms and special metrics}, Global differential geometry:
  the mathematical legacy of {A}lfred {G}ray ({B}ilbao, 2000), Contemp. Math.,
  vol. 288, Amer. Math. Soc., Providence, RI, 2001, pp.~70--89.

\bibitem{Hori:A7}
K.~Hori, K.~Hosomichi, D.~C. Page, R.~Rabad\'an, and J.~Walcher,
  \emph{Non-perturbative orientifold transitions at the conifold}, J. High
  Energy Phys. (2005), no.~10, 026, 61.

\bibitem{Joyce:Book}
D.~Joyce, \emph{Compact manifolds with special holonomy}, Oxford Mathematical
  Monographs, Oxford University Press, Oxford, 2000.

\bibitem{Joyce:ALE:CY}
\bysame, \emph{Asymptotically locally {E}uclidean metrics with holonomy {${\rm
  SU}(m)$}}, Ann. Global Anal. Geom. \textbf{19} (2001), no.~1, 55--73.

\bibitem{Karigiannis}
S.~Karigiannis, \emph{Desingularization of {$G_2$} manifolds with isolated
  conical singularities}, Geom. Topol. \textbf{13} (2009), no.~3, 1583--1655.

\bibitem{Karigiannis:Lotay}
S.~Karigiannis and J.~Lotay, \emph{Deformation theory of {$G_2$} conifolds},
  2014, arXiv:1212.6457. To appear in Communications in Analysis and Geometry.

\bibitem{Kaste:et:al:I}
P.~Kaste, R.~Minasian, M.~Petrini, and A.~Tomasiello, \emph{Kaluza-{K}lein
  bundles and manifolds of exceptional holonomy}, Journal of High Energy
  Physics \textbf{2002} (2002), no.~9, 33 pages.

\bibitem{Kaste:et:al:II}
\bysame, \emph{Nontrivial {RR} two-form field strength and
  {$SU(3)$}-structure}, Fortschritte der Physik \textbf{51} (2003), no.~7-8,
  764--768.

\bibitem{Katz}
S.~Katz, \emph{Small resolutions of {G}orenstein threefold singularities},
  Algebraic geometry: {S}undance 1988, Contemp. Math., vol. 116, Amer. Math.
  Soc., Providence, RI, 1991, pp.~61--70.

\bibitem{Kodaira:Nirenberg:Spencer}
K.~Kodaira, L.~Nirenberg, and D.~C. Spencer, \emph{On the existence of
  deformations of complex analytic structures}, Ann. of Math. (2) \textbf{68}
  (1958), 450--459.

\bibitem{Kodaira}
K.~Kodaira, \emph{Complex manifolds and deformation of complex structures},
  Grundlehren der Mathematischen Wissenschaften, vol. 283, Springer-Verlag, New
  York, 1986.

\bibitem{Li:Sun}
C.~Li and S.~Sun, \emph{Conical {K}\"ahler-{E}instein metrics revisited}, Comm.
  Math. Phys. \textbf{331} (2014), no.~3, 927--973.

\bibitem{Lichnerowicz}
A.~Lichnerowicz, \emph{G\'{e}om\'{e}trie des groupes de transformations},
  Travaux et Recherches Math\'{e}matiques, III. Dunod, Paris, 1958.

\bibitem{Lockhart}
R.~Lockhart, \emph{Fredholm, {H}odge and {L}iouville theorems on noncompact
  manifolds}, Trans. Amer. Math. Soc. \textbf{301} (1987), no.~1, 1--35.

\bibitem{Lockhart:McOwen}
R.~B. Lockhart and R.~C. McOwen, \emph{Elliptic differential operators on
  noncompact manifolds}, Ann. Scuola Norm. Sup. Pisa Cl. Sci. (4) \textbf{12}
  (1985), no.~3, 409--447.

\bibitem{Marshall}
S.~P. Marshall, \emph{Deformations of special lagrangian submanifolds}, Ph.D.
  thesis, University of Oxford, 2002, available at
  https://people.maths.ox.ac.uk/joyce/theses/MarshallDPhil.pdf.

\bibitem{Mazzeo:Melrose}
R.~Mazzeo and R.~B. Melrose, \emph{The adiabatic limit, {H}odge cohomology and
  {L}eray's spectral sequence for a fibration}, J. Differential Geom.
  \textbf{31} (1990), no.~1, 185--213.

\bibitem{Mazzeo:Melrose:fibred}
\bysame, \emph{Pseudodifferential operators on manifolds with fibred
  boundaries}, Asian J. Math. \textbf{2} (1998), no.~4, 833--866, Mikio Sato: a
  great Japanese mathematician of the twentieth century.

\bibitem{Melrose}
R.~B. Melrose, \emph{The {A}tiyah-{P}atodi-{S}inger index theorem}, Research
  Notes in Mathematics, vol.~4, A K Peters, Ltd., Wellesley, MA, 1993.

\bibitem{Milnor}
J.~Milnor, \emph{Singular points of complex hypersurfaces}, Annals of
  Mathematics Studies, No. 61, Princeton University Press, Princeton, N.J.;
  University of Tokyo Press, Tokyo, 1968.

\bibitem{Moroianu:Nagy:Semmelmann}
A.~Moroianu, P.-A. Nagy, and U.~Semmelmann, \emph{Deformations of nearly
  {K}\"ahler structures}, Pacific J. Math. \textbf{235} (2008), no.~1, 57--72.

\bibitem{Moroianu:Semmelmann:generalised:Killing}
A.~Moroianu and U.~Semmelmann, \emph{Generalized {K}illing spinors on
  {E}instein manifolds}, Internat. J. Math. \textbf{25} (2014), no.~4, 1450033,
  19.

\bibitem{Nordstrom:Thesis}
J.~Nordstr\"om, \emph{Deformations and gluing of asymptotically cylindrical
  manifolds with exceptional holonomy}, Ph.D. thesis, University of Cambridge,
  2008, available at
  http://people.bath.ac.uk/jlpn20/thesis\_final\_twoside.pdf.

\bibitem{Obata}
M.~Obata, \emph{Certain conditions for a {R}iemannian manifold to be isometric
  with a sphere}, J. Math. Soc. Japan \textbf{14} (1962), 333--340.

\bibitem{Oliveira:Thesis}
G.~Oliveira, \emph{Monopoles in higher dimensions}, Ph.D. thesis, Imperial
  College London, 2014, available at http://hdl.handle.net/10044/1/23570.

\bibitem{Zayas:et:al}
L.~A. Pando~Zayas and A.~A. Tseytlin, \emph{3-branes on spaces with
  ${R}\ifmmode\times\else\texttimes\fi{}{S}^{2}\ifmmode\times\else\texttimes\fi{}{S}^{3}$
  topology}, Phys. Rev. D \textbf{63} (2001), 086006.

\bibitem{Pinkham}
H.~C. Pinkham, \emph{Factorization of birational maps in dimension {$3$}},
  Singularities, {P}art 2 ({A}rcata, {C}alif., 1981), Proc. Sympos. Pure Math.,
  vol.~40, Amer. Math. Soc., Providence, RI, 1983, pp.~343--371.

\bibitem{Reid:Pagoda}
M.~Reid, \emph{Minimal models of canonical {$3$}-folds}, Algebraic varieties
  and analytic varieties ({T}okyo, 1981), Adv. Stud. Pure Math., vol.~1,
  North-Holland, Amsterdam, 1983, pp.~131--180.

\bibitem{Sparks:SE}
J.~Sparks, \emph{Sasaki-{E}instein manifolds}, Surveys in differential
  geometry. {V}olume {XVI}. {G}eometry of special holonomy and related topics,
  Int. Press, Somerville, MA, 2011, pp.~265--324.

\bibitem{Tian}
G.~Tian, \emph{Smoothness of the universal deformation space of compact
  {C}alabi-{Y}au manifolds and its {P}etersson-{W}eil metric}, Mathematical
  aspects of string theory ({S}an {D}iego, {C}alif., 1986), Adv. Ser. Math.
  Phys., vol.~1, World Sci. Publishing, Singapore, 1987, pp.~629--646.

\bibitem{Todorov}
A.~N. Todorov, \emph{The {W}eil-{P}etersson geometry of the moduli space of
  {${\rm SU}(n\geq 3)$} ({C}alabi-{Y}au) manifolds. {I}}, Comm. Math. Phys.
  \textbf{126} (1989), no.~2, 325--346.

\bibitem{vafa:top:strings}
C.~Vafa, \emph{Superstrings and topological strings at large {$N$}}, J. Math.
  Phys. \textbf{42} (2001), no.~7, 2798--2817, Strings, branes, and M-theory.

\bibitem{VanCoevering:Existence}
C.~van Coevering, \emph{Ricci-flat {K}\"ahler metrics on crepant resolutions of
  {K}\"ahler cones}, Math. Ann. \textbf{347} (2010), no.~3, 581--611.

\bibitem{VanCoevering:Examples}
\bysame, \emph{Examples of asymptotically conical {R}icci-flat {K}\"ahler
  manifolds}, Math. Z. \textbf{267} (2011), no.~1-2, 465--496.

\bibitem{Wall}
C.~T.~C. Wall, \emph{Classification problems in differential topology. {V}.
  {O}n certain {$6$}-manifolds}, Invent. Math. 1 (1966), 355-374; corrigendum,
  ibid \textbf{2} (1966), 306.

\bibitem{Wang:Ziller}
M.~Y. Wang and W.~Ziller, \emph{Einstein metrics on principal torus bundles},
  J. Differential Geom. \textbf{31} (1990), no.~1, 215--248.

\end{thebibliography}

\end{document}